\documentclass[pdflatex,sn-mathphys-num]{sn-jnl}

\usepackage{graphicx}%
\usepackage{multirow}%
\usepackage{amsmath,amssymb,amsfonts}%
\usepackage{amsthm}%
\usepackage{tikz-cd}%
\usepackage{mathrsfs}%
\usepackage[title]{appendix}%
\usepackage{xcolor}%
\usepackage{textcomp}%
\usepackage{manyfoot}%
\usepackage{booktabs}%
\usepackage{algorithm}%
\usepackage{algorithmicx}%
\usepackage{algpseudocode}%
\usepackage{listings}%
\usepackage{wrapfig}%

\theoremstyle{thmstyleone}%
\newtheorem{theorem}{Theorem}[section]

\newtheorem{lemma}[theorem]{Lemma}
\newtheorem{corollary}[theorem]{Corollary}

\theoremstyle{thmstyletwo}%

\newtheorem{remark}[theorem]{Remark}

\theoremstyle{thmstylethree}%
\newtheorem{definition}[theorem]{Definition}

\numberwithin{equation}{section}

\renewcommand{\orcidlogo}{%
  \includegraphics[width=10pt]{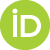}%
}

\begin{document}

\title[Square Metric Spaces]{Square Metric Spaces}


\author*{\fnm{Charles} \sur{Fanning} \orcid{https://orcid.org/0009-0002-8251-2802}}\email{cfannin8@students.kennesaw.edu}

\author{\fnm{Mehmet} \sur{Aktas} \orcid{https://orcid.org/0000-0002-9527-9600}}\email{maktas1@kennesaw.edu}

\affil{\orgdiv{School of Data Science and Analytics}, 
\orgname{Kennesaw State University}, 
\orgaddress{\street{1000 Chastain Rd NW}, 
\city{Kennesaw}, 
\postcode{30144}, 
\state{Georgia}, 
\country{United States}}}


\abstract{
Product decompositions of metric spaces are built from coordinate maps, but these maps are not part of the resulting metric space. We recover this missing coordinate structure through equivalence relations whose classes are candidate coordinate fibers, and the resulting quotient metrics reconstruct the coordinate factors. This framework characterizes exactly when a metric space admits a finite product or power presentation. We prove an equivalence of categories showing that these equivalence-relation data preserve exactly the ordered coordinate information of power presentations. For spaces with suitable \(\ell^\infty\)-prime factorizations, we use prime multiplicities to determine the existence and classification of roots. We also study metric spaces satisfying \(X\cong X\times_\infty X\), where repeated coordinate splitting gives a family of metric quotients indexed by infinite binary sequences. We prove that these binary tree structures exactly characterize metric spaces satisfying \(X\cong X\times_\infty X\). As an application to persistent homology, we show how to recover filtration parameters whose products or powers form a given space of intervals. 
}



\maketitle

\tableofcontents

\section{Introduction}\label{sec1}

Let \((X,d)\) be a metric space. We ask whether the metric \(d\) arises from an isometric identification of \(X\) with a finite product of metric spaces. The coordinate factors need not be mutually isometric. When all coordinate factors are isometric to a single metric space, we ask whether \(X\) is a finite power of a metric space. We study this recognition problem for finite \(\ell^\infty\)-products, finite \(\ell^p\)-products with \(1\leq p<\infty\), and the corresponding \(k\)-fold powers with \(k\geq 2\). In the power case, the common factor is a \(k\)-th root of \(X\). A metric space does not specify its coordinate factors or the corresponding coordinate maps, which we must reconstruct from intrinsic data on \((X,d)\).

Questions about whether products and powers determine their factors ask how much of the coordinate structure remains visible in the product metric. Ulam's problem asks whether an isometry \(A^2\cong B^2\) forces \(A\cong B\) for metric spaces \(A\) and \(B\) \cite{ulam1960collection}. Fournier showed that this implication fails for noncomplete spaces by constructing metric spaces \(A\) and \(B\) with \(A^2\cong B^2\) and \(A\not\cong B\) \cite[p.~622]{fournier1971problem}. The corresponding complete-space question remains open. This uniqueness problem extends from powers to more general metric products, where Moszy{\'n}ska studied the corresponding product question \cite{moszynska1992uniqueness}. In the broader theory of metric products, Tardif developed a factorization theory for Cartesian products based on prefibers, Herburt and Moszy{\'n}ska analyzed families of product metrics, and Herburt studied cancellation for metric products \cite{tardif1992prefibers,herburt1991metric,herburt1994there}. We reverse this perspective by starting with the metric space \((X,d)\) itself and asking how to recover its product and power structures.

The key step in this paper is to replace the unavailable coordinate maps by equivalence relations on \(X\). A coordinate map in a product presentation determines fibers, and equality of coordinate values defines an equivalence relation on \(X\). We use equivalence relations on \(X\) whose classes are candidate coordinate fibers and ask when these classes define coordinates. We identify exactly when the quotient distances induced by \(d\) define metrics on the quotient sets. The equivalence classes then supply the coordinate values absent from the original data, as shown in Figure~\ref{fig:square-metric-space}, and the quotient metric spaces become the reconstructed coordinate factors. The product recognition problem becomes the question of deciding when the resulting quotient-coordinate map is a bijective isometry. For power recognition, we also require that the quotient metric spaces are mutually isometric.

\begin{wrapfigure}[18]{r}{0.48\textwidth}
\vspace{-0.75em}
\centering
\includegraphics[width=0.46\textwidth]{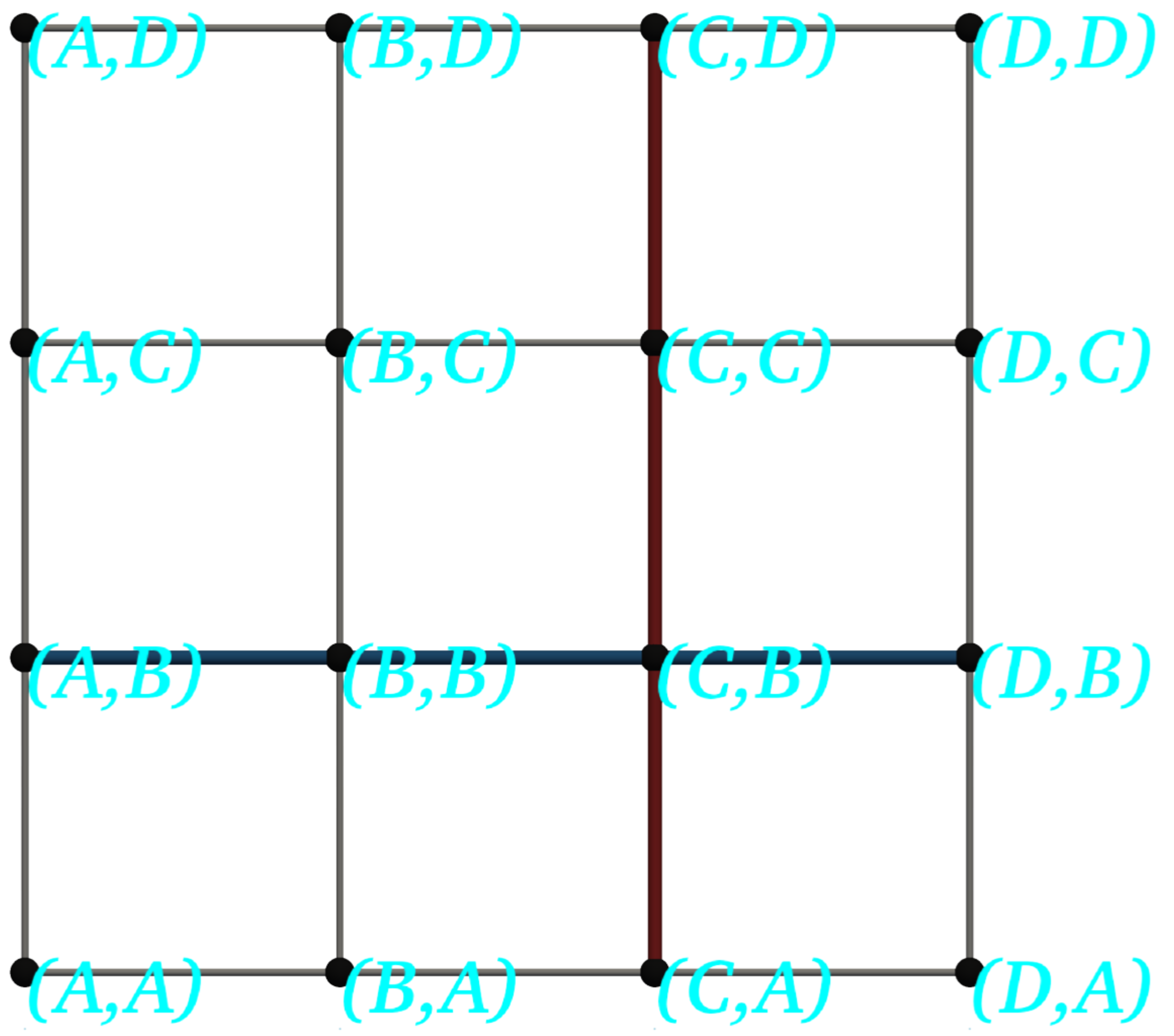}
\caption{
A square metric space \(P^2\) formed from the four-point path metric space \(P=\{A,B,C,D\}\). The horizontal and vertical copies of \(P\) are coordinate fibers.
}
\label{fig:square-metric-space}
\vspace{-0.75em}
\end{wrapfigure}

We first prove necessary and sufficient conditions for equivalence relations on \(X\) to give a product presentation. The quotient-coordinate map gives the product decomposition precisely when it is a bijective isometry.

For \(1\leq p<\infty\), we characterize \(\ell^p\)-product decompositions by three intrinsic conditions on the equivalence relations, the quotient distances determined by these equivalence relations, and the metric \(d\). They require each quotient distance to be a metric, the intersection of one equivalence class from each quotient to contain exactly one point, and the metric \(d\) to equal the \(\ell^p\)-product formula defined using the quotient distances.

\par\WFclear

\begin{theorem}
Let \((X,d)\) be a metric space, let \(k\geq2\), and let \(1\leq p<\infty\). The following are equivalent.
\begin{enumerate}
\item
There are metric spaces \((P_j,d_j)\) for \(1\leq j\leq k\) such that \((X,d)\cong\prod_{j=1}^k(P_j,d_j)\) with the \(\ell^p\)-product metric from Equation~\eqref{eq:kp-product-metric}.

\item
There are equivalence relations \(E_1,\dots,E_k\) on \(X\) such that:
\begin{enumerate}
\item
For each \(1\leq j\leq k\), the quotient distance \(\bar d_{E_j}\) is a metric on \(X/E_j\).

\item
For every choice of classes \(C_j\in X/E_j\), where \(1\leq j\leq k\), the intersection \(\bigcap_{j=1}^k C_j\) consists of exactly one point.

\item
For all \(x,y\in X\),
\[
d(x,y)
=
\left(
\sum_{j=1}^k
\bar d_{E_j}([x]_{E_j},[y]_{E_j})^p
\right)^{1/p}.
\]
\end{enumerate}
\end{enumerate}
\end{theorem}

For \(\ell^\infty\)-products, the criterion uses the same equivalence relations, with the maximum in place of the \(\ell^p\)-norm. For powers, one further condition is necessary and sufficient: the quotient metric spaces \((X/E_j,\bar d_{E_j})\) must all be isometric to a common metric space. Under this condition, as Figure~\ref{fig:cubic-metric-space} shows, the quotient-coordinate map \(x\mapsto([x]_{E_1},\dots,[x]_{E_k})\) is the required bijective isometry onto the \(k\)-fold product.

\begin{figure}[htbp]
\centering
\includegraphics[width=0.72\textwidth]{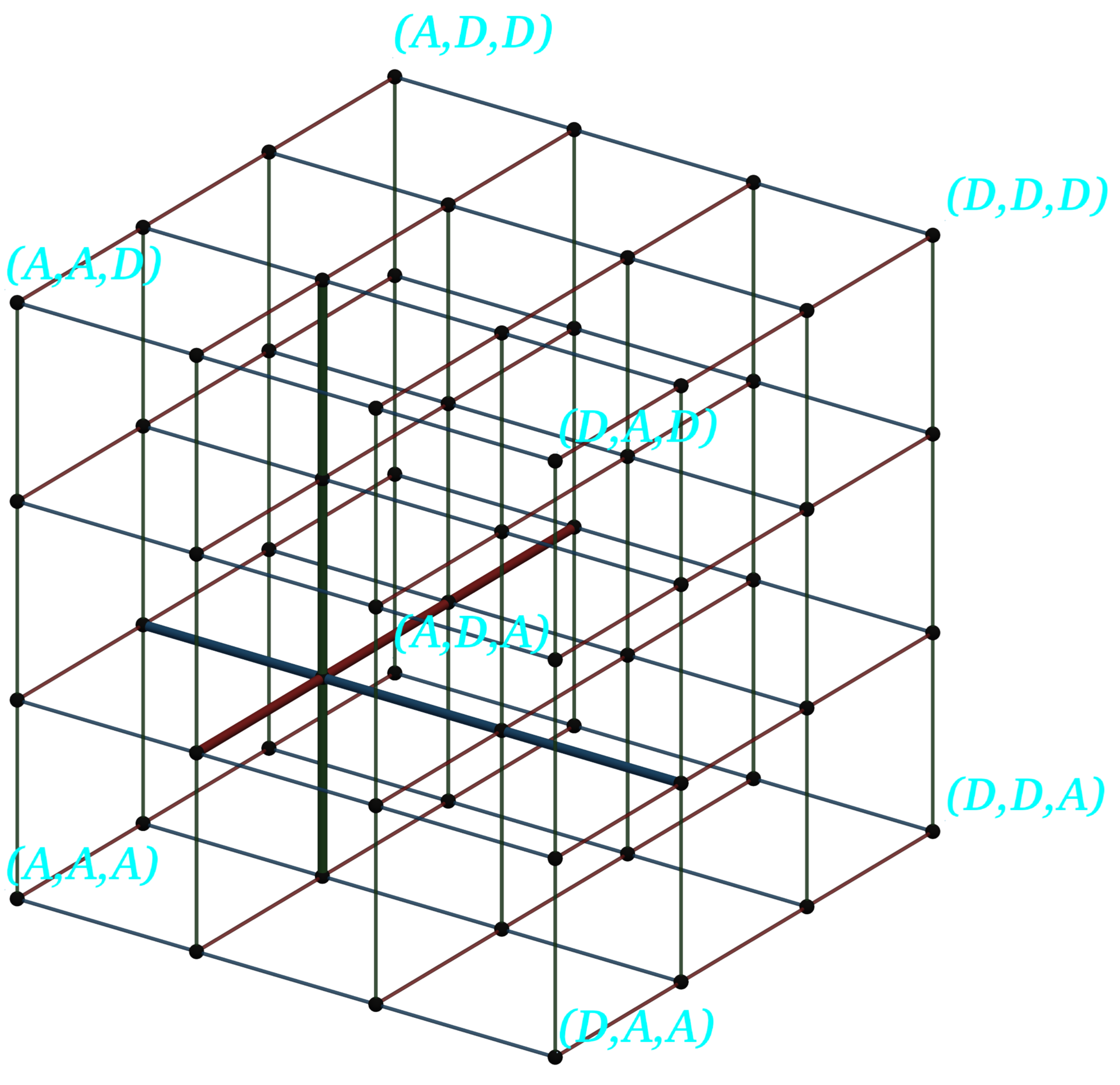}
\caption{
A cubic metric space \(P^3\) obtained from three copies of the same four-point metric space \(P\). A \(k\)-power structure requires the coordinate quotients to recover mutually isometric factors. The higher-dimensional coordinate structure is therefore encoded by compatible equivalence relations whose quotient-coordinate map reconstructs the product.
}
\label{fig:cubic-metric-space}
\end{figure}

The power case has two ways to record the same ordered coordinate information. In the presentation category \(\mathbf{PowPres}_{k,p}(X)\), objects consist of a common factor together with an ordered presentation map onto \(X\). In the structure category \(\mathbf{PowStr}_{k,p}(X)\), objects consist of equivalence relations on \(X\), quotient metrics, and comparison isometries between the quotient factors. We identify these two ways of recording the data through an equivalence of categories.

\begin{theorem}
There is a functor \(\mathcal F_{k,p}:\mathbf{PowPres}_{k,p}(X)\to\mathbf{PowStr}_{k,p}(X)\), and there is a functor \(\mathcal G_{k,p}:\mathbf{PowStr}_{k,p}(X)\to\mathbf{PowPres}_{k,p}(X)\). These functors form an equivalence of categories.

Moreover, \(\mathcal F_{k,p}\mathcal G_{k,p}=\mathrm{id}_{\mathbf{PowStr}_{k,p}(X)}\). There is a natural isomorphism
\[
\eta:
\mathcal G_{k,p}\mathcal F_{k,p}
\Rightarrow
\mathrm{id}_{\mathbf{PowPres}_{k,p}(X)}
\]
which identifies the second composite with the identity.

Consequently, the isomorphism classes of ordered \(\ell^p\) \(k\)-power presentations of \((X,d)\) correspond bijectively to ordered \(\ell^p\) \(k\)-power structures on \((X,d)\). The category \(\mathbf{PowPres}_{k,p}(X)\) is thin, and every ordered \(\ell^p\) \(k\)-power presentation has trivial automorphism group.
\end{theorem}

The thinness statement means that there is at most one coordinatewise comparison map between any two ordered \(\ell^p\) \(k\)-power presentations. In particular, such a presentation has no nontrivial coordinatewise automorphisms. Through the categorical equivalence, we identify each ordered presentation with the corresponding ordered quotient-coordinate structure on \(X\), with only the representative of the common factor varying up to isometry. We therefore express root classification through structures defined directly on \(X\).

We separate the study of \(k\)-th roots into the existence of roots and the classification of ordered \(k\)-power presentations arising from roots. From a root \(R\) and an isometry \(\Delta_k(R)\cong(X,d)\), we obtain an ordered \(k\)-power presentation of \(X\). We then use the categorical equivalence to identify these ordered presentations with ordered quotient-coordinate structures on \(X\). When \(X\) has a finite \(\ell^\infty\)-prime decomposition, a \(k\)-th root exists precisely when \(k\) divides each prime multiplicity.


In the \(\ell^\infty\)-case, finite decompositions into \(\ell^\infty\)-prime factors determine prime multiplicities, and root existence is equivalent to divisibility conditions on these multiplicities. For the finite distance set \(D=d(X\times X)\), the threshold relational structure \(\mathcal T_D(X)\) records the threshold relations determined by the distances in \(D\). Under \(\mathcal T_D\), \(\ell^\infty\)-product decompositions correspond to direct decompositions of relational structures. We then apply Theorem~\ref{thm:mckenzie-aq} to those relational decompositions. Using this encoding and strict refinement, we prove existence of finite \(\ell^\infty\)-prime factorizations in the finite case and uniqueness under the connected and separating threshold hypothesis.

Here \(\mathbf{PowStr}_k(X)\) denotes the category of ordered \(\ell^\infty\) \(k\)-power structures on \(X\).

\begin{theorem}
Let \((X,d)\) satisfy the hypotheses of Theorem~\ref{thm:finite-distance-prime-uniqueness}. Suppose that
\[
(X,d)
\cong
A_1^{m_1}
\times_\infty
\cdots
\times_\infty
A_s^{m_s},
\]
where \(A_1,\dots,A_s\) are pairwise nonisometric \(\ell^\infty\)-prime metric spaces and \(m_i\geq1\) for every \(1\leq i\leq s\). Let \(k\geq2\).

If some \(m_i\) is not divisible by \(k\), then \(\mathbf{PowStr}_k(X)\) has no objects. If \(k\mid m_i\) for every \(1\leq i\leq s\), then
\[
\sqrt[k]{X}
=
A_1^{m_1/k}
\times_\infty
\cdots
\times_\infty
A_s^{m_s/k}
\]
is well defined up to isometry, and every ordered \(\ell^\infty\) \(k\)-power structure on \(X\) arises from a presentation whose root is isometric to \(\sqrt[k]{X}\).

After fixing an isometry \(\phi_0:\Delta_k(\sqrt[k]{X})\to X\), define
\[
H_{\phi_0,k}
=
\{\phi_0\circ f^{\times k}\circ\phi_0^{-1}:f\in\operatorname{Iso}(\sqrt[k]{X})\}.
\]
This set is a subgroup of \(\operatorname{Iso}(X)\). Then the right-coset set \(\operatorname{Iso}(X)/H_{\phi_0,k}\) is in bijection with \(\operatorname{Ob}(\mathbf{PowStr}_k(X))\).
\end{theorem}

Thus a \(k\)-th root exists exactly when \(k\) divides every prime multiplicity. Once a root exists, the ordered \(\ell^\infty\) \(k\)-power structures are in bijection with the right cosets of \(H_{\phi_0,k}\) in \(\operatorname{Iso}(X)\).

We next consider metric spaces satisfying \((X,d)\cong(X\times X,d_\infty)\). An isometry \(\Phi:(X,d)\to X\times_\infty X\) defines coordinate components \(f_i:=\pi_i\circ\Phi\) for \(i\in\{0,1\}\). Iterating these maps gives coordinate maps indexed by finite binary words. For \(\Sigma:=\{0,1\}^{\mathbb N}\), limits along branches define pseudometrics indexed by \(\Sigma\). A reconstructive binary branch system consists of the associated quotient metric spaces, the branch isometries, and the subset on which reconstruction occurs. Such a system defines a metric space \((S,d_{\sup})\) and an associated map \(\Theta:S\to S\times_\infty S\).

\begin{theorem}
A metric space \((X,d)\) is isometric to its \(\ell^\infty\)-square if and only if there exists a reconstructive binary branch system whose associated metric space \((S,d_{\sup})\) is isometric to \((X,d)\).

For such a reconstructive binary branch system, the associated map \(\Theta:S\to S\times_\infty S\) is a bijective isometry from \((S,d_{\sup})\) to \((S\times S,d_\infty)\).
\end{theorem}

Along each branch \(\beta\in\Sigma\), the iterated coordinate maps define a limiting pseudometric \(\rho_\beta\). For all \(x,y\in X\), the branch pseudometrics satisfy
\[
d(x,y)
=
\sup_{\beta\in\Sigma}\rho_\beta(x,y).
\]
The reconstructive condition then makes \(\Theta:(S,d_{\sup})\to S\times_\infty S\) a bijective isometry.

In products with finitely many isometry classes of prime factors, infinite multiplicity for every such isometry class implies \(X\cong X\times_\infty X\). Under cardinal multiplicity uniqueness, the isometry \(X\cong X\times_\infty X\) implies that every isometry class of prime factors has infinite multiplicity.

\begin{theorem}
Let \((X,d)\) satisfy the hypotheses of Theorem~\ref{thm:finite-distance-prime-uniqueness}. Suppose that
\[
(X,d)
\cong
\prod_{a\in A}
P_a^{\kappa_a},
\]
where \(A\) is finite, where the metric spaces \(P_a\) are pairwise nonisometric finite \(\ell^\infty\)-prime metric spaces, and where every \(\kappa_a\) is a nonzero cardinal. Assume that \(X\) has cardinal multiplicity uniqueness. Then \((X,d)\) is isometric to its \(\ell^\infty\)-square if and only if every \(\kappa_a\) is infinite.
\end{theorem}

The decomposition of \(X\times_\infty X\) doubles each multiplicity, so the isometry class of \(P_a\) has multiplicity \(2\kappa_a\), where \(2\kappa_a\) denotes the cardinal sum \(\kappa_a+\kappa_a\). By cardinal multiplicity uniqueness, one has \(\kappa_a=2\kappa_a\) for each \(a\in A\), and a nonzero cardinal satisfies this equality exactly when it is infinite.

Ulam asked whether an isometry between metric squares determines an isometry between the original spaces \cite{ulam1960collection}. Fournier separated square isometry from factor recovery for the Euclidean product metric with the subspaces \(X=\mathbb Q\) and \(Y=\sqrt{2}\mathbb Q\) of \(\mathbb R\), equipped with the subspace metric \cite[p.~622]{fournier1971problem}. Restricting the rotation of \(\mathbb R^2\) through angle \(\pi/4\) gives an isometry between their Euclidean squares.

Moszy{\'n}ska formulated Ulam's question as a special case of the general uniqueness problem for metric products \cite{moszynska1992uniqueness}. For compact connected metric spaces of finite dimension, Moszy{\'n}ska reduced metric-product uniqueness to uniqueness for direct sums of subsets of a linear space. In this formulation, the product metric remains part of the uniqueness question. Herburt and Moszy{\'n}ska studied this dependence through product metrics induced by functions from \((\mathbb R^+)^2\) to \(\mathbb R^+\) \cite{herburt1991metric}. Herburt showed that cancellation can fail in general for metric products satisfying \(A\times C\cong B\times C\) and \(A\not\cong B\) \cite{herburt1994there}.

For the Cartesian product metric obtained by summing factor distances, Tardif introduced prefibers as intrinsic analogues of coordinate fibers \cite[Definition~2.1]{tardif1992prefibers}. A prefiber \(A\subseteq X\) has a projection \(p_A:X\to A\) such that \(d(x,y)=d(x,p_A(x))+d(p_A(x),y)\) for each \(x\in X\) and each \(y\in A\). Tardif used prefiber projections to prove finite Cartesian factorization into indecomposable metric spaces, with factors determined up to isomorphism and coordinate permutation \cite[Theorem~1.2]{tardif1992prefibers}. Avgustinovich and Fon-Der-Flaass proved uniqueness of finite decompositions for a class of metric product operations that includes finite \(\ell^p\)-products with \(1\leq p<\infty\) \cite{avgustinovich2000cartesian}. These uniqueness results compare decompositions once product structures have been specified. We study recognition from the total metric by characterizing when equivalence relations, quotient metrics, and the quotient-coordinate map reconstruct a product or power presentation.

We use the interaction between the maximum in the supremum product and relations defined by distance thresholds to encode metric spaces with finite distance sets as threshold relational structures. For a distance value \(r\in D\), the corresponding binary relation consists of pairs whose distance is at most \(r\). Under this encoding, finite \(\ell^\infty\)-product decompositions correspond to direct decompositions of the associated relational structures. McKenzie's strict refinement theory provides the corresponding uniqueness theory for finite relational decompositions \cite[Definition~2.1, p.~71 and Theorem~9.2, p.~98]{mckenzie1971cardinal}. This identifies McKenzie's refinement theory as the uniqueness mechanism for \(\ell^\infty\)-prime decompositions under the connected and separating threshold hypothesis. The resulting prime multiplicities serve as the invariants connecting finite \(\ell^\infty\)-factorization with the original power and square questions.

Foertsch and Lytchak proved a de Rham decomposition theorem for geodesic metric spaces of finite affine rank with the Euclidean product operation \cite[Theorem~1.1]{Foertsch2008}. They proved that such spaces decompose into a Euclidean factor and irreducible factors that are neither points nor isometric to the real line, with the non-Euclidean factor fibers through each point uniquely determined up to permutation. We prove recognition results for finite \(\ell^p\)- and \(\ell^\infty\)-product and power presentations of arbitrary metric spaces.

One could expect to use the categorical classification of square structures to reduce square cancellation to the comparison of square structures. Given a square presentation \(\phi:\Delta_2(P)\to X\), we form the coordinate maps \(q_i^\phi=\pi_i\circ\phi^{-1}\) for \(i\in\{1,2\}\). We define equivalence relations \(E_i^\phi\) from the maps \(q_i^\phi\), and we recover \(P\), up to isometry, as the quotient metric space \((X/E_i^\phi,\bar d_{E_i^\phi})\) for each \(i\in\{1,2\}\). By Theorem~\ref{thm:categorical-classification}, we preserve all ordered presentation data through this reconstruction. Let \(\mathfrak s=(E_1,E_2,\tau)\) and \(\mathfrak t=(F_1,F_2,\sigma)\) be square structures on \((X,d)\). A proof of square cancellation through quotient-coordinate structures would require a theorem that proves \((X/E_1,\bar d_{E_1})\cong(X/F_1,\bar d_{F_1})\) for every two square structures \(\mathfrak s\) and \(\mathfrak t\). We use \(\tau\) and \(\sigma\) to compare the quotient factors inside their respective square structures. A cross-structure comparison therefore requires either a rule defined from \(d\) that recovers or enumerates the coordinate classes, or an invariant common to all square structures on \(X\) that determines the quotient-factor isometry type.

It is natural to ask whether the finite-spectrum root theory extends to compact or complete square-cancellation problems. Within the finite-distance threshold hypotheses, we use finite \(\ell^\infty\)-prime multiplicities in place of direct recognition of coordinate classes. Using Theorem~\ref{thm:finite-prime-factorization}, Lemma~\ref{lem:k-root-existence}, and Lemma~\ref{lem:k-root-uniqueness}, we reduce \(k\)-th root existence and uniqueness to divisibility of these multiplicities. After identifying the root, we use Theorem~\ref{thm:ordered-k-power-structure-classification} to parametrize ordered structures by the corresponding isometry-group quotient. To obtain such an extension, one would need replacements for the finite distance set, the finite threshold relational structure, and the finite collection of prime isometry classes with unique multiplicities. Thus a compact or complete cancellation theorem would require a multiplicity invariant, or another cross-structure invariant, that compares quotient factors across square structures.

One could also try to use binary branch systems to obtain canonical invariants of self-square spaces. Given a chosen self-square isometry \(X\to X\times_\infty X\), we obtain binary branch data that record its iterated coordinate splittings. Under cardinal multiplicity uniqueness, we use Theorem~\ref{thm:conditional-finite-spectrum-self-square-classification} to identify infinite prime multiplicities as the finite-spectrum source of self-square behavior. The construction uses the chosen self-square isometry, so the resulting branch data may contain features from the chosen splitting map together with features forced by \(d\). Thus repeated splitting of one space and comparison of quotient factors from two square structures require different types of invariants. A future self-square classification must therefore distinguish metric-forced branch features from splitting-dependent branch features.

Let \(X\) be an arbitrary nonempty metric space. We characterize finite ordered product presentations and finite ordered power presentations in terms of equivalence relations on \(X\), the quotient distances that these relations define, and the associated quotient-coordinate map. When the quotient distances are metrics and the relevant product formula for the chosen product metric holds, the quotient metric spaces are the coordinate factors, and the quotient-coordinate map is the presentation map. We formulate the categorical equivalence for ordered power presentations and coordinatewise morphisms. A separate quotient construction identifies unordered presentations and product isometries that are not induced coordinatewise. For \(1\leq p<\infty\), we apply analogous criteria to finite \(\ell^p\)-product presentations and finite ordered \(\ell^p\)-power presentations.

For \(\ell^\infty\), we state root existence and classification under finite prime multiplicities determined by a finite \(\ell^\infty\)-prime decomposition. We prove that finite metric spaces admit such decompositions. For metric spaces with finite distance set, we prove uniqueness and root classification under a finite \(\ell^\infty\)-prime decomposition and the connected and separating threshold hypothesis. We use the equilateral examples to show that finite prime decomposability alone does not imply the threshold hypothesis. For \(1\leq p<\infty\), our results on this product metric concern finite product and power recognition.

A choice of self-square isometry determines an ordered binary splitting of \(X\), and different choices may determine different binary branch systems. Thus the branch system is part of the chosen self-square structure rather than part of \(X\) alone. We derive the converse from a self-square structure to infinite prime multiplicities under the cardinal multiplicity uniqueness hypothesis. For infinite products, we require the supremum product metric to be finite-valued.

In the setting of persistent homology, let \(P\) denote a metric space of filtration parameters. In classical persistence, both birth and death parameters take values in the same filtration-parameter metric space \(P\) \cite{edelsbrunner2008persistent}. Consequently, the space of intervals has the form \(P^2\). The same filtration-parameter description applies to multidimensional intervals in multiparameter persistence \cite{botnan2023introduction}. When a multidimensional interval has \(k\) filtration parameters, the space of intervals has the form \(P^k\) if all filtration parameters take values in a common metric space \(P\), and it has the form \(P_1\times\cdots\times P_k\) if the filtration parameters take values in metric spaces \(P_1,\dots,P_k\). Thus spaces of intervals in persistence have power structures when all filtration parameters take values in a common metric space and product structures when the filtration parameters take values in different metric spaces.

We write \(X\) for the resulting space of intervals and equip \(X\) with a product metric \(d\) of one of the types considered in this paper. A filtration compares parameter values through a preorder on the filtration parameters. Parameter values that are equivalent under the preorder determine the same filtration step. Intervals whose filtration parameter values determine the same filtration step are precisely the trivial intervals. The collection of such trivial intervals forms the diagonal subset \(A\subseteq X\) determined by the preorder. When equivalence under the preorder reduces to equality of parameter values, \(A\) equals the usual diagonal. When one equivalence class for the preorder contains more than one parameter value, several intervals can correspond to the same filtration step, so \(A\) can be larger than the usual diagonal defined by equality of filtration parameter values. Thus \((X,d,A)\) is the space of intervals together with the diagonal subset determined by the preorder and hence is a metric pair \cite[Sect.~2.1]{bubenik2022virtual}. The filtration-parameter metric spaces specify \(X\), the filtration preorder specifies \(A\), and the pair of choices gives the metric pair used for persistence diagrams.

The metric pair \((X,d,A)\) gives the input for persistence diagrams, and we define the diagrams as formal sums on metric pairs \cite[Sect.~2.3]{bubenik2022virtual}. Grothendieck completion then gives virtual persistence diagrams \cite[Sect.~4]{bubenik2022virtual}. As in Figure~\ref{fig:square-persistence-diagram}, the product and power structures that we study in this paper concern the space \(X\) before the formal-sum construction of persistence diagrams.

\begin{figure}[htbp]
\centering
\includegraphics[width=0.78\textwidth]{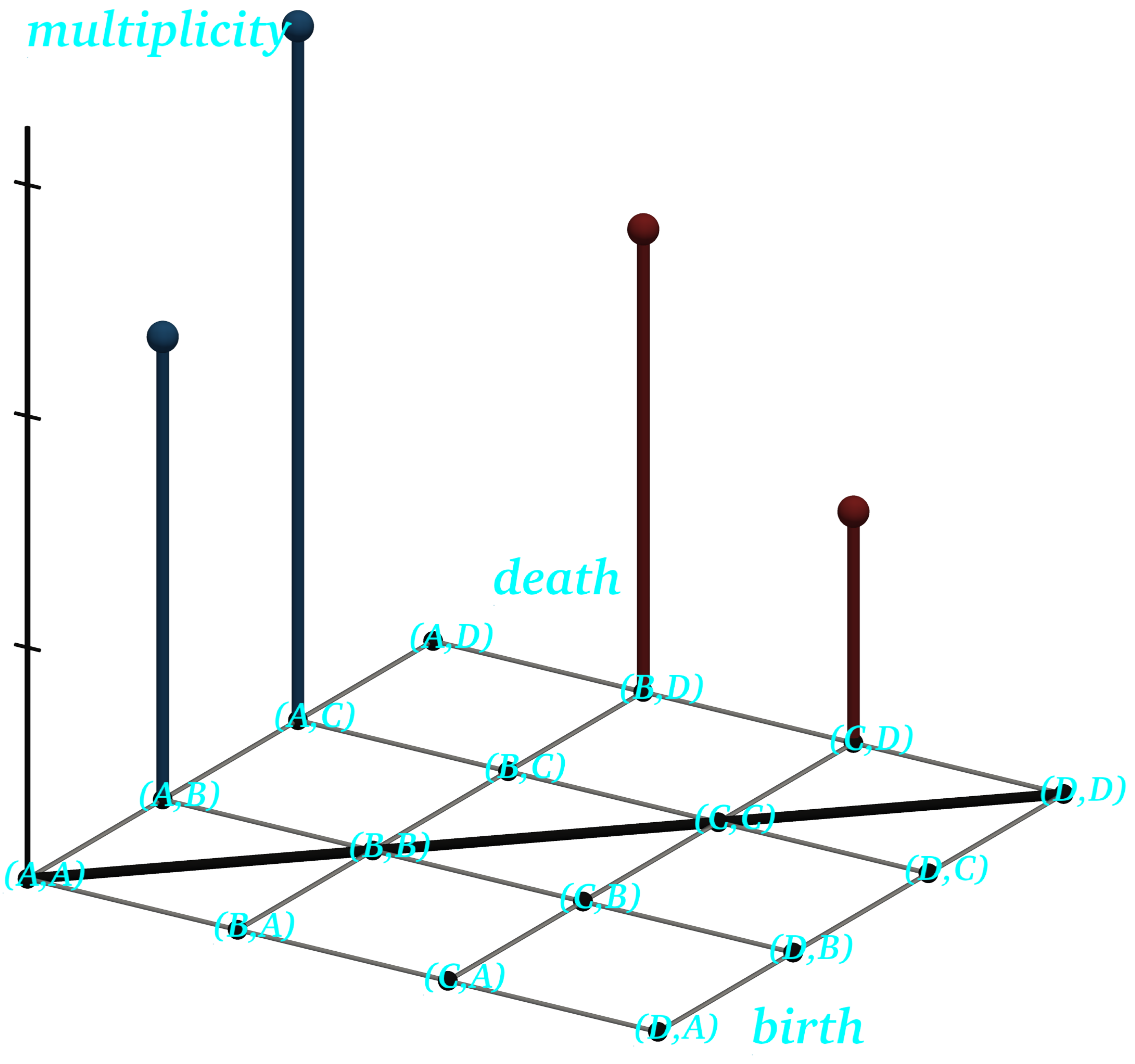}
\caption{
A persistence diagram represented as a formal sum over an interval space \(P^2\). The base square is the same product metric space as in Figure~\ref{fig:square-metric-space}, and the vertical coordinate records multiplicity. Product and power recognition apply to the underlying interval geometry before passing to persistence diagrams.
}
\label{fig:square-persistence-diagram}
\end{figure}

We may therefore study interval geometries independently of a chosen filtration realization. This viewpoint gives an inverse reconstruction problem in which we recover the filtration-parameter metric spaces whose products or powers realize a given interval geometry. Under the hypotheses of the relevant classification results, we use the quotient-coordinate and categorical classifications to recover the product and power presentations of the underlying interval space \(X\). From product presentations, we recover collections of filtration-parameter metric spaces. From power presentations, we identify common filtration-parameter metric spaces. In the \(k\)-power case, roots are precisely the common filtration-parameter metric spaces whose \(k\)-fold powers realize the interval space. When \(X\cong X\times_\infty X\), we use the self-square theory to describe interval geometries that admit iterated coordinate splitting and to obtain branch representations of those splittings. Through those branch representations, we record how parameter-space realizations of the interval geometry split into iterated coordinates.

\section{McKenzie's Refinement Theory for Finite Relational Structures}

A metric space with finite distance set determines a finite system of threshold relations through the conditions \(d(x,y)\leq r\), and these relations convert the metric decomposition problem into a decomposition problem for finite relational structures. The maximum operation in an \(\ell^\infty\)-product makes threshold relations compatible with direct products, so direct decompositions of the threshold structure encode \(\ell^\infty\)-product decompositions. We use McKenzie's strict refinement theory to compare these direct decompositions and obtain the directly indecomposable factor multiplicities needed for root existence and root classification.

The strength of this reduction is that local relational hypotheses control global product decompositions. McKenzie's criterion turns thinness and membership in \(\mathcal Q\) for one binary relation into strict refinement for the whole structure. Thus one relation can impose common refinements on all finite direct decompositions of the threshold structure. The resulting prime multiplicities give the divisibility criterion for \(k\)-th roots and provide the classification input for ordered root structures.

We later represent metric spaces with finite distance set by nested threshold relations. For a finite distance set \(D\), the indexed family of relations \(R_r=\{(x,y):d(x,y)\leq r\}\), where \(r\in D\), determines the metric \(d\). This representation places finite-distance \(\ell^\infty\)-product questions inside the theory of direct decompositions of finite relational structures. McKenzie's refinement theory gives the corresponding decomposition theory for relational structures through factor relations, decomposition functions, and strict refinement \cite[pp.~60--101]{mckenzie1971cardinal}.

We follow McKenzie~\cite[pp.~60--101]{mckenzie1971cardinal}. We state only the definitions and results that we use later. The definitions below apply to relational structures in general, and we apply McKenzie's criterion to finite relational structures later.

A structure is a system \(\mathcal M=\langle A,(R_t)_{t\in T}\rangle\), where \(A\) is a nonempty fundamental set and each \(R_t\) is a fundamental relation of finite rank over \(A\). Let \(\mathcal M_t=\langle A_t,(R_{t,s})_{s\in S}\rangle\) for \(t\in T\) be a family of structures of the same similarity type. Their direct product is the structure \(\prod_{t\in T}\mathcal M_t=\langle\prod_{t\in T}A_t,(R_s)_{s\in S}\rangle\), where, for every \(s\in S\) of rank \(n\),
\begin{equation}
\label{eq:mckenzie-product-relation}
(\mathbf a_1,\dots,\mathbf a_n)\in R_s
\Longleftrightarrow
(a_1(t),\dots,a_n(t))\in R_{t,s}
\quad\text{for every }t\in T.
\end{equation}

Let \(F\) be an equivalence relation on the fundamental set \(A\) of a structure \(\mathcal M\). Let \(\pi_F:A\to A/F\) denote the canonical quotient map. McKenzie defines the quotient structure by \(\mathcal M/F=\pi_F(\mathcal M)\).

Suppose that \(H\subseteq F_i\) for each \(i\in I\), where \(H\) and each \(F_i\) are equivalence relations on \(A\). McKenzie defines
\begin{equation}
\label{eq:mckenzie-product-equivalence}
H
=
\prod_{i\in I}F_i
\end{equation}
to mean that the map
\begin{equation}
\label{eq:mckenzie-factor-map}
a/H
\mapsto
(a/F_i)_{i\in I}
\end{equation}
defines an isomorphism \(\mathcal M/H\cong\prod_{i\in I}\mathcal M/F_i\). McKenzie proves~\cite[p.~63, Definition~1.1]{mckenzie1971cardinal} that Equation~\eqref{eq:mckenzie-product-equivalence} holds if and only if the following conditions hold:
\begin{enumerate}
\item
One has \(H=\bigcap_{i\in I}F_i\).

\item
For each \((a_i)_{i\in I}\in A^I\) such that \(a_iF_ja_j\) for \(i,j\in I\), there exists \(a\in A\) such that \(aF_ia_i\) for \(i\in I\).

\item
For every fundamental relation \(R_t\) of rank \(n\), if \(\mathbf a=(a_1,\dots,a_n)\) and \(\mathbf b=(b_1,\dots,b_n)\) satisfy \(a_kF_ib_k\) for \(i\in I\) and \(1\le k\le n\), then
\begin{equation}
\label{eq:mckenzie-factor-invariance}
\mathbf a\in R_t
\Longleftrightarrow
\mathbf b\in R_t.
\end{equation}
\end{enumerate}

McKenzie calls each \(F_i\) in Equation~\eqref{eq:mckenzie-product-equivalence} a factor relation of \(H\) over \(\mathcal M\). He writes \(\mathrm{FR}(\mathcal M,H)\) for the set of all factor relations of \(H\) over \(\mathcal M\). He also writes \(\mathrm{FR}(\mathcal M)=\mathrm{FR}(\mathcal M,\mathrm{id}_A)\).

When \(\mathrm{id}_A=F\times F'\), the factor relations determine quotient maps induced by inclusion of equivalence relations. The quotient structures form the following commutative diagram:
\[
\begin{tikzcd}[row sep=large,column sep=large]
&
\mathcal M/(A\times A)
&
\\
\mathcal M/F
\arrow[ur]
&
&
\mathcal M/F'
\arrow[ul]
\\
&
\mathcal M/\mathrm{id}_A
\arrow[ul]
\arrow[ur]
&
\end{tikzcd}
\]
We identify \(\mathcal M/\mathrm{id}_A\) with \(\mathcal M\), while the middle quotients are the two coordinate factors determined by the factor relations. The decomposition condition \(\mathrm{id}_A=F\times F'\) states that the map \(a/\mathrm{id}_A\mapsto(a/F,a/F')\) defines an isomorphism \(\mathcal M/\mathrm{id}_A\cong\mathcal M/F\times\mathcal M/F'\).

Suppose that \(\mathrm{id}_A=F\times F'\) in the sense of Equation~\eqref{eq:mckenzie-product-equivalence}. McKenzie defines the associated decomposition function as the unique function \(f:A^2\to A\) such that
\begin{equation}
\label{eq:mckenzie-decomposition-function}
xFf(x,y)
\qquad\text{and}\qquad
f(x,y)F'y
\end{equation}
for all \(x,y\in A\). McKenzie writes \(\mathrm{DF}(\mathcal M)\) for the set of all decomposition functions of \(\mathcal M\). He also defines \(f_x(y)=f(y,x)\) and \(f^x(y)=f(x,y)\).

Suppose that \(\mathcal M\cong\prod_{i\in I}\mathcal B_i\cong\prod_{j\in J}\mathcal C_j\). McKenzie says that these two decompositions admit a common refinement if there exist structures \(\mathcal D_{i,j}\) such that
\begin{equation}
\label{eq:mckenzie-common-refinement}
\mathcal B_i
\cong
\prod_{j\in J}\mathcal D_{i,j}
\qquad\text{and}\qquad
\mathcal C_j
\cong
\prod_{i\in I}\mathcal D_{i,j}
\end{equation}
for all \(i\in I\) and \(j\in J\). He says that \(\mathcal M\) has the refinement property if every pair of direct decompositions of \(\mathcal M\) admits a common refinement.

Let \(R,S\subseteq A^2\) be binary relations. McKenzie defines the relative product of binary relations by
\begin{equation}
\label{eq:mckenzie-relative-product}
x(RS)z
\Longleftrightarrow
\text{there exists }y\in A\text{ such that }xRy\text{ and }ySz.
\end{equation}
He also defines the converse relation by \(xR^{-1}y\Longleftrightarrow yRx\).

McKenzie defines the strict refinement property as follows~\cite[p.~64]{mckenzie1971cardinal}. A structure \(\mathcal M\) has the strict refinement property if Equation~\eqref{eq:mckenzie-strict-refinement} implies Equation~\eqref{eq:mckenzie-strict-refinement-components}, where
\begin{equation}
\label{eq:mckenzie-strict-refinement}
\mathrm{id}_A
=
\prod_{i\in I}F_i
=
\prod_{j\in J}G_j
\end{equation}
and
\begin{equation}
\label{eq:mckenzie-strict-refinement-components}
F_i
=
\prod_{j\in J}(F_iG_j)
\qquad\text{and}\qquad
G_j
=
\prod_{i\in I}(F_iG_j).
\end{equation}
Here the second equation holds for all \(i\in I\) and \(j\in J\), and juxtaposition denotes the relative product from Equation~\eqref{eq:mckenzie-relative-product}.

McKenzie proves~\cite[p.~64, Theorem~1.4]{mckenzie1971cardinal} that a structure \(\mathcal M\) has the strict refinement property if and only if
\begin{equation}
\label{eq:mckenzie-srp-functional}
f_x(g_y(z))
=
g_y(f_x(z))
\end{equation}
for all \(f,g\in\mathrm{DF}(\mathcal M)\) and all \(x,y,z\in A\).

Here \(\#(\mathcal M)\) denotes the cardinality of the fundamental set of \(\mathcal M\). McKenzie says that a structure \(\mathcal M\) is directly indecomposable if \(\#(\mathcal M)>1\) and every direct decomposition \(\mathcal M\cong\mathcal B\times\mathcal C\) satisfies \(\#(\mathcal B)=1\) or \(\#(\mathcal C)=1\).

McKenzie says that a binary relation \(R\subseteq A^2\) is reflexive over \(A\) if \(\mathrm{id}_A\subseteq R\). He says that \(R\) is connected over \(A\) if, for every distinct \(x,y\in A\), there exists a finite sequence \(x=x_0,x_1,\dots,x_n=y\) such that \(x_i(R\cup R^{-1})x_{i+1}\) for \(0\le i<n\).

McKenzie defines four classes of binary structures~\cite[p.~71, Definition~2.1]{mckenzie1971cardinal}:
\begin{align}
\mathcal R&=\{\langle B,S\rangle:S\ \text{is reflexive over}\ B\},\\
\mathcal R^\circ&=\{\langle B,S\rangle:S\ \text{is connected and reflexive over}\ B\},\\
\mathcal U&=\{\langle B,S\rangle:S=B^2\},\\
\mathcal Q&=\{\langle B,S\rangle:SS^{-1}\ \text{and}\ S^{-1}S\ \text{are connected over}\ B\}.
\end{align}

For a binary structure \(\mathcal B=\langle B,S\rangle\), McKenzie defines quasi-ordering relations by
\begin{align}
x\leq_\ell y
&\Longleftrightarrow
xSz\Rightarrow ySz\text{ for every }z\in B,
\\
x\leq_r y
&\Longleftrightarrow
zSx\Rightarrow zSy\text{ for every }z\in B.
\end{align}
He then defines
\begin{equation}
\label{eq:mckenzie-skeleton-equivalence}
x\sim y
\Longleftrightarrow
x\leq_\ell y,\ y\leq_\ell x,\ x\leq_r y,\ y\leq_r x.
\end{equation}
He writes \(\operatorname{Sk}(\mathcal B)=\mathcal B/{\sim}\) and calls this quotient the skeleton of \(\mathcal B\). He says that \(\mathcal B\) is thin if \(\sim=\mathrm{id}_B\).

\begin{lemma}
\label{lem:mckenzie-connected-reflexive-in-q}
If \(S\) is connected and reflexive over \(B\), then \(\langle B,S\rangle\in\mathcal Q\).
\end{lemma}

\begin{proof}
Let \(\mathcal B=\langle B,S\rangle\). Since \(S\) is reflexive,
\[
S
\subseteq
SS^{-1}
\qquad\text{and}\qquad
S
\subseteq
S^{-1}S.
\]
Equation~\eqref{eq:mckenzie-relative-product} gives these inclusions. Every sequence witnessing connectedness of \(S\) also witnesses connectedness of \(SS^{-1}\) and \(S^{-1}S\). Hence \(\mathcal B\in\mathcal Q\).
\end{proof}

McKenzie also defines the set \(A(\mathcal M)\) associated to a structure \(\mathcal M=\langle A,(R_t)_{t\in T}\rangle\). He defines \(A(\mathcal M)\) to be the set of all finitary relations \(R\) on \(A\) such that every decomposition function of \(\mathcal M\) is also a decomposition function of \(\langle A,R\rangle\)~\cite[p.~98]{mckenzie1971cardinal}.

McKenzie proves the following theorem~\cite[p.~98, Theorem~9.2]{mckenzie1971cardinal}.

\begin{theorem}[McKenzie]
\label{thm:mckenzie-aq}
Suppose that there exists a binary relation \(R\in A(\mathcal M)\) such that \(\langle A,R\rangle\) is thin and belongs to \(\mathcal Q\). Then \(\mathcal M\) has the strict refinement property.
\end{theorem}

\section{The Classification of Square Metric Spaces} \label{sec:Classification:Square}

When we only know the metric space \((X,d)\), a possible product representation does not provide its coordinates. We use equivalence relations on \(X\) as candidates for this missing coordinate information. Their quotient spaces serve as candidate coordinate factors, and their classes serve as candidate coordinate fibers. Any reconstruction must recover the factor spaces, the points of the product, and the metric. Our goal is to solve these reconstruction problems for finite \(\ell^\infty\)-products and finite \(\ell^p\)-products with \(1\leq p<\infty\).

Theorem~\ref{thm:two-factor-product-classification} characterizes product presentations through equivalence relations \(E_1,\dots,E_k\) on \(X\), where \(k\geq2\). The quotient distance induced by each \(E_j\) must define a metric on \(X/E_j\), giving the candidate coordinate factors. Each choice of one equivalence class from every quotient \(X/E_j\) must have intersection consisting of exactly one point, which reconstructs points through quotient coordinates. The final requirement reconstructs the metric by requiring \(d(x,y)\) to equal the maximum of the corresponding quotient distances in the \(\ell^\infty\)-product case, or the \(\ell^p\)-product formula determined by the quotient distances when \(1\leq p<\infty\). Every product presentation yields such equivalence relations. Conversely, such equivalence relations make the quotient-coordinate map a bijective isometry onto the corresponding product.

Metric-product uniqueness and cancellation results compare the factors arising from product decompositions and product isometries \cite{moszynska1992uniqueness,herburt1994there,tardif1992prefibers,avgustinovich2000cartesian}. We ask the prior recognition question of whether the metric space \((X,d)\) itself contains enough information to reconstruct product coordinates. We use equivalence relations and quotient metrics as intrinsic replacements for the missing product data.

We reduce product recognition for \((X,d)\) to intrinsic conditions on equivalence relations and quotient metrics on \(X\), without assuming a product representation at the outset. We extend this reconstruction to powers by requiring the quotient factors to share a common metric model, and to self-powers by requiring the common model to be isometric to \((X,d)\). The same quotient reconstruction framework handles products, powers, and self-powers.


All metric spaces in this section are nonempty.

Let $E$ be an equivalence relation on a metric space $(X,d)$. For classes $[x]_E,[y]_E\in X/E$, define
\begin{equation}\label{eq:quotient-distance}
\bar d_E([x]_E,[y]_E)
:=
\inf\{d(x',y'):x' E x,\ y' E y\}.
\end{equation}
The right side depends only on the classes $[x]_E$ and $[y]_E$, since replacing $x$ by any representative of $[x]_E$ and replacing $y$ by any representative of $[y]_E$ leaves the admissible set unchanged. When $\bar d_E$ is a metric on $X/E$, we call it the quotient metric induced by $E$.

\begin{lemma}\label{lem:quotient-transfer}
Let $(X,d_X)$ and $(Y,d_Y)$ be metric spaces. Let $E$ and $F$ be equivalence relations on $X$ and $Y$. Let $\Phi:X\to Y$ be a bijective isometry, and assume that $x E y$ if and only if $\Phi(x) F \Phi(y)$. Define $\Theta:X/E\to Y/F$ by $\Theta([x]_E):=[\Phi(x)]_F$. Then $\Theta$ is a bijection, and for all $x,y\in X$,
\[
\bar d_E([x]_E,[y]_E)
=
\bar d_F(\Theta([x]_E),\Theta([y]_E)).
\]
Consequently, $\Theta$ is an isometry whenever $\bar d_E$ and $\bar d_F$ are metrics.
\end{lemma}

\begin{proof}
The equivalence $x E y$ if and only if $\Phi(x) F \Phi(y)$ shows that $\Theta$ is well-defined and injective. Let $A$ be an $F$-class. Since $\Phi$ is surjective, we choose $x\in X$ such that $\Phi(x)\in A$. Then $\Theta([x]_E)=A$. Hence $\Theta$ is surjective.

For all $x,y\in X$,
\[
\begin{aligned}
\bar d_E([x]_E,[y]_E)
&=
\inf\{d_X(x',y'):x' E x,\ y' E y\} \\
&=
\inf\{d_Y(\Phi(x'),\Phi(y')):x' E x,\ y' E y\} \\
&=
\inf\{d_Y(u,v):u F \Phi(x),\ v F \Phi(y)\} \\
&=
\bar d_F([\Phi(x)]_F,[\Phi(y)]_F).
\end{aligned}
\]
The definition of $\Theta$ gives $\Theta([x]_E)=[\Phi(x)]_F$ and $\Theta([y]_E)=[\Phi(y)]_F$. Therefore
$\bar d_E([x]_E,[y]_E)
=
\bar d_F(\Theta([x]_E),\Theta([y]_E))$.
If $\bar d_E$ and $\bar d_F$ are metrics, then the previous identity shows that $\Theta$ preserves distances.
\end{proof}

\begin{lemma}\label{lem:singleton-reconstruction}
Let $(X,d)$ be a metric space, let $k\geq 2$, and let $E_1,\dots,E_k$ be equivalence relations on $X$. Assume that, for every choice of classes $C_j\in X/E_j$, where $1\leq j\leq k$, the intersection $\bigcap_{j=1}^k C_j$ consists of exactly one point. Define $\Psi:X\to\prod_{j=1}^k X/E_j$ by $\Psi(x):=([x]_{E_1},\dots,[x]_{E_k})$. Then $\Psi$ is bijective.
\end{lemma}

\begin{proof}
Assume that $\Psi(x)=\Psi(y)$. Then $[x]_{E_j}=[y]_{E_j}$ for every $j\in\{1,\dots,k\}$. Hence both $x$ and $y$ lie in $\bigcap_{j=1}^k [x]_{E_j}$. The singleton-intersection hypothesis therefore gives $x=y$. Hence $\Psi$ is injective.

Let $(C_1,\dots,C_k)\in\prod_{j=1}^k X/E_j$. Let $x$ denote the unique point in $\bigcap_{j=1}^k C_j$. Then $\Psi(x)=(C_1,\dots,C_k)$. Hence $\Psi$ is surjective.
\end{proof}

\subsection{The Two-Factor Case}

We begin with a given square product presentation \(P\times P\). We show that its coordinate fibers define equivalence relations whose classes recover points by intersection, whose quotient spaces recover the factors, and whose quotient distances recover the product metric.

\begin{lemma}\label{lem:product-intersection}
Let $(P,d_P)$ be a metric space. Define equivalence relations $E_0,E_1$ on $P\times P$ by declaring $(p,q) E_0 (p',q')$ exactly when $p=p'$, and $(p,q) E_1 (p',q')$ exactly when $q=q'$. For every class $C_0$ of $E_0$ and every class $C_1$ of $E_1$, the intersection $C_0\cap C_1$ consists of exactly one point.
\end{lemma}

\begin{proof}
Each class of $E_0$ has the form $\{p\}\times P$ for some $p\in P$, and each class of $E_1$ has the form $P\times\{q\}$ for some $q\in P$. Therefore $C_0\cap C_1=\{(p,q)\}$.
\end{proof}

\begin{lemma}\label{lem:product-structure}
Let $(P,d_P)$ be a metric space, and equip $P\times P$ with the metric
\begin{equation}\label{eq:max-product-metric}
d_{P\times P}\bigl((p_0,p_1),(q_0,q_1)\bigr)
=
\max\{d_P(p_0,q_0),d_P(p_1,q_1)\}.
\end{equation}
Let $E_0,E_1$ be the equivalence relations from Lemma~\ref{lem:product-intersection}. Then the quotient distances $\bar d_{E_0}$ and $\bar d_{E_1}$ are metrics on $(P\times P)/E_0$ and $(P\times P)/E_1$. Moreover, the maps $[(p,q)]_{E_0}\mapsto p$ and $[(p,q)]_{E_1}\mapsto q$ define bijective isometries onto $(P,d_P)$, and
\begin{equation}\label{eq:product-quotient-reconstruction}
d_{P\times P}\bigl((p,q),(p',q')\bigr)
=
\max\bigl\{
\bar d_{E_0}([(p,q)]_{E_0},[(p',q')]_{E_0}),
\bar d_{E_1}([(p,q)]_{E_1},[(p',q')]_{E_1})
\bigr\}.
\end{equation}
\end{lemma}

\begin{proof}
The map $[(p,q)]_{E_0}\mapsto p$ is well-defined because all representatives of $[(p,q)]_{E_0}$ have first coordinate $p$. The definition of $E_0$ shows that the map is injective. Since $P$ is nonempty, every $p\in P$ occurs as the first coordinate of some point of $P\times P$. Hence the map is surjective.

Every representative of $[(p,q)]_{E_0}$ has the form $(p,u)$, and every representative of $[(p',q')]_{E_0}$ has the form $(p',v)$. Therefore Equation~\eqref{eq:max-product-metric} gives
\[
\bar d_{E_0}([(p,q)]_{E_0},[(p',q')]_{E_0})
=
\inf_{u,v\in P}
\max\{d_P(p,p'),d_P(u,v)\}.
\]
The right side is at least $d_P(p,p')$. If we choose $u=v$, then $\max\{d_P(p,p'),d_P(u,v)\}=d_P(p,p')$. Hence $\bar d_{E_0}([(p,q)]_{E_0},[(p',q')]_{E_0})=d_P(p,p')$.

The same well-definedness and bijectivity argument applied to second coordinates gives $\bar d_{E_1}([(p,q)]_{E_1},[(p',q')]_{E_1})=d_P(q,q')$. These identities show that the displayed maps are bijective isometries and that $\bar d_{E_0}$ and $\bar d_{E_1}$ are metrics. Equation~\eqref{eq:max-product-metric} and the quotient-distance identities give Equation~\eqref{eq:product-quotient-reconstruction}.
\end{proof}

\begin{definition}\label{def:distance-decomposition}
Let $(X,d)$ be a metric space, and let $E_0,E_1$ be equivalence relations on $X$. We say that $(E_0,E_1)$ satisfies the distance decomposition identity if, for all $x,y\in X$,
\begin{equation}\label{eq:distance-decomposition-identity}
d(x,y)
=
\max\{
\bar d_{E_0}([x]_{E_0},[y]_{E_0}),
\bar d_{E_1}([x]_{E_1},[y]_{E_1})
\}.
\end{equation}
\end{definition}

\begin{figure}[ht]
\centering
\includegraphics[width=0.72\textwidth]{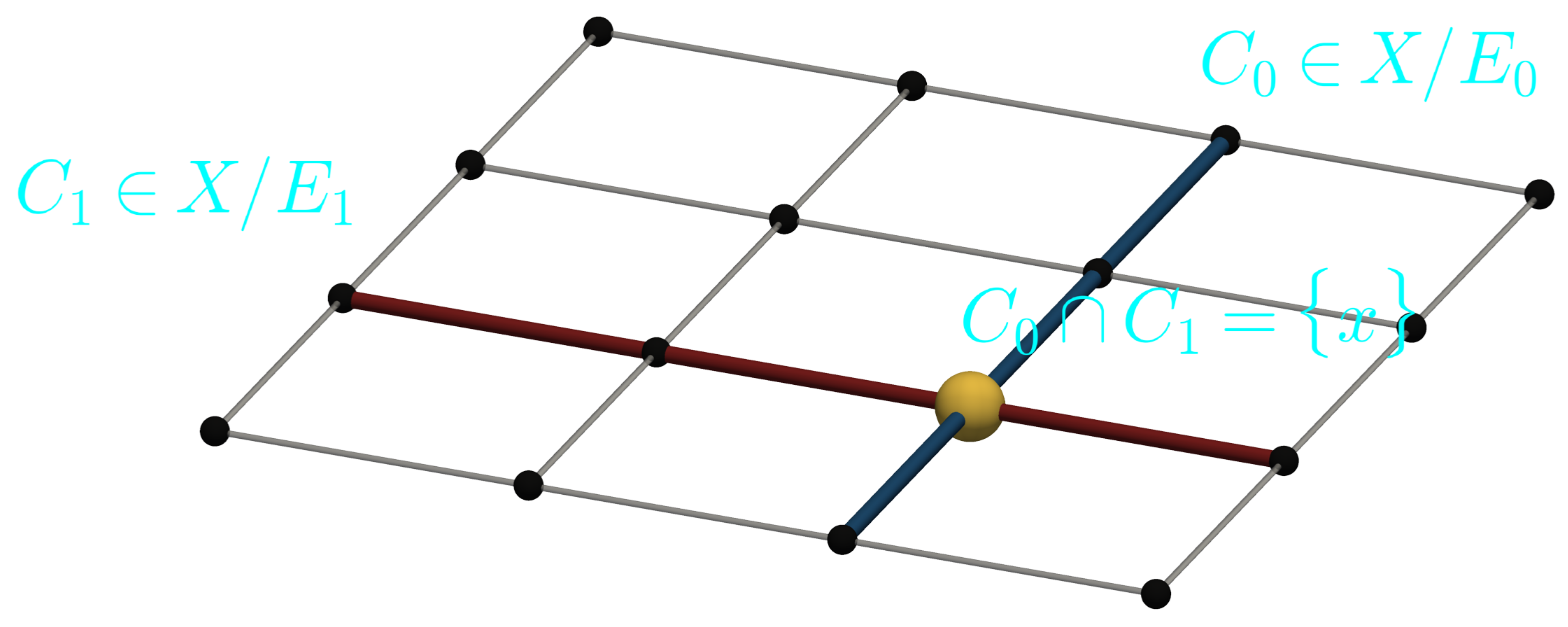}
\caption{
A two-factor product recovered from quotient fibers. The coordinates of the product space are not assumed in advance. The equivalence classes \(C_0\in X/E_0\) and \(C_1\in X/E_1\) determine the hidden coordinate directions, and their singleton intersection \(C_0\cap C_1=\{x\}\) reconstructs the point \(x\in X\).
}
\label{fig:square-coordinate-recovery}
\end{figure}

Figure~\ref{fig:square-coordinate-recovery} shows how the equivalence classes satisfying these conditions replace the missing coordinate fibers of a product presentation.

\begin{lemma}\label{lem:relations-to-product}
Let $(X,d)$ be a metric space, and let $E_0,E_1$ be equivalence relations on $X$. Assume that the quotient distances $\bar d_{E_0}$ and $\bar d_{E_1}$ are metrics on $X/E_0$ and $X/E_1$. Assume also that every class of $E_0$ meets every class of $E_1$ in exactly one point and that $(E_0,E_1)$ satisfies the distance decomposition identity. Equip $(X/E_0)\times(X/E_1)$ with the metric
\begin{equation}\label{eq:quotient-product-metric}
d_\times((C_0,C_1),(D_0,D_1))
=
\max\{
\bar d_{E_0}(C_0,D_0),
\bar d_{E_1}(C_1,D_1)
\}.
\end{equation}
Define $\Psi:X\to(X/E_0)\times(X/E_1)$ by $\Psi(x):=([x]_{E_0},[x]_{E_1})$. Then $\Psi$ is a bijective isometry from $(X,d)$ onto $\bigl((X/E_0)\times(X/E_1),d_\times\bigr)$.
\end{lemma}

\begin{proof}
Lemma~\ref{lem:singleton-reconstruction}, applied with $k=2$, shows that $\Psi$ is bijective. Equations~\eqref{eq:quotient-product-metric} and \eqref{eq:distance-decomposition-identity} give
\[
d_\times(\Psi(x),\Psi(y))
=
d(x,y).
\]
Therefore $\Psi$ is a bijective isometry.
\end{proof}

\begin{theorem}\label{thm:two-factor-product-classification}
Let $(X,d)$ be a metric space. The following are equivalent.
\begin{enumerate}
\item\label{item:two-factor-product}
There exists a metric space $(P,d_P)$ such that $(X,d)$ is isometric to $P\times P$ with the metric from Equation~\eqref{eq:max-product-metric}.

\item\label{item:two-factor-relations}
There exist equivalence relations $E_0,E_1$ on $X$ such that:
\begin{enumerate}
\item\label{item:two-factor-metrics}
the quotient distances $\bar d_{E_0}$ and $\bar d_{E_1}$ are metrics on $X/E_0$ and $X/E_1$;

\item\label{item:two-factor-intersection}
every class of $E_0$ meets every class of $E_1$ in exactly one point;

\item\label{item:two-factor-decomposition}
the pair $(E_0,E_1)$ satisfies the distance decomposition identity;

\item\label{item:two-factor-isometric}
the quotient metric spaces $(X/E_0,\bar d_{E_0})$ and $(X/E_1,\bar d_{E_1})$ are isometric.
\end{enumerate}
\end{enumerate}
\end{theorem}

\begin{proof}
Assume condition~\ref{item:two-factor-product}. Let $\Phi:X\to P\times P$ be a bijective isometry, and let $F_0,F_1$ denote the equivalence relations from Lemma~\ref{lem:product-intersection}. Let $\pi_0,\pi_1:P\times P\to P$ denote the coordinate projections. Define equivalence relations $E_0,E_1$ on $X$ by declaring $x E_0 y$ exactly when $\pi_0(\Phi(x))=\pi_0(\Phi(y))$, and $x E_1 y$ exactly when $\pi_1(\Phi(x))=\pi_1(\Phi(y))$.

For each $i\in\{0,1\}$, define $\Theta_i([x]_{E_i}):=[\Phi(x)]_{F_i}$. Lemma~\ref{lem:quotient-transfer} shows that $\Theta_i$ is a bijective isometry from $(X/E_i,\bar d_{E_i})$ onto $((P\times P)/F_i,\bar d_{F_i})$. Lemma~\ref{lem:product-structure} gives condition~\ref{item:two-factor-metrics}. The same lemma shows that both quotient metric spaces are isometric to $(P,d_P)$. Hence condition~\ref{item:two-factor-isometric} holds.

Let $C_i\in X/E_i$ for each $i\in\{0,1\}$. The definition of $\Theta_i$ gives $\Phi(C_0\cap C_1)=\Theta_0(C_0)\cap\Theta_1(C_1)$. Lemma~\ref{lem:product-intersection} shows that the right side consists of exactly one point. Since $\Phi$ is bijective, the same holds for $C_0\cap C_1$. Hence condition~\ref{item:two-factor-intersection} holds.

Equation~\eqref{eq:product-quotient-reconstruction} and Lemma~\ref{lem:quotient-transfer} give, for all $x,y\in X$,
\[
\begin{aligned}
d(x,y)
&=
d_{P\times P}(\Phi(x),\Phi(y)) \\
&=
\max\{
\bar d_{F_0}([\Phi(x)]_{F_0},[\Phi(y)]_{F_0}),
\bar d_{F_1}([\Phi(x)]_{F_1},[\Phi(y)]_{F_1})
\} \\
&=
\max\{
\bar d_{E_0}([x]_{E_0},[y]_{E_0}),
\bar d_{E_1}([x]_{E_1},[y]_{E_1})
\}.
\end{aligned}
\]
Hence condition~\ref{item:two-factor-decomposition} holds.

Assume condition~\ref{item:two-factor-relations}. Lemma~\ref{lem:relations-to-product} shows that $\Psi(x):=([x]_{E_0},[x]_{E_1})$ defines a bijective isometry from $(X,d)$ onto $(X/E_0)\times(X/E_1)$ equipped with the metric from Equation~\eqref{eq:quotient-product-metric}. Define $(P,d_P):=(X/E_0,\bar d_{E_0})$. Condition~\ref{item:two-factor-isometric} gives a bijective isometry $\Theta:(X/E_1,\bar d_{E_1})\to(P,d_P)$.

Define $\Xi(C_0,C_1):=(C_0,\Theta(C_1))$. For all $(C_0,C_1),(D_0,D_1)\in(X/E_0)\times(X/E_1)$,
\[
\begin{aligned}
&d_{P\times P}(\Xi(C_0,C_1),\Xi(D_0,D_1)) \\
&=
\max\{
\bar d_{E_0}(C_0,D_0),
d_P(\Theta(C_1),\Theta(D_1))
\} \\
&=
\max\{
\bar d_{E_0}(C_0,D_0),
\bar d_{E_1}(C_1,D_1)
\} \\
&=
d_\times((C_0,C_1),(D_0,D_1)).
\end{aligned}
\]
Since $\Theta$ is bijective, $\Xi$ is a bijective isometry from $(X/E_0)\times(X/E_1)$, with the metric from Equation~\eqref{eq:quotient-product-metric}, onto $P\times P$ with the metric from Equation~\eqref{eq:max-product-metric}. Therefore the composition $\Xi\circ\Psi$ is a bijective isometry from $(X,d)$ onto $P\times P$ with the metric from Equation~\eqref{eq:max-product-metric}.
\end{proof}

We now specialize to the case in which both quotient metric spaces are isometric to the original metric space.

\begin{corollary}\label{cor:self-product-classification}
Let $(X,d)$ be a metric space. The following are equivalent.
\begin{enumerate}
\item\label{item:self-product-product}
The metric space $(X,d)$ is isometric to $X\times X$, where $X\times X$ carries the metric
\begin{equation}\label{eq:self-product-metric}
d_{X\times X}\bigl((x_0,x_1),(y_0,y_1)\bigr)
=
\max\{d(x_0,y_0),d(x_1,y_1)\}.
\end{equation}

\item\label{item:self-product-relations}
There exist equivalence relations $E_0,E_1$ on $X$ such that:
\begin{enumerate}
\item\label{item:self-product-metrics}
the quotient distances $\bar d_{E_0}$ and $\bar d_{E_1}$ are metrics on $X/E_0$ and $X/E_1$;

\item\label{item:self-product-intersection}
every class of $E_0$ meets every class of $E_1$ in exactly one point;

\item\label{item:self-product-decomposition}
the pair $(E_0,E_1)$ satisfies the distance decomposition identity;

\item\label{item:self-product-isometric}
each quotient metric space $(X/E_0,\bar d_{E_0})$ and $(X/E_1,\bar d_{E_1})$ is isometric to $(X,d)$.
\end{enumerate}
\end{enumerate}
\end{corollary}

\begin{proof}
Assume condition~\ref{item:self-product-product}. The forward implication in Theorem~\ref{thm:two-factor-product-classification}, together with Lemma~\ref{lem:product-structure}, gives condition~\ref{item:self-product-relations}.

Assume condition~\ref{item:self-product-relations}. The hypotheses give condition~\ref{item:two-factor-relations}. Therefore Theorem~\ref{thm:two-factor-product-classification} shows that $(X,d)$ is isometric to $X\times X$ with the metric from Equation~\eqref{eq:self-product-metric}.
\end{proof}

\subsection{The Multi-Factor Case}

We now treat finite products with the max product metric.

\begin{lemma}\label{lem:k-product-intersection}
Let $k\geq 2$, let $(P_j,d_j)$ be metric spaces for $1\leq j\leq k$, and let $Y:=P_1\times\cdots\times P_k$. For each $j\in\{1,\dots,k\}$, define an equivalence relation $F_j$ on $Y$ by
\[
(p_1,\dots,p_k) F_j (q_1,\dots,q_k)
\quad\text{exactly when}\quad
p_j=q_j.
\]
For every choice of classes $C_j\in Y/F_j$, where $1\leq j\leq k$, the intersection $\bigcap_{j=1}^k C_j$ consists of exactly one point.
\end{lemma}

\begin{proof}
Each class $C_j$ has the form $\{p\in Y:p_j=a_j\}$ for a unique $a_j\in P_j$. Hence $\bigcap_{j=1}^k C_j=\{(a_1,\dots,a_k)\}$.
\end{proof}

\begin{lemma}\label{lem:k-product-structure}
Let $k\geq 2$, let $(P_j,d_j)$ be metric spaces for $1\leq j\leq k$, and let $Y:=P_1\times\cdots\times P_k$. Equip $Y$ with the metric
\begin{equation}\label{eq:k-product-metric}
d_Y(p,q)
=
\max_{1\leq j\leq k} d_j(p_j,q_j),
\end{equation}
where $p=(p_1,\dots,p_k)$ and $q=(q_1,\dots,q_k)$. Let $F_1,\dots,F_k$ be the equivalence relations from Lemma~\ref{lem:k-product-intersection}. Then, for each $j\in\{1,\dots,k\}$, the quotient distance $\bar d_{F_j}$ is a metric on $Y/F_j$, the map $[p]_{F_j}\mapsto p_j$ defines a bijective isometry from $Y/F_j$ onto $(P_j,d_j)$, and
\begin{equation}\label{eq:k-product-quotient-reconstruction}
d_Y(p,q)
=
\max_{1\leq j\leq k}
\bar d_{F_j}([p]_{F_j},[q]_{F_j}).
\end{equation}
\end{lemma}

\begin{proof}
Fix $j\in\{1,\dots,k\}$. The same argument as in Lemma~\ref{lem:product-structure} shows that $[p]_{F_j}\mapsto p_j$ defines a bijection from $Y/F_j$ onto $P_j$.

Let $p=(p_1,\dots,p_k)$ and $q=(q_1,\dots,q_k)$. If $p' F_j p$ and $q' F_j q$, then the $j$-th coordinates of $p'$ and $q'$ are $p_j$ and $q_j$. Equation~\eqref{eq:k-product-metric} gives $d_Y(p',q')\geq d_j(p_j,q_j)$, so $\bar d_{F_j}([p]_{F_j},[q]_{F_j})\geq d_j(p_j,q_j)$.

Conversely, choose representatives of $[p]_{F_j}$ and $[q]_{F_j}$ that agree in every coordinate except the $j$-th coordinate. For such representatives, their distance in Equation~\eqref{eq:k-product-metric} equals $d_j(p_j,q_j)$. Therefore $\bar d_{F_j}([p]_{F_j},[q]_{F_j})=d_j(p_j,q_j)$.

Hence $\bar d_{F_j}$ is a metric, and the coordinate map is a bijective isometry. Substituting the quotient-distance identities into Equation~\eqref{eq:k-product-metric} gives Equation~\eqref{eq:k-product-quotient-reconstruction}.
\end{proof}

\begin{definition}\label{def:k-factor-distance-decomposition}
Let $(X,d)$ be a metric space, let $k\geq 2$, and let $E_1,\dots,E_k$ be equivalence relations on $X$. We say that $(E_1,\dots,E_k)$ satisfies the $k$-factor distance decomposition identity if, for all $x,y\in X$,
\begin{equation}\label{eq:k-factor-distance-decomposition-identity}
d(x,y)
=
\max_{1\leq j\leq k}
\bar d_{E_j}([x]_{E_j},[y]_{E_j}).
\end{equation}
\end{definition}

\begin{figure}[ht]
\centering
\includegraphics[width=0.78\textwidth]{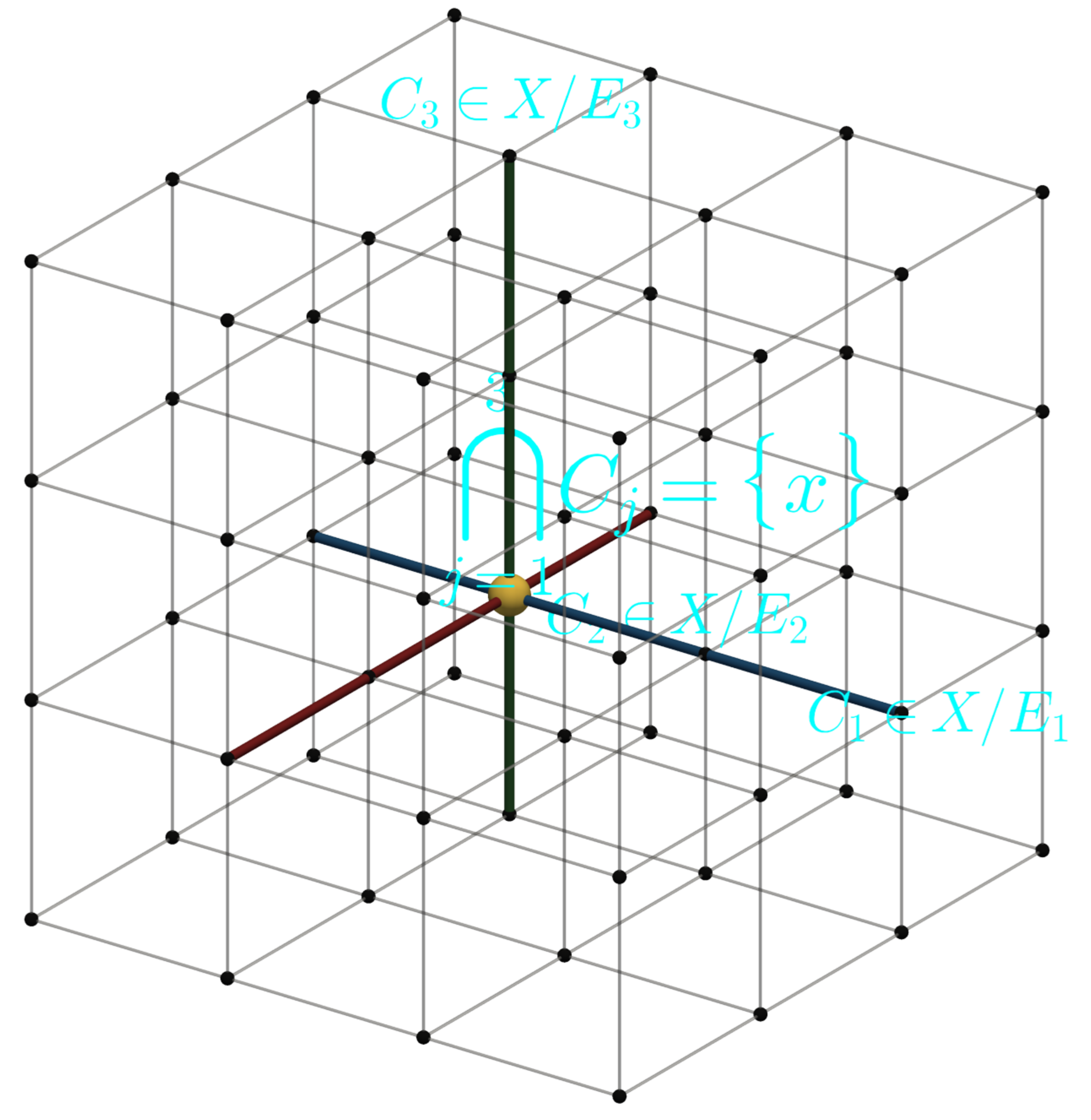}
\caption{
A three-factor instance of coordinate recovery through quotient fibers. The classes \(C_1\in X/E_1\), \(C_2\in X/E_2\), and \(C_3\in X/E_3\) recover the three independent product directions. Their intersection consists of one point, giving the finite-factor analogue of coordinate reconstruction.
}
\label{fig:cube-coordinate-recovery}
\end{figure}

Figure~\ref{fig:cube-coordinate-recovery} shows the multi-factor quotient-coordinate reconstruction from the equivalence classes of \(E_1,\dots,E_k\).

\begin{lemma}\label{lem:k-relations-to-product}
Let $(X,d)$ be a metric space, let $k\geq 2$, and let $E_1,\dots,E_k$ be equivalence relations on $X$. Assume that:
\begin{enumerate}
\item
the quotient distance $\bar d_{E_j}$ is a metric on $X/E_j$ for each $j$;

\item
for every choice of classes $C_j\in X/E_j$, where $1\leq j\leq k$, the intersection $\bigcap_{j=1}^k C_j$ consists of exactly one point;

\item
the family $(E_1,\dots,E_k)$ satisfies the $k$-factor distance decomposition identity.
\end{enumerate}
Equip $\prod_{j=1}^k X/E_j$ with the metric
\begin{equation}\label{eq:k-quotient-product-metric}
d_\times((C_1,\dots,C_k),(D_1,\dots,D_k))
=
\max_{1\leq j\leq k}
\bar d_{E_j}(C_j,D_j).
\end{equation}
Define $\Psi:X\to\prod_{j=1}^k X/E_j$ by $\Psi(x):=([x]_{E_1},\dots,[x]_{E_k})$. Then $\Psi$ is a bijective isometry.
\end{lemma}

\begin{proof}
Lemma~\ref{lem:singleton-reconstruction} shows that $\Psi$ is bijective.

For all $x,y\in X$, Equations~\eqref{eq:k-quotient-product-metric} and \eqref{eq:k-factor-distance-decomposition-identity} give
$d_\times(\Psi(x),\Psi(y))
=
\max_{1\leq j\leq k}
\bar d_{E_j}([x]_{E_j},[y]_{E_j})
=
d(x,y)$.
Therefore $\Psi$ is an isometry.
\end{proof}

\begin{theorem}\label{thm:k-fold-product-classification}
Let $(X,d)$ be a metric space, and let $k\geq 2$. The following are equivalent.
\begin{enumerate}
\item\label{item:k-factor-product}
There exist metric spaces $(P_j,d_j)$ for $1\leq j\leq k$ such that $(X,d)$ is isometric to $P_1\times\cdots\times P_k$ with the max product metric from Equation~\eqref{eq:k-product-metric}.

\item\label{item:k-factor-relations}
There exist equivalence relations $E_1,\dots,E_k$ on $X$ such that:
\begin{enumerate}
\item\label{item:k-factor-metrics}
the quotient distance $\bar d_{E_j}$ is a metric on $X/E_j$ for each $j$;

\item\label{item:k-factor-intersection}
for every choice of classes $C_j\in X/E_j$, where $1\leq j\leq k$, the intersection $\bigcap_{j=1}^k C_j$ consists of exactly one point;

\item\label{item:k-factor-decomposition}
the family $(E_1,\dots,E_k)$ satisfies the $k$-factor distance decomposition identity.
\end{enumerate}
\end{enumerate}
\end{theorem}

\begin{proof}
Assume condition~\ref{item:k-factor-product}. Let $Y:=P_1\times\cdots\times P_k$ carry the metric from Equation~\eqref{eq:k-product-metric}, and let $\Phi:X\to Y$ be a bijective isometry. For each $j\in\{1,\dots,k\}$, let $F_j$ denote the equivalence relation from Lemma~\ref{lem:k-product-intersection}. Define an equivalence relation $E_j$ on $X$ by declaring $x E_j y$ exactly when $\Phi(x) F_j \Phi(y)$.

For each $j\in\{1,\dots,k\}$, define $\Theta_j:X/E_j\to Y/F_j$ by $\Theta_j([x]_{E_j}):=[\Phi(x)]_{F_j}$. Lemma~\ref{lem:quotient-transfer} gives a bijective isometry
\[
\Theta_j:(X/E_j,\bar d_{E_j})
\to
(Y/F_j,\bar d_{F_j}).
\]
Lemma~\ref{lem:k-product-structure} therefore gives condition~\ref{item:k-factor-metrics}.

Let $C_j\in X/E_j$ for $1\leq j\leq k$. The definition of $E_j$ gives $\Phi\left(\bigcap_{j=1}^k C_j\right)=\bigcap_{j=1}^k \Theta_j(C_j)$. Lemma~\ref{lem:k-product-intersection} shows that the right side consists of exactly one point. Since $\Phi$ is bijective, the same holds for $\bigcap_{j=1}^k C_j$. Thus condition~\ref{item:k-factor-intersection} holds.

Equation~\eqref{eq:k-product-quotient-reconstruction} and Lemma~\ref{lem:quotient-transfer} give
$d(x,y)
=
\max_{1\leq j\leq k}
\bar d_{E_j}([x]_{E_j},[y]_{E_j})$
for all $x,y\in X$. Therefore condition~\ref{item:k-factor-decomposition} holds.

Assume condition~\ref{item:k-factor-relations}. Define $\Psi:X\to\prod_{j=1}^k X/E_j$ by $\Psi(x):=([x]_{E_1},\dots,[x]_{E_k})$. Lemma~\ref{lem:k-relations-to-product} shows that $\Psi$ is a bijective isometry from $(X,d)$ onto $\prod_{j=1}^k X/E_j$ equipped with the metric from Equation~\eqref{eq:k-quotient-product-metric}. Define $(P_j,d_j):=(X/E_j,\bar d_{E_j})$ for each $j$. Then Equation~\eqref{eq:k-quotient-product-metric} is the max product metric from Equation~\eqref{eq:k-product-metric}. Therefore condition~\ref{item:k-factor-product} holds.
\end{proof}

\begin{corollary}\label{cor:k-fold-power-classification}
Let $(X,d)$ be a metric space, and let $k\geq 2$. The following are equivalent.
\begin{enumerate}
\item\label{item:k-power-product}
There exists a metric space $(P,d_P)$ such that $(X,d)$ is isometric to $P^k$, where
\begin{equation}\label{eq:k-power-metric}
d_{P^k}(p,q)
=
\max_{1\leq j\leq k} d_P(p_j,q_j)
\end{equation}
for all $p=(p_1,\dots,p_k)$ and $q=(q_1,\dots,q_k)$.

\item\label{item:k-power-relations}
There exist equivalence relations $E_1,\dots,E_k$ on $X$ such that:
\begin{enumerate}
\item\label{item:k-power-metrics}
the quotient distance $\bar d_{E_j}$ is a metric on $X/E_j$ for each $j$;

\item\label{item:k-power-intersection}
for every choice of classes $C_j\in X/E_j$, where $1\leq j\leq k$, the intersection $\bigcap_{j=1}^k C_j$ consists of exactly one point;

\item\label{item:k-power-decomposition}
the family $(E_1,\dots,E_k)$ satisfies the $k$-factor distance decomposition identity;

\item\label{item:k-power-common-model}
there exists a metric space $(P,d_P)$ such that each $(X/E_j,\bar d_{E_j})$ is isometric to $(P,d_P)$.
\end{enumerate}
\end{enumerate}
\end{corollary}

\begin{proof}
Assume condition~\ref{item:k-power-product}. The forward implication in Theorem~\ref{thm:k-fold-product-classification}, applied with $P_j=P$ for every $j$, gives conditions~\ref{item:k-power-metrics}, \ref{item:k-power-intersection}, and \ref{item:k-power-decomposition}. Lemma~\ref{lem:k-product-structure} shows that each $(X/E_j,\bar d_{E_j})$ is isometric to $(P,d_P)$. Hence condition~\ref{item:k-power-common-model} holds.

Assume condition~\ref{item:k-power-relations}. By condition~\ref{item:k-power-common-model}, choose a metric space $(P,d_P)$ such that each $(X/E_j,\bar d_{E_j})$ is isometric to $(P,d_P)$. Theorem~\ref{thm:k-fold-product-classification} gives a bijective isometry from $(X,d)$ onto $\prod_{j=1}^k (X/E_j,\bar d_{E_j})$ equipped with the metric from Equation~\eqref{eq:k-quotient-product-metric}. For each $j\in\{1,\dots,k\}$, choose a bijective isometry $\theta_j:(X/E_j,\bar d_{E_j})\to(P,d_P)$. Define $\Theta(C_1,\dots,C_k):=(\theta_1(C_1),\dots,\theta_k(C_k))$.

For all $C=(C_1,\dots,C_k)$ and $D=(D_1,\dots,D_k)$ in $\prod_{j=1}^k X/E_j$, we have
$d_{P^k}(\Theta(C),\Theta(D))
=
\max_{1\leq j\leq k} \bar d_{E_j}(C_j,D_j)
=
d_\times(C,D)$.
Since each $\theta_j$ is bijective, the map $\Theta$ is a bijective isometry. Hence condition~\ref{item:k-power-product} holds.
\end{proof}

\begin{corollary}\label{cor:k-fold-self-power-classification}
Let $(X,d)$ be a metric space, and let $k\geq 2$. The following are equivalent.
\begin{enumerate}
\item\label{item:k-self-power-product}
The metric space $(X,d)$ is isometric to $X^k$, where
\begin{equation}\label{eq:k-self-power-metric}
d_{X^k}(x,y)
=
\max_{1\leq j\leq k} d(x_j,y_j)
\end{equation}
for all $x=(x_1,\dots,x_k)$ and $y=(y_1,\dots,y_k)$.

\item\label{item:k-self-power-relations}
There exist equivalence relations $E_1,\dots,E_k$ on $X$ such that:
\begin{enumerate}
\item\label{item:k-self-power-metrics}
the quotient distance $\bar d_{E_j}$ is a metric on $X/E_j$ for each $j$;

\item\label{item:k-self-power-intersection}
for every choice of classes $C_j\in X/E_j$, where $1\leq j\leq k$, the intersection $\bigcap_{j=1}^k C_j$ consists of exactly one point;

\item\label{item:k-self-power-decomposition}
the family $(E_1,\dots,E_k)$ satisfies the $k$-factor distance decomposition identity;

\item\label{item:k-self-power-isometric}
for each $j$, the metric space $(X/E_j,\bar d_{E_j})$ is isometric to $(X,d)$.
\end{enumerate}
\end{enumerate}
\end{corollary}

\begin{proof}
Assume condition~\ref{item:k-self-power-product}. Apply the forward implication in Corollary~\ref{cor:k-fold-power-classification} with $(P,d_P)=(X,d)$. This gives conditions~\ref{item:k-self-power-metrics}, \ref{item:k-self-power-intersection}, and \ref{item:k-self-power-decomposition}. Condition~\ref{item:k-self-power-isometric} gives condition~\ref{item:k-power-common-model} with common model $(X,d)$. Hence condition~\ref{item:k-self-power-relations} holds.

Assume condition~\ref{item:k-self-power-relations}. Conditions~\ref{item:k-self-power-metrics}, \ref{item:k-self-power-intersection}, and \ref{item:k-self-power-decomposition} give conditions~\ref{item:k-power-metrics}, \ref{item:k-power-intersection}, and \ref{item:k-power-decomposition}. Condition~\ref{item:k-self-power-isometric} gives condition~\ref{item:k-power-common-model} with common model $(X,d)$. Therefore Corollary~\ref{cor:k-fold-power-classification} gives condition~\ref{item:k-self-power-product}.
\end{proof}

\subsection{The \texorpdfstring{$\ell^p$}{ell-p} Product Case}\label{sec:lp-product-case}

Let $1\leq p<\infty$. We now consider the $\ell^p$ product metric.

\begin{lemma}\label{lem:kp-product-structure}
Let $k\geq 2$, let $(P_j,d_j)$ be metric spaces for $1\leq j\leq k$, and let $Y:=P_1\times\cdots\times P_k$. Equip $Y$ with the metric
\begin{equation}\label{eq:kp-product-metric}
d_{Y,p}(u,v)
=
\left(
\sum_{j=1}^k d_j(u_j,v_j)^p
\right)^{1/p},
\end{equation}
where $u=(u_1,\dots,u_k)$ and $v=(v_1,\dots,v_k)$. For each $j\in\{1,\dots,k\}$, let $F_j$ denote the equivalence relation from Lemma~\ref{lem:k-product-intersection}. Then, for each $j\in\{1,\dots,k\}$, the quotient distance $\bar d_{F_j}$ is a metric on $Y/F_j$, the map $[u]_{F_j}\mapsto u_j$ defines a bijective isometry from $Y/F_j$ onto $(P_j,d_j)$, and
\begin{equation}\label{eq:kp-product-quotient-reconstruction}
d_{Y,p}(u,v)
=
\left(
\sum_{j=1}^k
\bar d_{F_j}([u]_{F_j},[v]_{F_j})^p
\right)^{1/p}.
\end{equation}
\end{lemma}

\begin{proof}
Fix $j\in\{1,\dots,k\}$. The map $[u]_{F_j}\mapsto u_j$ is well-defined and injective by the definition of $F_j$. Since $P_j$ is nonempty, the map is surjective.

If $u' F_j u$ and $v' F_j v$, then Equation~\eqref{eq:kp-product-metric} gives $d_{Y,p}(u',v')\geq d_j(u_j,v_j)$. Hence $\bar d_{F_j}([u]_{F_j},[v]_{F_j})\geq d_j(u_j,v_j)$.

Conversely, choose representatives of $[u]_{F_j}$ and $[v]_{F_j}$ that agree in every coordinate except the $j$-th coordinate. For these representatives, Equation~\eqref{eq:kp-product-metric} gives $d_{Y,p}(u',v')=d_j(u_j,v_j)$. Therefore $\bar d_{F_j}([u]_{F_j},[v]_{F_j})=d_j(u_j,v_j)$.

Hence $\bar d_{F_j}$ is a metric, and the map $[u]_{F_j}\mapsto u_j$ is a bijective isometry. Equation~\eqref{eq:kp-product-quotient-reconstruction} follows from this quotient-distance identity and Equation~\eqref{eq:kp-product-metric}.
\end{proof}

\begin{definition}\label{def:p-distance-decomposition}
Let $(X,d)$ be a metric space, let $k\geq 2$, and let $E_1,\dots,E_k$ be equivalence relations on $X$. We say that $(E_1,\dots,E_k)$ satisfies the $p$-distance decomposition identity if, for all $x,y\in X$,
\begin{equation}\label{eq:p-distance-decomposition-identity}
d(x,y)
=
\left(
\sum_{j=1}^k
\bar d_{E_j}([x]_{E_j},[y]_{E_j})^p
\right)^{1/p}.
\end{equation}
\end{definition}

\begin{figure}[ht]
\centering
\includegraphics[width=0.86\textwidth]{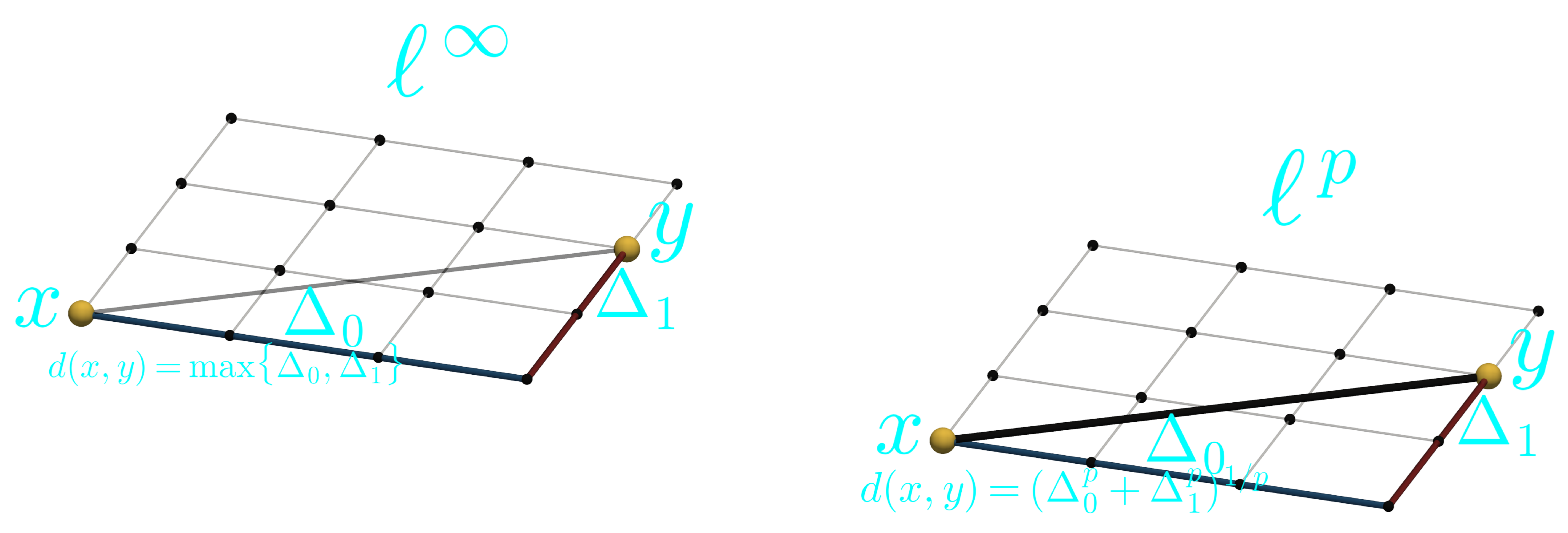}
\caption{
Comparison of the max product metric and the \(\ell^p\) product metric on the same recovered quotient coordinates. Both settings recover the coordinate distances from quotient spaces, while the max product metric combines these distances by a maximum and the \(\ell^p\) product metric combines them through the \(\ell^p\) expression.
}
\label{fig:ell-infty-ell-p-comparison}
\end{figure}

Figure~\ref{fig:ell-infty-ell-p-comparison} shows the difference between the max and \(\ell^p\) product metrics in how they combine the recovered quotient distances.

\begin{lemma}\label{lem:p-relations-to-product}
Let $(X,d)$ be a metric space, let $k\geq 2$, and let $E_1,\dots,E_k$ be equivalence relations on $X$. Assume that:
\begin{enumerate}
\item
the quotient distance $\bar d_{E_j}$ is a metric on $X/E_j$ for each $j$;

\item
for every choice of classes $C_j\in X/E_j$, where $1\leq j\leq k$, the intersection $\bigcap_{j=1}^k C_j$ consists of exactly one point;

\item
the family $(E_1,\dots,E_k)$ satisfies the $p$-distance decomposition identity.
\end{enumerate}
Equip $\prod_{j=1}^k X/E_j$ with the metric
\begin{equation}\label{eq:p-quotient-product-metric}
d_{\times,p}((C_1,\dots,C_k),(D_1,\dots,D_k))
=
\left(
\sum_{j=1}^k
\bar d_{E_j}(C_j,D_j)^p
\right)^{1/p}.
\end{equation}
Define $\Psi:X\to\prod_{j=1}^k X/E_j$ by $\Psi(x):=([x]_{E_1},\dots,[x]_{E_k})$. Then $\Psi$ is a bijective isometry.
\end{lemma}

\begin{proof}
Lemma~\ref{lem:singleton-reconstruction} shows that $\Psi$ is bijective. Equations~\eqref{eq:p-quotient-product-metric} and~\eqref{eq:p-distance-decomposition-identity} give $d_{\times,p}(\Psi(x),\Psi(y))=d(x,y)$ for all $x,y\in X$. Hence $\Psi$ is a bijective isometry.
\end{proof}

\begin{theorem}\label{thm:p-product-classification}
Let $(X,d)$ be a metric space, let $k\geq 2$, and let $1\leq p<\infty$. The following are equivalent.
\begin{enumerate}
\item\label{item:p-factor-product}
There exist metric spaces $(P_j,d_j)$ for $1\leq j\leq k$ such that $(X,d)$ is isometric to $P_1\times\cdots\times P_k$ with the metric from Equation~\eqref{eq:kp-product-metric}.

\item\label{item:p-factor-relations}
There exist equivalence relations $E_1,\dots,E_k$ on $X$ such that:
\begin{enumerate}
\item\label{item:p-factor-metrics}
the quotient distance $\bar d_{E_j}$ is a metric on $X/E_j$ for each $j$;

\item\label{item:p-factor-intersection}
for every choice of classes $C_j\in X/E_j$, where $1\leq j\leq k$, the intersection $\bigcap_{j=1}^k C_j$ consists of exactly one point;

\item\label{item:p-factor-decomposition}
the family $(E_1,\dots,E_k)$ satisfies the $p$-distance decomposition identity.
\end{enumerate}
\end{enumerate}
\end{theorem}

\begin{proof}
Assume condition~\ref{item:p-factor-product}. Let $Y:=P_1\times\cdots\times P_k$ carry the metric from Equation~\eqref{eq:kp-product-metric}, and let $\Phi:X\to Y$ be a bijective isometry. For each $j\in\{1,\dots,k\}$, let $F_j$ denote the equivalence relation from Lemma~\ref{lem:k-product-intersection}. Define an equivalence relation $E_j$ on $X$ by declaring $x E_j y$ exactly when $\Phi(x) F_j \Phi(y)$.

For each $j\in\{1,\dots,k\}$, define $\Theta_j:X/E_j\to Y/F_j$ by $\Theta_j([x]_{E_j}):=[\Phi(x)]_{F_j}$. Lemma~\ref{lem:quotient-transfer} shows that $\Theta_j$ is a bijective isometry from $(X/E_j,\bar d_{E_j})$ onto $(Y/F_j,\bar d_{F_j})$. Lemma~\ref{lem:kp-product-structure} therefore gives condition~\ref{item:p-factor-metrics}.

Let $C_j\in X/E_j$ for $1\leq j\leq k$. The definition of $E_j$ gives $\Phi\left(\bigcap_{j=1}^k C_j\right)=\bigcap_{j=1}^k \Theta_j(C_j)$. Lemma~\ref{lem:k-product-intersection} shows that the right side consists of exactly one point. Since $\Phi$ is bijective, the same holds for $\bigcap_{j=1}^k C_j$. Hence condition~\ref{item:p-factor-intersection} holds.

Equation~\eqref{eq:kp-product-quotient-reconstruction} and Lemma~\ref{lem:quotient-transfer} give $d(x,y)=\left(\sum_{j=1}^k \bar d_{E_j}([x]_{E_j},[y]_{E_j})^p\right)^{1/p}$. Hence condition~\ref{item:p-factor-decomposition} holds.

Assume condition~\ref{item:p-factor-relations}. Define $\Psi:X\to\prod_{j=1}^k X/E_j$ by $\Psi(x):=([x]_{E_1},\dots,[x]_{E_k})$. Lemma~\ref{lem:p-relations-to-product} shows that $\Psi$ is a bijective isometry from $(X,d)$ onto $\prod_{j=1}^k X/E_j$ equipped with the metric from Equation~\eqref{eq:p-quotient-product-metric}. Define $(P_j,d_j):=(X/E_j,\bar d_{E_j})$ for each $j$. The metric from Equation~\eqref{eq:p-quotient-product-metric} coincides with the metric from Equation~\eqref{eq:kp-product-metric} for this choice of factors. Therefore condition~\ref{item:p-factor-product} holds.
\end{proof}

\begin{corollary}\label{cor:p-power-classification}
Let $(X,d)$ be a metric space, let $k\geq 2$, and let $1\leq p<\infty$. The following are equivalent.
\begin{enumerate}
\item\label{item:p-power-product}
There exists a metric space $(P,d_P)$ such that $(X,d)$ is isometric to $P^k$, where
\begin{equation}\label{eq:p-power-metric}
d_{P^k,p}(u,v)
=
\left(
\sum_{j=1}^k
d_P(u_j,v_j)^p
\right)^{1/p}
\end{equation}
for all $u=(u_1,\dots,u_k)$ and $v=(v_1,\dots,v_k)$.

\item\label{item:p-power-relations}
There exist equivalence relations $E_1,\dots,E_k$ on $X$ such that:
\begin{enumerate}
\item\label{item:p-power-metrics}
the quotient distance $\bar d_{E_j}$ is a metric on $X/E_j$ for each $j$;

\item\label{item:p-power-intersection}
for every choice of classes $C_j\in X/E_j$, where $1\leq j\leq k$, the intersection $\bigcap_{j=1}^k C_j$ consists of exactly one point;

\item\label{item:p-power-decomposition}
the family $(E_1,\dots,E_k)$ satisfies the $p$-distance decomposition identity;

\item\label{item:p-power-common-model}
there exists a metric space $(P,d_P)$ such that each metric space $(X/E_j,\bar d_{E_j})$ is isometric to $(P,d_P)$.
\end{enumerate}
\end{enumerate}
\end{corollary}

\begin{proof}
Assume condition~\ref{item:p-power-product}. Apply the forward implication in Theorem~\ref{thm:p-product-classification} with $P_j=P$ for every $j$. Lemma~\ref{lem:kp-product-structure} shows that each $(X/E_j,\bar d_{E_j})$ is isometric to $(P,d_P)$. Hence condition~\ref{item:p-power-common-model} holds.

Assume condition~\ref{item:p-power-relations}. By condition~\ref{item:p-power-common-model}, choose a metric space $(P,d_P)$ such that each $(X/E_j,\bar d_{E_j})$ is isometric to $(P,d_P)$. Theorem~\ref{thm:p-product-classification} gives a bijective isometry from $(X,d)$ onto $\prod_{j=1}^k (X/E_j,\bar d_{E_j})$ equipped with the metric from Equation~\eqref{eq:p-quotient-product-metric}. For each $j\in\{1,\dots,k\}$, choose a bijective isometry $\theta_j:(X/E_j,\bar d_{E_j})\to(P,d_P)$. Define $\Theta(C_1,\dots,C_k):=(\theta_1(C_1),\dots,\theta_k(C_k))$.

For all $C,D\in\prod_{j=1}^k X/E_j$, we have $d_{P^k,p}(\Theta(C),\Theta(D))=d_{\times,p}(C,D)$. Therefore $\Theta$ is a bijective isometry, so condition~\ref{item:p-power-product} holds.
\end{proof}

\begin{corollary}\label{cor:p-self-power-classification}
Let $(X,d)$ be a metric space, let $k\geq 2$, and let $1\leq p<\infty$. The following are equivalent.
\begin{enumerate}
\item\label{item:p-self-power-product}
The metric space $(X,d)$ is isometric to $X^k$, where
\begin{equation}\label{eq:p-self-power-metric}
d_{X^k,p}(u,v)
=
\left(
\sum_{j=1}^k
d(u_j,v_j)^p
\right)^{1/p}
\end{equation}
for all $u=(u_1,\dots,u_k)$ and $v=(v_1,\dots,v_k)$.

\item\label{item:p-self-power-relations}
There exist equivalence relations $E_1,\dots,E_k$ on $X$ such that:
\begin{enumerate}
\item\label{item:p-self-power-metrics}
the quotient distance $\bar d_{E_j}$ is a metric on $X/E_j$ for each $j$;

\item\label{item:p-self-power-intersection}
for every choice of classes $C_j\in X/E_j$, where $1\leq j\leq k$, the intersection $\bigcap_{j=1}^k C_j$ consists of exactly one point;

\item\label{item:p-self-power-decomposition}
the family $(E_1,\dots,E_k)$ satisfies the $p$-distance decomposition identity;

\item\label{item:p-self-power-isometric}
for each $j$, the metric space $(X/E_j,\bar d_{E_j})$ is isometric to $(X,d)$.
\end{enumerate}
\end{enumerate}
\end{corollary}

\begin{proof}
Assume condition~\ref{item:p-self-power-product}. Apply the forward implication in Corollary~\ref{cor:p-power-classification} with $(P,d_P)=(X,d)$. Condition~\ref{item:p-self-power-isometric} gives condition~\ref{item:p-power-common-model} with common model $(X,d)$. Hence condition~\ref{item:p-self-power-relations} holds.

Assume condition~\ref{item:p-self-power-relations}. The hypotheses give condition~\ref{item:p-power-relations} with common model $(X,d)$. Therefore Corollary~\ref{cor:p-power-classification} gives condition~\ref{item:p-self-power-product}.
\end{proof}

\section{The Classification of Square Structures}\label{sec:Classification-Square-Structures}

A presentation \(P^k\cong X\) contains coordinate data not captured by the root isometry type alone. We encode that coordinate data by quotient-coordinate structures on \(X\). We classify square, ordered \(k\)-power, and ordered \(\ell^p\) \(k\)-power presentations through these structures.

After classifying presentations, we ask when \(X\) has a root and when that root has a unique isometry type. We use prime decompositions to turn this root question into a multiplicity problem. For self-square spaces, we study the metric data obtained by iterating a chosen coordinate splitting.

The categorical classifications form the main results of this section. Theorems~\ref{thm:categorical-classification}, \ref{thm:k-categorical-classification}, and \ref{thm:p-categorical-classification} prove equivalences between the corresponding presentation categories and intrinsic structure categories. Each presentation determines equivalence relations, quotient metrics, singleton-intersection data, distance identities, and quotient-factor isometries. Conversely, each such structure reconstructs the corresponding presentation. The same theorems, together with Lemma~\ref{lem:presentation-thinness}, show that the presentation categories are thin and that every presentation has trivial automorphism group.

The root-classification results first determine the root isometry type and then classify the ordered structures realizing it. We decide \(k\)-th root existence and the root isometry type from the prime multiplicities. After fixing a root presentation, right cosets of the subgroup induced by coordinatewise root isometries classify the ordered structures.

We also classify self-square structures through repeated coordinate splitting. From a chosen self-square isometry, we construct branch-indexed pseudometrics whose supremum recovers \(d\). In the finite-spectrum setting with cardinal multiplicity uniqueness, \(X\cong X\times_\infty X\) holds exactly when every prime multiplicity is infinite. The proof uses the cardinal identity \(2\kappa=\kappa\).

Previous work on metric product uniqueness, cancellation, and factorization studies product factors or decompositions after a product setting has been specified \cite{moszynska1992uniqueness,herburt1991metric,herburt1994there,tardif1992prefibers}. Our comparison point is the ordered presentation itself, not only the factor isometry type. We use the quotient-coordinate results from the preceding product-recognition section to replace unavailable coordinate maps in the finite-coordinate case.

For finite-distance \(\ell^\infty\)-spaces, we use threshold relations to translate metric products into relational direct products. We apply McKenzie's strict refinement theorem to obtain uniqueness of prime multiplicities under the stated hypotheses \cite[p.~98, Theorem~9.2]{mckenzie1971cardinal}.

The finite quotient-coordinate classification treats products with a fixed number of factors, while our self-square analysis must control indefinitely repeated coordinate splitting. We iterate a self-square isometry to obtain coordinate systems indexed by finite binary words. The infinite branches replace the finite list of quotient coordinates used in the product-recognition setting.

Square and power questions depend not only on the factor but also on how its copies appear as coordinates of \(X\). We can therefore compare square and power presentations intrinsically, rather than only comparing their factors.

A root and a presentation contain different information. The root determines the isometry type of the repeated factor, while the presentation determines how copies of that factor appear as coordinates of \(X\). A uniqueness theorem for roots does not by itself classify ordered presentations.

The self-square case requires additional structure because iterating a chosen splitting yields infinitely many compatible coordinate systems. We use the distinction between finite and infinite multiplicities to explain when repeated factors can duplicate without changing the isometry type. Branch pseudometrics let us distinguish multiplicity-based self-similarity from splitting-dependent metric data.


\subsection{Classification of Power Presentations}

Throughout this subsection, every isometry is bijective and every metric space is nonempty. We let \(\mathbf{Met}_{\mathrm{iso}}\) denote the category of metric spaces and isometries. We fix a metric space \((X,d)\). For every metric space \(P\), every integer \(k\geq 2\), and every \(1\leq j\leq k\), we write \(\pi_j:P^k\to P\) for the \(j\)-th coordinate projection. For every map \(f:P\to Q\) and every integer \(k\geq 2\), we write \(f^{\times k}:P^k\to Q^k\) for the coordinatewise product map.

\begin{lemma}\label{lem:presentation-thinness}
Let \(k\geq 2\) be an integer. Let \(\mathbf{Pres}\) be a category whose objects are pairs \((P,\phi)\), where \(P\) is a metric space, \(P^k\) is a metric space, and \(\phi:P^k\to(X,d)\) is a bijective isometry. Assume that every morphism \(f:(P,\phi)\to(Q,\psi)\) in \(\mathbf{Pres}\) is a map \(f:P\to Q\) such that \(\psi\circ f^{\times k}=\phi\). Then \(\mathbf{Pres}\) is thin, and every object of \(\mathbf{Pres}\) has trivial automorphism group.

\[
\begin{tikzcd}[column sep=large,row sep=large]
P^k
\arrow[r,"f^{\times k}"]
\arrow[dr,"\phi"']
&
Q^k
\arrow[d,"\psi"]
\\
&
X
\end{tikzcd}
\]
\end{lemma}

\begin{proof}
We consider morphisms \(f,g:(P,\phi)\to(Q,\psi)\) in \(\mathbf{Pres}\). The morphism condition gives \(\psi\circ f^{\times k}=\phi\) and \(\psi\circ g^{\times k}=\phi\). Hence \(\psi\circ f^{\times k}=\psi\circ g^{\times k}\). Since \(\psi\) is injective, \(f^{\times k}=g^{\times k}\). We fix \(p\in P\). Evaluating \(f^{\times k}=g^{\times k}\) at \((p,\dots,p)\in P^k\) shows that \((f(p),\dots,f(p))=(g(p),\dots,g(p))\). Hence \(f(p)=g(p)\). Thus any two morphisms in \(\mathbf{Pres}\) with the same source and target agree. Therefore \(\mathbf{Pres}\) is thin. Since every object has an identity endomorphism and a thin category has at most one endomorphism of each object, every automorphism group is trivial.
\end{proof}

We now treat square presentations, ordered presentations with the max product metric, and ordered presentations with the \(\ell^p\) product metric.

\subsubsection{Square Presentations and Square Structures}

We identify \(P\times P\) with \(P^2\) when using the coordinate projections \(\pi_1,\pi_2\) and the product maps \(f^{\times 2}\). We define the square functor \(\Delta_2:\mathbf{Met}_{\mathrm{iso}}\to\mathbf{Met}_{\mathrm{iso}}\) by \(\Delta_2(P,d_P):=(P\times P,d_{P\times P})\), where Equation~\eqref{eq:max-product-metric} defines \(d_{P\times P}\), and by \(\Delta_2(f):=f^{\times 2}\) on isometries. Equation~\eqref{eq:max-product-metric} shows that \(\Delta_2(f)\) is a bijective isometry, and the identity and composition laws are immediate. We write \(\Delta_2(P)\) when the metric on \(P\) is clear.

\begin{definition}\label{def:sqpres}
We define the category \(\mathbf{SqPres}(X)\) as follows. An object is a pair \((P,\phi)\), where \((P,d_P)\) is a metric space and \(\phi:\Delta_2(P)\to(X,d)\) is an isometry. A morphism \(f:(P,\phi)\to(Q,\psi)\) is an isometry \(f:(P,d_P)\to(Q,d_Q)\) such that
\begin{equation}\label{eq:sqpres-morphism-condition}
\psi\circ f^{\times 2}=\phi.
\end{equation}
\end{definition}

\begin{definition}\label{def:sqstr}
A square structure on \((X,d)\) is a triple \(\mathfrak s=(E_1,E_2,\tau)\), where \(E_1\) and \(E_2\) are equivalence relations on \(X\) and \(\tau:X/E_2\to X/E_1\) is a map such that:
\begin{enumerate}
\item\label{item:sqstr-quotient-metrics}
for each \(i\in\{1,2\}\), the quotient distance \(\bar d_{E_i}\) is a metric on \(X/E_i\);

\item\label{item:sqstr-intersections}
for every \(C_1\in X/E_1\) and every \(C_2\in X/E_2\), the intersection \(C_1\cap C_2\) consists of exactly one point;

\item\label{item:sqstr-decomposition}
the pair \((E_1,E_2)\) satisfies the distance decomposition identity of Definition~\ref{def:distance-decomposition}, with \(E_1,E_2\) in place of \(E_0,E_1\);

\item\label{item:sqstr-tau-isometry}
the map \(\tau:(X/E_2,\bar d_{E_2})\to(X/E_1,\bar d_{E_1})\) is an isometry.
\end{enumerate}
\end{definition}

We let \(\mathbf{SqStr}(X)\) denote the discrete category of square structures on \((X,d)\).

Let \((P,\phi)\) be an object of \(\mathbf{SqPres}(X)\). For each \(i\in\{1,2\}\), we define \(q_i^\phi:=\pi_i\circ\phi^{-1}:X\to P\). For \(i\in\{1,2\}\), define \(E_i^\phi\) by declaring that \(xE_i^\phi y\) if and only if \(q_i^\phi(x)=q_i^\phi(y)\).

\begin{lemma}\label{lem:quotient-identification-cat}
For each \(i\in\{1,2\}\), the map \(\bar q_i^\phi:X/E_i^\phi\to P\), defined by \(\bar q_i^\phi([x]_{E_i^\phi})=q_i^\phi(x)\), is a bijective isometry.
\end{lemma}

\begin{proof}
The definition of \(E_i^\phi\) shows that \(\bar q_i^\phi\) is well-defined and injective.

We fix \(r\in P\). For every \(p\in P\), we have \(q_1^\phi(\phi(p,r))=p\) and \(q_2^\phi(\phi(r,p))=p\). Hence \(q_i^\phi\) and \(\bar q_i^\phi\) are surjective for \(i\in\{1,2\}\).

Let \(F_1\) and \(F_2\) denote the first- and second-coordinate equivalence relations on \(P\times P\). For each \(i\in\{1,2\}\), \(xE_i^\phi y\) if and only if \(\phi^{-1}(x)F_i\phi^{-1}(y)\). Hence \(\phi^{-1}\) sends \(E_i^\phi\)-classes bijectively onto \(F_i\)-classes. Equation~\eqref{eq:quotient-distance} and the fact that \(\phi^{-1}\) is an isometry give
\begin{equation}\label{eq:quotient-distance-transfer-presentation}
\begin{aligned}
\bar d_{E_i^\phi}([x]_{E_i^\phi},[y]_{E_i^\phi})
&=
\inf\{d(x',y'):x'E_i^\phi x,\ y'E_i^\phi y\} \\
&=
\inf\{
d_{P\times P}(\phi^{-1}(x'),\phi^{-1}(y')):
x'E_i^\phi x,\ y'E_i^\phi y
\} \\
&=
\inf\{
d_{P\times P}(u,v):
uF_i\phi^{-1}(x),\ vF_i\phi^{-1}(y)
\} \\
&=
\bar d_{F_i}([\phi^{-1}(x)]_{F_i},[\phi^{-1}(y)]_{F_i}).
\end{aligned}
\end{equation}
Equation~\eqref{eq:quotient-distance-transfer-presentation} and Lemma~\ref{lem:product-structure}, applied to the relabeled coordinate pair \((F_1,F_2)\), give
\begin{equation}\label{eq:sqpres-quotient-distance}
\bar d_{E_i^\phi}([x]_{E_i^\phi},[y]_{E_i^\phi})
=
d_P(q_i^\phi(x),q_i^\phi(y)).
\end{equation}
Hence \(\bar q_i^\phi\) is a bijective isometry.
\end{proof}

We define \(\tau^\phi:=(\bar q_1^\phi)^{-1}\circ\bar q_2^\phi\).

\begin{lemma}\label{lem:structure-from-presentation}
The triple \(\mathfrak s(\phi)=(E_1^\phi,E_2^\phi,\tau^\phi)\) is a square structure on \((X,d)\).
\end{lemma}

\begin{proof}
Lemma~\ref{lem:quotient-identification-cat} shows that each quotient distance \(\bar d_{E_i^\phi}\) is a metric, and the definition of \(\tau^\phi\) shows that \(\tau^\phi\) is an isometry.

Let \(F_1\) and \(F_2\) denote the first- and second-coordinate equivalence relations on \(P\times P\). For each \(i\in\{1,2\}\), \(xE_i^\phi y\) if and only if \(\phi^{-1}(x)F_i\phi^{-1}(y)\). Hence \(\phi^{-1}\) maps intersections of \(E_1^\phi\)- and \(E_2^\phi\)-classes bijectively onto intersections of \(F_1\)- and \(F_2\)-classes. Lemma~\ref{lem:product-intersection}, applied to the relabeled coordinate pair \((F_1,F_2)\), gives the required one-point intersection property.

For every \(x,y\in X\), Equation~\eqref{eq:max-product-metric} and the fact that \(\phi^{-1}\) is an isometry from \((X,d)\) onto \(\Delta_2(P)\) give
\[
d(x,y)
=
\max\{
d_P(q_1^\phi(x),q_1^\phi(y)),
d_P(q_2^\phi(x),q_2^\phi(y))
\}.
\]
Equation~\eqref{eq:sqpres-quotient-distance} identifies these terms with \(\bar d_{E_1^\phi}([x]_{E_1^\phi},[y]_{E_1^\phi})\) and \(\bar d_{E_2^\phi}([x]_{E_2^\phi},[y]_{E_2^\phi})\), respectively. Hence the distance decomposition identity follows. Therefore \(\mathfrak s(\phi)\) is a square structure.
\end{proof}

\begin{lemma}\label{lem:functoriality}
Let \(f:(P,\phi)\to(Q,\psi)\) be a morphism in \(\mathbf{SqPres}(X)\). Then \(\mathfrak s(\phi)=\mathfrak s(\psi)\). Moreover, for each \(i\in\{1,2\}\),
\begin{equation}\label{eq:quotient-map-functoriality-general}
\bar q_i^\psi=f\circ\bar q_i^\phi.
\end{equation}
\end{lemma}

\begin{proof}
Equation~\eqref{eq:sqpres-morphism-condition} gives \(\psi^{-1}(\phi(z))=f^{\times 2}(z)\) for \(z\in P\times P\). Taking \(z=\phi^{-1}(x)\) and applying \(\pi_i\) gives \(q_i^\psi(x)=f(q_i^\phi(x))\) for \(x\in X\) and \(i\in\{1,2\}\).

For \(x,y\in X\) and \(i\in\{1,2\}\), since \(q_i^\psi=f\circ q_i^\phi\) and \(f\) is injective, \(xE_i^\phi y\) if and only if \(xE_i^\psi y\). Hence \(E_i^\phi=E_i^\psi\). The identity \(q_i^\psi=f\circ q_i^\phi\) gives Equation~\eqref{eq:quotient-map-functoriality-general}.

Since \(\bar q_i^\psi=f\circ\bar q_i^\phi\) for \(i\in\{1,2\}\) and \(f\) is bijective, \(\tau^\psi=(\bar q_1^\psi)^{-1}\circ\bar q_2^\psi=(\bar q_1^\phi)^{-1}\circ f^{-1}\circ f\circ\bar q_2^\phi=\tau^\phi\). Therefore \(\mathfrak s(\phi)=\mathfrak s(\psi)\).
\end{proof}

We define \(\mathcal F:\mathbf{SqPres}(X)\to\mathbf{SqStr}(X)\) on objects by \(\mathcal F(P,\phi):=\mathfrak s(\phi)\). For every morphism \(f:(P,\phi)\to(Q,\psi)\), Lemma~\ref{lem:functoriality} gives \(\mathfrak s(\phi)=\mathfrak s(\psi)\), and we define \(\mathcal F(f)\) to be the identity morphism of this square structure.

Let \(\mathfrak s=(E_1,E_2,\tau)\) be a square structure on \((X,d)\). Set \(P_{\mathfrak s}:=X/E_1\), and equip \(P_{\mathfrak s}\) with the metric \(\bar d_{E_1}\). Define
\begin{equation}\label{eq:psi-from-structure}
\Psi_{\mathfrak s}(x)
=
([x]_{E_1},\tau([x]_{E_2})).
\end{equation}

\begin{lemma}\label{lem:presentation-from-structure}
The map \(\Psi_{\mathfrak s}\) is a bijective isometry from \((X,d)\) onto \(\Delta_2(P_{\mathfrak s})\). Consequently, if \(\phi_{\mathfrak s}:=\Psi_{\mathfrak s}^{-1}\), then \((P_{\mathfrak s},\phi_{\mathfrak s})\) is an object of \(\mathbf{SqPres}(X)\).
\end{lemma}

\begin{proof}
Conditions~\ref{item:sqstr-quotient-metrics}, \ref{item:sqstr-intersections}, and \ref{item:sqstr-decomposition} in Definition~\ref{def:sqstr} are the hypotheses of Lemma~\ref{lem:relations-to-product}, with \(E_1,E_2\) in place of \(E_0,E_1\). Hence the map \(x\mapsto([x]_{E_1},[x]_{E_2})\) is a bijective isometry from \((X,d)\) onto \((X/E_1)\times(X/E_2)\) with the metric from Equation~\eqref{eq:quotient-product-metric}.

Condition~\ref{item:sqstr-tau-isometry} states that \(\tau\) is an isometry, so \((C_1,C_2)\mapsto(C_1,\tau(C_2))\) is a bijective isometry from \((X/E_1)\times(X/E_2)\) onto \(P_{\mathfrak s}^2\). Equation~\eqref{eq:psi-from-structure} shows that the composition of these two isometries equals \(\Psi_{\mathfrak s}\). Hence \(\Psi_{\mathfrak s}\) is a bijective isometry.

Therefore \(\phi_{\mathfrak s}=\Psi_{\mathfrak s}^{-1}\) is an isometry \(\Delta_2(P_{\mathfrak s})\to(X,d)\). Thus \((P_{\mathfrak s},\phi_{\mathfrak s})\) is an object of \(\mathbf{SqPres}(X)\).
\end{proof}

We define \(\mathcal G:\mathbf{SqStr}(X)\to\mathbf{SqPres}(X)\) by \(\mathcal G(\mathfrak s):=(P_{\mathfrak s},\phi_{\mathfrak s})\) on objects and by sending each identity morphism to the identity morphism of its image.

\begin{lemma}\label{lem:roundtrip-structure-cat}
For every square structure \(\mathfrak s\) on \((X,d)\), we have \(\mathcal F\mathcal G(\mathfrak s)=\mathfrak s\).
\end{lemma}

\begin{proof}
We let \(\mathfrak s=(E_1,E_2,\tau)\). Since \(\phi_{\mathfrak s}^{-1}=\Psi_{\mathfrak s}\), Equation~\eqref{eq:psi-from-structure} gives \(q_1^{\phi_{\mathfrak s}}(x)=[x]_{E_1}\) and \(q_2^{\phi_{\mathfrak s}}(x)=\tau([x]_{E_2})\). For \(x,y\in X\), these identities give \(xE_1^{\phi_{\mathfrak s}}y\) if and only if \(xE_1y\) and, since \(\tau\) is injective, \(xE_2^{\phi_{\mathfrak s}}y\) if and only if \(xE_2y\). Thus \(E_1^{\phi_{\mathfrak s}}=E_1\) and \(E_2^{\phi_{\mathfrak s}}=E_2\). Under these identifications, \(\bar q_1^{\phi_{\mathfrak s}}=\mathrm{id}_{X/E_1}\) and \(\bar q_2^{\phi_{\mathfrak s}}=\tau\), so \(\tau^{\phi_{\mathfrak s}}=\tau\). Therefore \(\mathcal F\mathcal G(\mathfrak s)=\mathfrak s\).
\end{proof}

\begin{lemma}\label{lem:roundtrip-presentation-cat}
For every object \((P,\phi)\) of \(\mathbf{SqPres}(X)\), the map \(\eta_{(P,\phi)}:=\bar q_1^\phi:P_{\mathfrak s(\phi)}=X/E_1^\phi\to P\) is a morphism \(\mathcal G\mathcal F(P,\phi)\to(P,\phi)\) in \(\mathbf{SqPres}(X)\). Moreover,
\begin{equation}\label{eq:eta-compatibility}
\phi\circ\eta_{(P,\phi)}^{\times 2}
=
\phi_{\mathfrak s(\phi)}.
\end{equation}
\end{lemma}

\begin{proof}
Lemma~\ref{lem:quotient-identification-cat} shows that \(\eta_{(P,\phi)}=\bar q_1^\phi\) is an isometry. Since \(\phi_{\mathfrak s(\phi)}=\Psi_{\mathfrak s(\phi)}^{-1}\), it suffices to prove
\begin{equation}\label{eq:eta-inverse-compatibility}
\eta_{(P,\phi)}^{\times 2}\circ\Psi_{\mathfrak s(\phi)}
=
\phi^{-1}.
\end{equation}
For \(x\in X\), Equation~\eqref{eq:psi-from-structure} shows that the first coordinate of \(\eta_{(P,\phi)}^{\times 2}\Psi_{\mathfrak s(\phi)}(x)\) is \(\bar q_1^\phi([x]_{E_1^\phi})=q_1^\phi(x)\). Since \(\tau^\phi=(\bar q_1^\phi)^{-1}\circ\bar q_2^\phi\), we have \(\bar q_1^\phi(\tau^\phi([x]_{E_2^\phi}))=\bar q_2^\phi([x]_{E_2^\phi})\). Hence the second coordinate is \(q_2^\phi(x)\). Hence \(\eta_{(P,\phi)}^{\times 2}\Psi_{\mathfrak s(\phi)}(x)=\phi^{-1}(x)\), so Equation~\eqref{eq:eta-inverse-compatibility} holds. Therefore Equation~\eqref{eq:eta-compatibility} holds, so \(\eta_{(P,\phi)}\) is a morphism in \(\mathbf{SqPres}(X)\).
\end{proof}

\begin{lemma}\label{lem:eta-natural-cat}
The morphisms \(\eta_{(P,\phi)}:\mathcal G\mathcal F(P,\phi)\to(P,\phi)\) form a natural isomorphism \(\eta:\mathcal G\mathcal F\Rightarrow\mathrm{id}_{\mathbf{SqPres}(X)}\).
\end{lemma}

\begin{proof}
By Lemma~\ref{lem:roundtrip-presentation-cat}, each \(\eta_{(P,\phi)}\) is a morphism in \(\mathbf{SqPres}(X)\).

Since \(\eta_{(P,\phi)}\) is a bijective isometry, its inverse is a bijective isometry. Composing Equation~\eqref{eq:eta-compatibility} on the right with \((\eta_{(P,\phi)}^{-1})^{\times 2}\) gives
\begin{equation}\label{eq:eta-inverse-morphism-compatibility}
\phi_{\mathfrak s(\phi)}
\circ
(\eta_{(P,\phi)}^{-1})^{\times 2}
=
\phi.
\end{equation}
Equation~\eqref{eq:eta-inverse-morphism-compatibility} shows that \(\eta_{(P,\phi)}^{-1}\) is a morphism from \((P,\phi)\) to \(\mathcal G\mathcal F(P,\phi)\). Therefore \(\eta_{(P,\phi)}\) is an isomorphism in \(\mathbf{SqPres}(X)\).

We let \(f:(P,\phi)\to(Q,\psi)\) be a morphism in \(\mathbf{SqPres}(X)\). Lemma~\ref{lem:functoriality} gives \(\mathfrak s(\phi)=\mathfrak s(\psi)\), so \(\mathcal G\mathcal F(f)\) is the identity morphism of \(\mathcal G\mathcal F(P,\phi)=\mathcal G\mathcal F(Q,\psi)\). Equation~\eqref{eq:quotient-map-functoriality-general} with \(i=1\) gives \(f\circ\eta_{(P,\phi)}=f\circ\bar q_1^\phi=\bar q_1^\psi=\eta_{(Q,\psi)}\). Therefore \(\eta\) is a natural isomorphism.
\end{proof}

\begin{theorem}\label{thm:categorical-classification}
The functors \(\mathcal F:\mathbf{SqPres}(X)\to\mathbf{SqStr}(X)\) and \(\mathcal G:\mathbf{SqStr}(X)\to\mathbf{SqPres}(X)\) define an equivalence of categories. Moreover, \(\mathcal F\mathcal G=\mathrm{id}_{\mathbf{SqStr}(X)}\) and \(\mathcal G\mathcal F\cong\mathrm{id}_{\mathbf{SqPres}(X)}\) through the natural isomorphism \(\eta\).

Consequently, the isomorphism classes of square presentations of \((X,d)\) correspond bijectively to square structures on \((X,d)\). The category \(\mathbf{SqPres}(X)\) is thin, and every square presentation has trivial automorphism group.
\end{theorem}

\begin{proof}
Lemma~\ref{lem:functoriality} shows that \(\mathcal F\) is well-defined on morphisms. Since \(\mathbf{SqStr}(X)\) is discrete, \(\mathcal F\) preserves identities and composition.

Lemma~\ref{lem:presentation-from-structure} shows that \(\mathcal G\) assigns an object of \(\mathbf{SqPres}(X)\) to each object of \(\mathbf{SqStr}(X)\). Since \(\mathbf{SqStr}(X)\) is discrete and \(\mathcal G\) sends identity morphisms to identity morphisms, \(\mathcal G\) preserves identities and composition.

Lemma~\ref{lem:roundtrip-structure-cat} proves that \(\mathcal F\mathcal G=\mathrm{id}_{\mathbf{SqStr}(X)}\). Lemma~\ref{lem:eta-natural-cat} gives a natural isomorphism \(\mathcal G\mathcal F\cong\mathrm{id}_{\mathbf{SqPres}(X)}\). Therefore \(\mathcal F\) and \(\mathcal G\) define an equivalence of categories.

Since \(\mathbf{SqStr}(X)\) is discrete and \(\mathcal F\mathcal G=\mathrm{id}_{\mathbf{SqStr}(X)}\), the equivalence identifies isomorphism classes of square presentations with square structures on \((X,d)\).

Lemma~\ref{lem:presentation-thinness} with \(k=2\) proves that \(\mathbf{SqPres}(X)\) is thin and that every square presentation has trivial automorphism group.
\end{proof}

\subsubsection{Ordered Power Presentations and Structures}

This subsubsection extends Theorem~\ref{thm:categorical-classification} to ordered \(k\)-fold max-product presentations and structures. We fix an integer \(k\geq 2\). The order of the coordinate factors is part of the presentation data.

We define the functor \(\Delta_k:\mathbf{Met}_{\mathrm{iso}}\to\mathbf{Met}_{\mathrm{iso}}\) on objects by \(\Delta_k(P,d_P):=(P^k,d_{P^k})\), where, for \(p=(p_1,\dots,p_k)\) and \(q=(q_1,\dots,q_k)\),
\begin{equation}\label{eq:ordered-power-metric}
d_{P^k}(p,q)
=
\max_{1\leq j\leq k}d_P(p_j,q_j).
\end{equation}
We define \(\Delta_k(f):=f^{\times k}\) on isometries. Equation~\eqref{eq:ordered-power-metric} shows that \(\Delta_k(f)\) is a bijective isometry, and the identity and composition laws for \(\Delta_k\) are immediate. We write \(\Delta_k(P)\) when the metric on \(P\) is clear.

\begin{definition}\label{def:ordered-power-presentation}
We define \(\mathbf{PowPres}_k(X)\) as follows. An object is a pair \((P,\phi)\), where \((P,d_P)\) is a metric space and \(\phi:\Delta_k(P)\to(X,d)\) is an isometry. A morphism \(f:(P,\phi)\to(Q,\psi)\) is an isometry \(f:(P,d_P)\to(Q,d_Q)\) such that
\begin{equation}\label{eq:powpres-morphism-compatibility}
\psi\circ f^{\times k}=\phi.
\end{equation}
\end{definition}

\begin{definition}\label{def:ordered-power-structure}
An ordered \(k\)-power structure on \((X,d)\) is a tuple \(\mathfrak s=(E_1,\dots,E_k,\tau_2,\dots,\tau_k)\), where \(E_1,\dots,E_k\) are equivalence relations on \(X\) and, for each \(2\leq j\leq k\), \(\tau_j:X/E_j\to X/E_1\) is a map such that:
\begin{enumerate}
\item\label{item:k-powstr-quotient-metrics}
for each \(1\leq j\leq k\), the quotient distance \(\bar d_{E_j}\) is a metric on \(X/E_j\);

\item\label{item:k-powstr-intersections}
for every choice of classes \(C_j\in X/E_j\), where \(1\leq j\leq k\), the intersection \(\bigcap_{j=1}^k C_j\) consists of exactly one point;

\item\label{item:k-powstr-decomposition}
the family \((E_1,\dots,E_k)\) satisfies the \(k\)-factor distance decomposition identity from Equation~\eqref{eq:k-factor-distance-decomposition-identity};

\item\label{item:k-powstr-tau-isometries}
for each \(2\leq j\leq k\), the map \(\tau_j:(X/E_j,\bar d_{E_j})\to(X/E_1,\bar d_{E_1})\) is an isometry.
\end{enumerate}
\end{definition}

Let \(\mathbf{PowStr}_k(X)\) denote the discrete category of ordered \(k\)-power structures on \((X,d)\).

Let \((P,\phi)\) be an object of \(\mathbf{PowPres}_k(X)\). For \(1\leq j\leq k\), define \(q_j^\phi:=\pi_j\circ\phi^{-1}:X\to P\). For \(1\leq j\leq k\), define \(E_j^\phi\) by declaring that \(xE_j^\phi y\) if and only if \(q_j^\phi(x)=q_j^\phi(y)\).

\begin{lemma}\label{lem:k-quotient-identification}
For each \(1\leq j\leq k\), the map \(\bar q_j^\phi:X/E_j^\phi\to P\), defined by \(\bar q_j^\phi([x]_{E_j^\phi})=q_j^\phi(x)\), is a bijective isometry.
\end{lemma}

\begin{proof}
The definition of \(E_j^\phi\) shows that \(\bar q_j^\phi\) is well-defined and injective. Since \(P\) is nonempty and \(\phi^{-1}:X\to P^k\) is surjective, \(q_j^\phi\) and \(\bar q_j^\phi\) are surjective.

We fix \(x,y\in X\). If \(x'E_j^\phi x\) and \(y'E_j^\phi y\), then \(q_j^\phi(x')=q_j^\phi(x)\) and \(q_j^\phi(y')=q_j^\phi(y)\). Equation~\eqref{eq:ordered-power-metric} and the fact that \(\phi\) is an isometry give
\[
d(x',y')
=
\max_{1\leq i\leq k}d_P(q_i^\phi(x'),q_i^\phi(y'))
\geq
d_P(q_j^\phi(x),q_j^\phi(y)).
\]
Taking the infimum over all such \(x'\) and \(y'\) gives \(\bar d_{E_j^\phi}([x]_{E_j^\phi},[y]_{E_j^\phi})\geq d_P(q_j^\phi(x),q_j^\phi(y))\).

We choose \(r\in P\). We let \(a,b\in P^k\) satisfy \(a_j=q_j^\phi(x)\), \(b_j=q_j^\phi(y)\), and \(a_i=b_i=r\) for every \(i\neq j\). Set \(u=\phi(a)\) and \(v=\phi(b)\). Then \(uE_j^\phi x\) and \(vE_j^\phi y\). Since \(a_i=b_i\) for every \(i\neq j\), Equation~\eqref{eq:ordered-power-metric} gives \(d(u,v)=d_P(q_j^\phi(x),q_j^\phi(y))\). Therefore
\begin{equation}\label{eq:k-presentation-quotient-distance}
\bar d_{E_j^\phi}([x]_{E_j^\phi},[y]_{E_j^\phi})
=
d_P(q_j^\phi(x),q_j^\phi(y)).
\end{equation}
Hence \(\bar q_j^\phi\) is a bijective isometry.
\end{proof}

For each \(2\leq j\leq k\), define
\begin{equation}\label{eq:k-tau-from-presentation}
\tau_j^\phi
=
(\bar q_1^\phi)^{-1}\circ\bar q_j^\phi
:
X/E_j^\phi\to X/E_1^\phi.
\end{equation}

\begin{lemma}\label{lem:k-presentation-to-structure}
The tuple \(\mathfrak s(\phi)=(E_1^\phi,\dots,E_k^\phi,\tau_2^\phi,\dots,\tau_k^\phi)\) is an ordered \(k\)-power structure on \((X,d)\).
\end{lemma}

\begin{proof}
Lemma~\ref{lem:k-quotient-identification} shows that each quotient distance \(\bar d_{E_j^\phi}\) is a metric, and Equation~\eqref{eq:k-tau-from-presentation} shows that each \(\tau_j^\phi\) is an isometry.

Let \(F_j\) denote equality of the \(j\)-th coordinate on \(P^k\). For each \(j\), \(xE_j^\phi y\) if and only if \(\phi^{-1}(x)F_j\phi^{-1}(y)\). Hence \(\phi^{-1}\) maps intersections of \(E_1^\phi,\dots,E_k^\phi\)-classes bijectively onto intersections of \(F_1,\dots,F_k\)-classes. Lemma~\ref{lem:k-product-intersection} gives the one-point intersection property.

For \(x,y\in X\), Equation~\eqref{eq:k-presentation-quotient-distance} identifies each quotient distance \(\bar d_{E_j^\phi}([x]_{E_j^\phi},[y]_{E_j^\phi})\) with \(d_P(q_j^\phi(x),q_j^\phi(y))\). Since \(\phi^{-1}\) is an isometry from \((X,d)\) to \(\Delta_k(P)\), Equation~\eqref{eq:ordered-power-metric} gives
\[
d(x,y)=\max_{1\leq j\leq k}d_P(q_j^\phi(x),q_j^\phi(y)).
\]
Thus the displayed terms are the corresponding quotient distances, so the \(k\)-factor distance decomposition identity follows. Therefore \(\mathfrak s(\phi)\) is an ordered \(k\)-power structure.
\end{proof}

\begin{lemma}\label{lem:k-functoriality}
Let \(f:(P,\phi)\to(Q,\psi)\) be a morphism in \(\mathbf{PowPres}_k(X)\). Then \(\mathfrak s(\phi)=\mathfrak s(\psi)\). Moreover, for each \(1\leq j\leq k\),
\begin{equation}\label{eq:k-quotient-map-functoriality}
\bar q_j^\psi=f\circ\bar q_j^\phi.
\end{equation}
\end{lemma}

\begin{proof}
Equation~\eqref{eq:powpres-morphism-compatibility} gives \(\psi^{-1}(\phi(z))=f^{\times k}(z)\) for \(z\in P^k\). Taking \(z=\phi^{-1}(x)\) and applying \(\pi_j\) gives \(q_j^\psi(x)=f(q_j^\phi(x))\) for \(x\in X\) and \(1\leq j\leq k\).

For \(x,y\in X\) and \(1\leq j\leq k\), since \(q_j^\psi=f\circ q_j^\phi\) and \(f\) is injective, \(xE_j^\phi y\) if and only if \(xE_j^\psi y\). Hence \(E_j^\phi=E_j^\psi\). Since \(E_j^\phi=E_j^\psi\), the quotient sets agree, and the identity \(q_j^\psi=f\circ q_j^\phi\) gives Equation~\eqref{eq:k-quotient-map-functoriality}.

For \(2\leq j\leq k\), Equation~\eqref{eq:k-quotient-map-functoriality} gives \(\tau_j^\psi=(\bar q_1^\psi)^{-1}\circ\bar q_j^\psi=(\bar q_1^\phi)^{-1}\circ f^{-1}\circ f\circ\bar q_j^\phi\). Hence \(\tau_j^\psi=\tau_j^\phi\). Therefore \(\mathfrak s(\phi)=\mathfrak s(\psi)\).
\end{proof}

We define \(\mathcal F_k:\mathbf{PowPres}_k(X)\to\mathbf{PowStr}_k(X)\) on objects by \(\mathcal F_k(P,\phi):=\mathfrak s(\phi)\). For every morphism \(f:(P,\phi)\to(Q,\psi)\), Lemma~\ref{lem:k-functoriality} gives \(\mathfrak s(\phi)=\mathfrak s(\psi)\), and we define \(\mathcal F_k(f)\) to be the identity morphism of this ordered \(k\)-power structure.

Let \(\mathfrak s=(E_1,\dots,E_k,\tau_2,\dots,\tau_k)\) be an ordered \(k\)-power structure on \((X,d)\). Set \(P_{\mathfrak s}:=X/E_1\), and equip \(P_{\mathfrak s}\) with the metric \(\bar d_{E_1}\). Define \(\Psi_{\mathfrak s}:X\to P_{\mathfrak s}^k\) by
\begin{equation}\label{eq:k-psi-from-structure}
\Psi_{\mathfrak s}(x)
=
([x]_{E_1},\tau_2([x]_{E_2}),\dots,\tau_k([x]_{E_k})).
\end{equation}

\begin{lemma}\label{lem:k-presentation-from-structure}
The map \(\Psi_{\mathfrak s}\) is a bijective isometry from \((X,d)\) onto \(\Delta_k(P_{\mathfrak s})\). Consequently, if \(\phi_{\mathfrak s}:=\Psi_{\mathfrak s}^{-1}\), then \((P_{\mathfrak s},\phi_{\mathfrak s})\) is an object of \(\mathbf{PowPres}_k(X)\).
\end{lemma}

\begin{proof}
Lemma~\ref{lem:k-relations-to-product} shows that the map \(x\mapsto([x]_{E_1},\dots,[x]_{E_k})\) is a bijective isometry from \((X,d)\) onto \(\prod_{j=1}^k X/E_j\) with the metric from Equation~\eqref{eq:k-quotient-product-metric}.

Define \(A:\prod_{j=1}^k X/E_j\to P_{\mathfrak s}^k\) by \(A(C_1,\dots,C_k)=(C_1,\tau_2(C_2),\dots,\tau_k(C_k))\). Since each \(\tau_j\) is bijective, \(A\) is bijective. For all \(C=(C_1,\dots,C_k)\) and \(D=(D_1,\dots,D_k)\), Equation~\eqref{eq:ordered-power-metric} gives
\[
\begin{aligned}
d_{P_{\mathfrak s}^k}(A(C),A(D))
&=
\max\{
\bar d_{E_1}(C_1,D_1),
\bar d_{E_1}(\tau_2(C_2),\tau_2(D_2)),
\dots,
\bar d_{E_1}(\tau_k(C_k),\tau_k(D_k))
\} \\
&=
\max_{1\leq j\leq k}\bar d_{E_j}(C_j,D_j).
\end{aligned}
\]
Thus \(A\) is a bijective isometry.

The composition of the map from Lemma~\ref{lem:k-relations-to-product} with \(A\) equals \(\Psi_{\mathfrak s}\) by the definitions of \(A\) and \(\Psi_{\mathfrak s}\). Hence \(\Psi_{\mathfrak s}\) is a bijective isometry.

Therefore \(\phi_{\mathfrak s}=\Psi_{\mathfrak s}^{-1}\) is an isometry \(\Delta_k(P_{\mathfrak s})\to(X,d)\). Thus \((P_{\mathfrak s},\phi_{\mathfrak s})\) is an object of \(\mathbf{PowPres}_k(X)\).
\end{proof}

We define \(\mathcal G_k:\mathbf{PowStr}_k(X)\to\mathbf{PowPres}_k(X)\) by \(\mathcal G_k(\mathfrak s):=(P_{\mathfrak s},\phi_{\mathfrak s})\) on objects and by sending each identity morphism to the identity morphism of its image.

\begin{lemma}\label{lem:k-roundtrip-structure}
For every ordered \(k\)-power structure \(\mathfrak s\) on \((X,d)\), one has \(\mathcal F_k\mathcal G_k(\mathfrak s)=\mathfrak s\).
\end{lemma}

\begin{proof}
We let \(\mathfrak s=(E_1,\dots,E_k,\tau_2,\dots,\tau_k)\). Since \(\phi_{\mathfrak s}^{-1}=\Psi_{\mathfrak s}\), Equation~\eqref{eq:k-psi-from-structure} gives \(q_1^{\phi_{\mathfrak s}}(x)=[x]_{E_1}\) and \(q_j^{\phi_{\mathfrak s}}(x)=\tau_j([x]_{E_j})\) for \(2\leq j\leq k\). For \(x,y\in X\), the first identity gives \(xE_1^{\phi_{\mathfrak s}}y\) if and only if \(xE_1y\). For \(2\leq j\leq k\), the second identity and the injectivity of \(\tau_j\) give \(xE_j^{\phi_{\mathfrak s}}y\) if and only if \(xE_jy\). Hence \(E_j^{\phi_{\mathfrak s}}=E_j\) for \(1\leq j\leq k\). Under these identifications, \(\bar q_1^{\phi_{\mathfrak s}}=\mathrm{id}_{X/E_1}\) and \(\bar q_j^{\phi_{\mathfrak s}}=\tau_j\) for \(2\leq j\leq k\). Hence \(\tau_j^{\phi_{\mathfrak s}}=\tau_j\) for \(2\leq j\leq k\), so \(\mathcal F_k\mathcal G_k(\mathfrak s)=\mathfrak s\).
\end{proof}

\begin{lemma}\label{lem:k-roundtrip-presentation}
For every object \((P,\phi)\) of \(\mathbf{PowPres}_k(X)\), the map \(\eta_{(P,\phi)}=\bar q_1^\phi:P_{\mathfrak s(\phi)}=X/E_1^\phi\to P\) is a morphism \(\mathcal G_k\mathcal F_k(P,\phi)\to(P,\phi)\) in \(\mathbf{PowPres}_k(X)\). Moreover,
\begin{equation}\label{eq:k-eta-compatibility}
\phi\circ\eta_{(P,\phi)}^{\times k}
=
\phi_{\mathfrak s(\phi)}.
\end{equation}
\end{lemma}

\begin{proof}
Lemma~\ref{lem:k-quotient-identification} shows that \(\eta_{(P,\phi)}=\bar q_1^\phi\) is an isometry. Since \(\phi_{\mathfrak s(\phi)}=\Psi_{\mathfrak s(\phi)}^{-1}\), it suffices to prove
\begin{equation}\label{eq:k-eta-inverse-compatibility}
\eta_{(P,\phi)}^{\times k}
\circ
\Psi_{\mathfrak s(\phi)}
=
\phi^{-1}.
\end{equation}
For \(x\in X\), Equation~\eqref{eq:k-psi-from-structure} shows that the first coordinate of \(\eta_{(P,\phi)}^{\times k}\Psi_{\mathfrak s(\phi)}(x)\) is \(q_1^\phi(x)\). For \(2\leq j\leq k\), Equation~\eqref{eq:k-tau-from-presentation} gives \(\bar q_1^\phi(\tau_j^\phi([x]_{E_j^\phi}))=\bar q_j^\phi([x]_{E_j^\phi})\), so the \(j\)-th coordinate is \(q_j^\phi(x)\). Hence \(\eta_{(P,\phi)}^{\times k}\Psi_{\mathfrak s(\phi)}(x)=\phi^{-1}(x)\), so Equation~\eqref{eq:k-eta-inverse-compatibility} holds. Therefore Equation~\eqref{eq:k-eta-compatibility} holds, so \(\eta_{(P,\phi)}\) is a morphism in \(\mathbf{PowPres}_k(X)\).
\end{proof}

\begin{lemma}\label{lem:k-naturality}
The morphisms \(\eta_{(P,\phi)}:\mathcal G_k\mathcal F_k(P,\phi)\to(P,\phi)\) form a natural isomorphism \(\eta:\mathcal G_k\mathcal F_k\Rightarrow\mathrm{id}_{\mathbf{PowPres}_k(X)}\).
\end{lemma}

\begin{proof}
By Lemma~\ref{lem:k-roundtrip-presentation}, each \(\eta_{(P,\phi)}\) is a morphism in \(\mathbf{PowPres}_k(X)\). Since \(\eta_{(P,\phi)}\) is a bijective isometry, its inverse is a bijective isometry. Composing Equation~\eqref{eq:k-eta-compatibility} on the right with \((\eta_{(P,\phi)}^{-1})^{\times k}\) gives
\begin{equation}\label{eq:k-eta-inverse-morphism-compatibility}
\phi_{\mathfrak s(\phi)}
\circ
(\eta_{(P,\phi)}^{-1})^{\times k}
=
\phi.
\end{equation}
Hence \(\eta_{(P,\phi)}^{-1}\) satisfies the morphism condition from \((P,\phi)\) to \(\mathcal G_k\mathcal F_k(P,\phi)\). Therefore \(\eta_{(P,\phi)}\) is an isomorphism in \(\mathbf{PowPres}_k(X)\).

We let \(f:(P,\phi)\to(Q,\psi)\) be a morphism in \(\mathbf{PowPres}_k(X)\). Lemma~\ref{lem:k-functoriality} gives \(\mathfrak s(\phi)=\mathfrak s(\psi)\), so \(\mathcal G_k\mathcal F_k(f)\) is the identity morphism of \(\mathcal G_k\mathcal F_k(P,\phi)=\mathcal G_k\mathcal F_k(Q,\psi)\). Equation~\eqref{eq:k-quotient-map-functoriality} with \(j=1\) gives \(f\circ\eta_{(P,\phi)}=f\circ\bar q_1^\phi=\bar q_1^\psi=\eta_{(Q,\psi)}\). Therefore \(\eta\) is a natural isomorphism.
\end{proof}

\begin{theorem}\label{thm:k-categorical-classification}
The functors \(\mathcal F_k:\mathbf{PowPres}_k(X)\to\mathbf{PowStr}_k(X)\) and \(\mathcal G_k:\mathbf{PowStr}_k(X)\to\mathbf{PowPres}_k(X)\) define an equivalence of categories. Moreover, \(\mathcal F_k\mathcal G_k=\mathrm{id}_{\mathbf{PowStr}_k(X)}\) and \(\mathcal G_k\mathcal F_k\cong\mathrm{id}_{\mathbf{PowPres}_k(X)}\) through the natural isomorphism \(\eta\).

Consequently, the isomorphism classes of ordered \(k\)-power presentations of \((X,d)\) correspond bijectively to ordered \(k\)-power structures on \((X,d)\). The category \(\mathbf{PowPres}_k(X)\) is thin, and every ordered \(k\)-power presentation has trivial automorphism group.

\[
\begin{tikzcd}[column sep=huge]
\mathbf{PowPres}_k(X)
\arrow[r,shift left=0.7ex,"\mathcal F_k"]
&
\mathbf{PowStr}_k(X)
\arrow[l,shift left=0.7ex,"\mathcal G_k"]
\end{tikzcd}
\]
\end{theorem}

\begin{proof}
Lemma~\ref{lem:k-functoriality} shows that \(\mathcal F_k\) is well-defined on morphisms. Since \(\mathbf{PowStr}_k(X)\) is discrete, \(\mathcal F_k\) preserves identities and composition.

Lemma~\ref{lem:k-presentation-from-structure} shows that \(\mathcal G_k\) is well-defined on objects. Since \(\mathbf{PowStr}_k(X)\) is discrete and \(\mathcal G_k\) sends identity morphisms to identity morphisms, \(\mathcal G_k\) preserves identities and composition.

Lemma~\ref{lem:k-roundtrip-structure} proves that \(\mathcal F_k\mathcal G_k=\mathrm{id}_{\mathbf{PowStr}_k(X)}\). Lemma~\ref{lem:k-naturality} gives a natural isomorphism \(\mathcal G_k\mathcal F_k\cong\mathrm{id}_{\mathbf{PowPres}_k(X)}\). Therefore \(\mathcal F_k\) and \(\mathcal G_k\) define an equivalence of categories.

Since \(\mathbf{PowStr}_k(X)\) is discrete and \(\mathcal F_k\mathcal G_k=\mathrm{id}_{\mathbf{PowStr}_k(X)}\), the equivalence identifies isomorphism classes of ordered \(k\)-power presentations with ordered \(k\)-power structures on \((X,d)\).

Lemma~\ref{lem:presentation-thinness} proves that \(\mathbf{PowPres}_k(X)\) is thin and that every ordered \(k\)-power presentation has trivial automorphism group.
\end{proof}

\subsubsection{\texorpdfstring{\(\ell^p\)}{ell-p} Ordered Power Presentations and Structures}

This subsubsection gives the ordered \(\ell^p\) analogue of Theorem~\ref{thm:k-categorical-classification}. We fix an integer \(k\geq 2\) and a real number \(1\leq p<\infty\).

We define \(\Delta_{k,p}:\mathbf{Met}_{\mathrm{iso}}\to\mathbf{Met}_{\mathrm{iso}}\) on objects by \(\Delta_{k,p}(P,d_P):=(P^k,d_{P^k,p})\), where, for \(u=(u_1,\dots,u_k)\) and \(v=(v_1,\dots,v_k)\),
\begin{equation}\label{eq:ordered-p-power-metric-cat}
d_{P^k,p}(u,v)
=
\left(
\sum_{j=1}^k d_P(u_j,v_j)^p
\right)^{1/p}.
\end{equation}
We define \(\Delta_{k,p}(f):=f^{\times k}\) on isometries. Equation~\eqref{eq:ordered-p-power-metric-cat} shows that \(\Delta_{k,p}(f)\) is a bijective isometry, and the identity and composition laws for \(\Delta_{k,p}\) are immediate. We write \(\Delta_{k,p}(P)\) when the metric on \(P\) is clear.

\begin{definition}\label{def:p-power-presentation}
We define the category \(\mathbf{PowPres}_{k,p}(X)\) as follows. An object is a pair \((P,\phi)\), where \((P,d_P)\) is a metric space and \(\phi:\Delta_{k,p}(P)\to(X,d)\) is an isometry. A morphism \(f:(P,\phi)\to(Q,\psi)\) is an isometry \(f:(P,d_P)\to(Q,d_Q)\) such that
\begin{equation}\label{eq:p-powpres-morphism-condition}
\psi\circ f^{\times k}
=
\phi.
\end{equation}
\end{definition}

\begin{definition}\label{def:p-power-structure}
An ordered \(\ell^p\) \(k\)-power structure on \((X,d)\) is a tuple \(\mathfrak s=(E_1,\dots,E_k,\tau_2,\dots,\tau_k)\), where \(E_1,\dots,E_k\) are equivalence relations on \(X\) and, for each \(2\leq j\leq k\), \(\tau_j:X/E_j\to X/E_1\) is a map such that:
\begin{enumerate}
\item\label{item:p-powstr-quotient-metrics}
for each \(1\leq j\leq k\), the quotient distance \(\bar d_{E_j}\) is a metric on \(X/E_j\);

\item\label{item:p-powstr-intersections}
for every choice of classes \(C_j\in X/E_j\), where \(1\leq j\leq k\), the intersection \(\bigcap_{j=1}^k C_j\) consists of exactly one point;

\item\label{item:p-powstr-decomposition}
the family \((E_1,\dots,E_k)\) satisfies the \(p\)-distance decomposition identity from Definition~\ref{def:p-distance-decomposition};

\item\label{item:p-powstr-tau-isometries}
for each \(2\leq j\leq k\), the map \(\tau_j:(X/E_j,\bar d_{E_j})\to(X/E_1,\bar d_{E_1})\) is an isometry.
\end{enumerate}
\end{definition}

Let \(\mathbf{PowStr}_{k,p}(X)\) denote the discrete category of ordered \(\ell^p\) \(k\)-power structures on \((X,d)\).

Let \((P,\phi)\) be an object of \(\mathbf{PowPres}_{k,p}(X)\). For \(1\leq j\leq k\), define \(q_j^\phi:=\pi_j\circ\phi^{-1}:X\to P\). For \(1\leq j\leq k\), define \(E_j^\phi\) by declaring that \(xE_j^\phi y\) if and only if \(q_j^\phi(x)=q_j^\phi(y)\).

\begin{lemma}\label{lem:p-quotient-identification}
For each \(1\leq j\leq k\), the map \(\bar q_j^\phi:X/E_j^\phi\to P\), defined by \(\bar q_j^\phi([x]_{E_j^\phi})=q_j^\phi(x)\), is a bijective isometry.
\end{lemma}

\begin{proof}
The definition of \(E_j^\phi\) shows that \(\bar q_j^\phi\) is well-defined and injective. Since \(P\) is nonempty and \(\phi^{-1}:X\to P^k\) is surjective, \(q_j^\phi\) and \(\bar q_j^\phi\) are surjective.

We fix \(x,y\in X\). If \(x'E_j^\phi x\) and \(y'E_j^\phi y\), then \(q_j^\phi(x')=q_j^\phi(x)\) and \(q_j^\phi(y')=q_j^\phi(y)\). The \(j\)-th summand in Equation~\eqref{eq:ordered-p-power-metric-cat} gives \(d(x',y')^p\geq d_P(q_j^\phi(x),q_j^\phi(y))^p\), and hence \(d(x',y')\geq d_P(q_j^\phi(x),q_j^\phi(y))\). Taking the infimum over all such \(x'\) and \(y'\) gives \(\bar d_{E_j^\phi}([x]_{E_j^\phi},[y]_{E_j^\phi})\geq d_P(q_j^\phi(x),q_j^\phi(y))\).

We choose \(r\in P\). We let \(a,b\in P^k\) satisfy \(a_j=q_j^\phi(x)\), \(b_j=q_j^\phi(y)\), and \(a_i=b_i=r\) for every \(i\neq j\). We set \(u=\phi(a)\) and \(v=\phi(b)\). Then \(uE_j^\phi x\) and \(vE_j^\phi y\). Since \(a_i=b_i\) for every \(i\neq j\), Equation~\eqref{eq:ordered-p-power-metric-cat} gives \(d(u,v)=d_P(q_j^\phi(x),q_j^\phi(y))\). Therefore
\begin{equation}\label{eq:p-presentation-quotient-distance}
\bar d_{E_j^\phi}([x]_{E_j^\phi},[y]_{E_j^\phi})
=
d_P(q_j^\phi(x),q_j^\phi(y)).
\end{equation}
Hence \(\bar q_j^\phi\) is a bijective isometry.
\end{proof}

For each \(2\leq j\leq k\), define
\begin{equation}\label{eq:p-tau-from-presentation}
\tau_j^\phi
=
(\bar q_1^\phi)^{-1}\circ\bar q_j^\phi
:
X/E_j^\phi\to X/E_1^\phi.
\end{equation}

\begin{lemma}\label{lem:p-presentation-to-structure}
The tuple \(\mathfrak s(\phi)=(E_1^\phi,\dots,E_k^\phi,\tau_2^\phi,\dots,\tau_k^\phi)\) is an ordered \(\ell^p\) \(k\)-power structure on \((X,d)\).
\end{lemma}

\begin{proof}
Lemma~\ref{lem:p-quotient-identification} shows that each quotient distance \(\bar d_{E_j^\phi}\) is a metric, and Equation~\eqref{eq:p-tau-from-presentation} shows that each \(\tau_j^\phi\) is an isometry.

Let \(F_j\) denote equality of the \(j\)-th coordinate on \(P^k\). For each \(j\), \(xE_j^\phi y\) if and only if \(\phi^{-1}(x)F_j\phi^{-1}(y)\). Hence \(\phi^{-1}\) maps intersections of chosen \(E_1^\phi,\dots,E_k^\phi\)-classes bijectively onto the corresponding intersections of \(F_1,\dots,F_k\)-classes. Lemma~\ref{lem:k-product-intersection} gives the one-point intersection property.

For \(x,y\in X\), Equation~\eqref{eq:p-presentation-quotient-distance} identifies each quotient distance \(\bar d_{E_j^\phi}([x]_{E_j^\phi},[y]_{E_j^\phi})\) with \(d_P(q_j^\phi(x),q_j^\phi(y))\). Since \(\phi^{-1}\) is an isometry from \((X,d)\) to \(\Delta_{k,p}(P)\), Equation~\eqref{eq:ordered-p-power-metric-cat} gives
\[
d(x,y)
=
\left(
\sum_{j=1}^k d_P(q_j^\phi(x),q_j^\phi(y))^p
\right)^{1/p}.
\]
Thus the displayed terms are the corresponding quotient distances, so the \(p\)-distance decomposition identity follows. Therefore \(\mathfrak s(\phi)\) is an ordered \(\ell^p\) \(k\)-power structure.
\end{proof}

\begin{lemma}\label{lem:p-functoriality}
Let \(f:(P,\phi)\to(Q,\psi)\) be a morphism in \(\mathbf{PowPres}_{k,p}(X)\). Then \(\mathfrak s(\phi)=\mathfrak s(\psi)\). Moreover, for each \(1\leq j\leq k\),
\begin{equation}\label{eq:p-quotient-map-functoriality}
\bar q_j^\psi=f\circ\bar q_j^\phi.
\end{equation}
\end{lemma}

\begin{proof}
Equation~\eqref{eq:p-powpres-morphism-condition} gives \(\psi^{-1}(\phi(z))=f^{\times k}(z)\) for \(z\in P^k\). Taking \(z=\phi^{-1}(x)\) and applying \(\pi_j\) gives \(q_j^\psi(x)=f(q_j^\phi(x))\) for \(x\in X\) and \(1\leq j\leq k\).

For \(x,y\in X\) and \(1\leq j\leq k\), since \(q_j^\psi=f\circ q_j^\phi\) and \(f\) is injective, \(xE_j^\phi y\) if and only if \(xE_j^\psi y\). Hence \(E_j^\phi=E_j^\psi\). Since \(E_j^\phi=E_j^\psi\), the quotient sets agree, and the identity \(q_j^\psi=f\circ q_j^\phi\) gives Equation~\eqref{eq:p-quotient-map-functoriality}.

For \(2\leq j\leq k\), Equation~\eqref{eq:p-quotient-map-functoriality} gives \(\tau_j^\psi=(\bar q_1^\psi)^{-1}\circ\bar q_j^\psi=(\bar q_1^\phi)^{-1}\circ f^{-1}\circ f\circ\bar q_j^\phi\). Hence \(\tau_j^\psi=\tau_j^\phi\). Therefore \(\mathfrak s(\phi)=\mathfrak s(\psi)\).
\end{proof}

We define \(\mathcal F_{k,p}:\mathbf{PowPres}_{k,p}(X)\to\mathbf{PowStr}_{k,p}(X)\) on objects by \(\mathcal F_{k,p}(P,\phi):=\mathfrak s(\phi)\). For every morphism \(f:(P,\phi)\to(Q,\psi)\), Lemma~\ref{lem:p-functoriality} gives \(\mathfrak s(\phi)=\mathfrak s(\psi)\), and we define \(\mathcal F_{k,p}(f)\) to be the identity morphism of this ordered \(\ell^p\) \(k\)-power structure.

Let \(\mathfrak s=(E_1,\dots,E_k,\tau_2,\dots,\tau_k)\) be an ordered \(\ell^p\) \(k\)-power structure on \((X,d)\). We set \(P_{\mathfrak s}:=X/E_1\), and equip \(P_{\mathfrak s}\) with the metric \(\bar d_{E_1}\). Define \(\Psi_{\mathfrak s}:X\to P_{\mathfrak s}^k\) by
\begin{equation}\label{eq:p-categorical-psi}
\Psi_{\mathfrak s}(x)
=
([x]_{E_1},\tau_2([x]_{E_2}),\dots,\tau_k([x]_{E_k})).
\end{equation}

\begin{lemma}\label{lem:p-presentation-from-structure}
The map \(\Psi_{\mathfrak s}\) is a bijective isometry from \((X,d)\) onto \(\Delta_{k,p}(P_{\mathfrak s})\). Consequently, if \(\phi_{\mathfrak s}:=\Psi_{\mathfrak s}^{-1}\), then \((P_{\mathfrak s},\phi_{\mathfrak s})\) is an object of \(\mathbf{PowPres}_{k,p}(X)\).
\end{lemma}

\begin{proof}
Conditions~\ref{item:p-powstr-quotient-metrics}, \ref{item:p-powstr-intersections}, and \ref{item:p-powstr-decomposition} in Definition~\ref{def:p-power-structure} are the hypotheses of Lemma~\ref{lem:p-relations-to-product}. Hence the map \(x\mapsto([x]_{E_1},\dots,[x]_{E_k})\) is a bijective isometry from \((X,d)\) onto \(\prod_{j=1}^k X/E_j\) with the metric from Equation~\eqref{eq:p-quotient-product-metric}.

Define \(A:\prod_{j=1}^k X/E_j\to P_{\mathfrak s}^k\) by \(A(C_1,\dots,C_k)=(C_1,\tau_2(C_2),\dots,\tau_k(C_k))\). Since each \(\tau_j\) is bijective, \(A\) is bijective. For all \(C=(C_1,\dots,C_k)\) and \(D=(D_1,\dots,D_k)\), Equation~\eqref{eq:ordered-p-power-metric-cat} gives
\[
\begin{aligned}
d_{P_{\mathfrak s}^k,p}(A(C),A(D))
&=
\left(
\bar d_{E_1}(C_1,D_1)^p
+
\sum_{j=2}^k
\bar d_{E_1}(\tau_j(C_j),\tau_j(D_j))^p
\right)^{1/p} \\
&=
\left(
\sum_{j=1}^k
\bar d_{E_j}(C_j,D_j)^p
\right)^{1/p}.
\end{aligned}
\]
Thus \(A\) is a bijective isometry.

The composition of the map from Lemma~\ref{lem:p-relations-to-product} with \(A\) equals \(\Psi_{\mathfrak s}\) by the definitions of \(A\) and \(\Psi_{\mathfrak s}\). Hence \(\Psi_{\mathfrak s}\) is a bijective isometry.

Therefore \(\phi_{\mathfrak s}=\Psi_{\mathfrak s}^{-1}\) is an isometry \(\Delta_{k,p}(P_{\mathfrak s})\to(X,d)\). Thus \((P_{\mathfrak s},\phi_{\mathfrak s})\) is an object of \(\mathbf{PowPres}_{k,p}(X)\).
\end{proof}

We define \(\mathcal G_{k,p}:\mathbf{PowStr}_{k,p}(X)\to\mathbf{PowPres}_{k,p}(X)\) by \(\mathcal G_{k,p}(\mathfrak s):=(P_{\mathfrak s},\phi_{\mathfrak s})\) on objects and by sending each identity morphism to the identity morphism of its image.

\begin{lemma}\label{lem:p-roundtrip-structure}
For every ordered \(\ell^p\) \(k\)-power structure \(\mathfrak s\) on \((X,d)\), one has \(\mathcal F_{k,p}\mathcal G_{k,p}(\mathfrak s)=\mathfrak s\).
\end{lemma}

\begin{proof}
We let \(\mathfrak s=(E_1,\dots,E_k,\tau_2,\dots,\tau_k)\). Since \(\phi_{\mathfrak s}^{-1}=\Psi_{\mathfrak s}\), Equation~\eqref{eq:p-categorical-psi} gives \(q_1^{\phi_{\mathfrak s}}(x)=[x]_{E_1}\) and \(q_j^{\phi_{\mathfrak s}}(x)=\tau_j([x]_{E_j})\) for \(2\leq j\leq k\). For \(x,y\in X\), the first identity gives \(xE_1^{\phi_{\mathfrak s}}y\) if and only if \(xE_1y\). For \(2\leq j\leq k\), the second identity and the injectivity of \(\tau_j\) give \(xE_j^{\phi_{\mathfrak s}}y\) if and only if \(xE_jy\). Hence \(E_j^{\phi_{\mathfrak s}}=E_j\) for \(1\leq j\leq k\). Under these identifications, \(\bar q_1^{\phi_{\mathfrak s}}=\mathrm{id}_{X/E_1}\) and \(\bar q_j^{\phi_{\mathfrak s}}=\tau_j\) for \(2\leq j\leq k\). Hence \(\tau_j^{\phi_{\mathfrak s}}=\tau_j\) for \(2\leq j\leq k\), so \(\mathcal F_{k,p}\mathcal G_{k,p}(\mathfrak s)=\mathfrak s\).
\end{proof}

\begin{lemma}\label{lem:p-roundtrip-presentation}
For every object \((P,\phi)\) of \(\mathbf{PowPres}_{k,p}(X)\), the map \(\eta_{(P,\phi)}=\bar q_1^\phi:P_{\mathfrak s(\phi)}=X/E_1^\phi\to P\) is a morphism \(\mathcal G_{k,p}\mathcal F_{k,p}(P,\phi)\to(P,\phi)\) in \(\mathbf{PowPres}_{k,p}(X)\). Moreover,
\begin{equation}\label{eq:p-eta-compatibility}
\phi\circ\eta_{(P,\phi)}^{\times k}
=
\phi_{\mathfrak s(\phi)}.
\end{equation}
\end{lemma}

\begin{proof}
Lemma~\ref{lem:p-quotient-identification} shows that \(\eta_{(P,\phi)}=\bar q_1^\phi\) is an isometry. Since \(\phi_{\mathfrak s(\phi)}=\Psi_{\mathfrak s(\phi)}^{-1}\), it suffices to prove
\begin{equation}\label{eq:p-eta-inverse-compatibility}
\eta_{(P,\phi)}^{\times k}
\circ
\Psi_{\mathfrak s(\phi)}
=
\phi^{-1}.
\end{equation}
For \(x\in X\), Equation~\eqref{eq:p-categorical-psi} shows that the first coordinate of \(\eta_{(P,\phi)}^{\times k}\Psi_{\mathfrak s(\phi)}(x)\) is \(q_1^\phi(x)\). For \(2\leq j\leq k\), Equation~\eqref{eq:p-tau-from-presentation} gives \(\bar q_1^\phi(\tau_j^\phi([x]_{E_j^\phi}))=\bar q_j^\phi([x]_{E_j^\phi})\), so the \(j\)-th coordinate is \(q_j^\phi(x)\). Hence \(\eta_{(P,\phi)}^{\times k}\Psi_{\mathfrak s(\phi)}(x)=\phi^{-1}(x)\), so Equation~\eqref{eq:p-eta-inverse-compatibility} holds. Therefore Equation~\eqref{eq:p-eta-compatibility} holds, so \(\eta_{(P,\phi)}\) is a morphism in \(\mathbf{PowPres}_{k,p}(X)\).
\end{proof}

\begin{lemma}\label{lem:p-naturality}
The morphisms \(\eta_{(P,\phi)}:\mathcal G_{k,p}\mathcal F_{k,p}(P,\phi)\to(P,\phi)\) form a natural isomorphism \(\eta:\mathcal G_{k,p}\mathcal F_{k,p}\Rightarrow\mathrm{id}_{\mathbf{PowPres}_{k,p}(X)}\).
\end{lemma}

\begin{proof}
By Lemma~\ref{lem:p-roundtrip-presentation}, each \(\eta_{(P,\phi)}\) is a morphism in \(\mathbf{PowPres}_{k,p}(X)\). Since \(\eta_{(P,\phi)}\) is a bijective isometry, its inverse is a bijective isometry. Composing Equation~\eqref{eq:p-eta-compatibility} on the right with \((\eta_{(P,\phi)}^{-1})^{\times k}\) gives
\begin{equation}\label{eq:p-eta-inverse-morphism-compatibility}
\phi_{\mathfrak s(\phi)}
\circ
(\eta_{(P,\phi)}^{-1})^{\times k}
=
\phi.
\end{equation}
Hence \(\eta_{(P,\phi)}^{-1}\) satisfies the morphism condition from \((P,\phi)\) to \(\mathcal G_{k,p}\mathcal F_{k,p}(P,\phi)\). Therefore \(\eta_{(P,\phi)}\) is an isomorphism in \(\mathbf{PowPres}_{k,p}(X)\).

We let \(f:(P,\phi)\to(Q,\psi)\) be a morphism in \(\mathbf{PowPres}_{k,p}(X)\). Lemma~\ref{lem:p-functoriality} gives \(\mathfrak s(\phi)=\mathfrak s(\psi)\), so \(\mathcal G_{k,p}\mathcal F_{k,p}(f)\) is the identity morphism of \(\mathcal G_{k,p}\mathcal F_{k,p}(P,\phi)=\mathcal G_{k,p}\mathcal F_{k,p}(Q,\psi)\). Equation~\eqref{eq:p-quotient-map-functoriality} with \(j=1\) gives \(f\circ\eta_{(P,\phi)}=f\circ\bar q_1^\phi=\bar q_1^\psi=\eta_{(Q,\psi)}\). Therefore \(\eta\) is a natural isomorphism.
\end{proof}

\begin{theorem}\label{thm:p-categorical-classification}
The functors \(\mathcal F_{k,p}:\mathbf{PowPres}_{k,p}(X)\to\mathbf{PowStr}_{k,p}(X)\) and \(\mathcal G_{k,p}:\mathbf{PowStr}_{k,p}(X)\to\mathbf{PowPres}_{k,p}(X)\) define an equivalence of categories. Moreover, \(\mathcal F_{k,p}\mathcal G_{k,p}=\mathrm{id}_{\mathbf{PowStr}_{k,p}(X)}\) and \(\mathcal G_{k,p}\mathcal F_{k,p}\cong\mathrm{id}_{\mathbf{PowPres}_{k,p}(X)}\) through the natural isomorphism \(\eta\).

Consequently, the isomorphism classes of ordered \(\ell^p\) \(k\)-power presentations of \((X,d)\) correspond bijectively to ordered \(\ell^p\) \(k\)-power structures on \((X,d)\). The category \(\mathbf{PowPres}_{k,p}(X)\) is thin, and every ordered \(\ell^p\) \(k\)-power presentation has trivial automorphism group.
\end{theorem}

\begin{proof}
Lemma~\ref{lem:p-functoriality} shows that \(\mathcal F_{k,p}\) is well-defined on morphisms. Since \(\mathbf{PowStr}_{k,p}(X)\) is discrete, \(\mathcal F_{k,p}\) preserves identities and composition.

Lemma~\ref{lem:p-presentation-from-structure} shows that \(\mathcal G_{k,p}\) is well-defined on objects. Since \(\mathbf{PowStr}_{k,p}(X)\) is discrete and \(\mathcal G_{k,p}\) sends identity morphisms to identity morphisms, \(\mathcal G_{k,p}\) preserves identities and composition.

Lemma~\ref{lem:p-roundtrip-structure} proves that \(\mathcal F_{k,p}\mathcal G_{k,p}=\mathrm{id}_{\mathbf{PowStr}_{k,p}(X)}\). Lemma~\ref{lem:p-naturality} gives a natural isomorphism \(\mathcal G_{k,p}\mathcal F_{k,p}\cong\mathrm{id}_{\mathbf{PowPres}_{k,p}(X)}\). Therefore \(\mathcal F_{k,p}\) and \(\mathcal G_{k,p}\) define an equivalence of categories.

Since \(\mathbf{PowStr}_{k,p}(X)\) is discrete and \(\mathcal F_{k,p}\mathcal G_{k,p}=\mathrm{id}_{\mathbf{PowStr}_{k,p}(X)}\), the equivalence identifies isomorphism classes of ordered \(\ell^p\) \(k\)-power presentations with ordered \(\ell^p\) \(k\)-power structures on \((X,d)\).

Lemma~\ref{lem:presentation-thinness} proves that \(\mathbf{PowPres}_{k,p}(X)\) is thin and that every ordered \(\ell^p\) \(k\)-power presentation has trivial automorphism group.
\end{proof}

\subsection{Prime Factorizations of Metric Spaces}

Throughout this subsection, all metric spaces are nonempty, every isometry is bijective, and every metric-space product carries the max product metric from Equation~\eqref{eq:k-product-metric}. We call these products \(\ell^\infty\)-products, and for metric spaces \(P_1,\dots,P_m\), we write \(P_1\times_\infty\cdots\times_\infty P_m\) for their product with this metric.

For a metric space \((X,d_X)\) and \(r\geq 0\), we define the threshold relation \(R_r^X\) by
\begin{equation}\label{eq:threshold-relation}
R_r^X(x,y)
\Longleftrightarrow
d_X(x,y)\leq r.
\end{equation}

\begin{definition}\label{def:threshold-structure}
Let \(D\subseteq[0,\infty)\) be finite and satisfy \(0\in D\). We say that a metric space \((X,d_X)\) is \(D\)-valued if \(d_X(X\times X)\subseteq D\).

We write \(\mathcal L_D\) for the relational language with one binary relation symbol \(R_r\) for each \(r\in D\). If \((X,d_X)\) is \(D\)-valued, then we define \(\mathcal T_D(X,d_X)=\langle X,(R_r^X)_{r\in D}\rangle\). When no ambiguity arises, we write \(\mathcal T_D(X)\).
\end{definition}

If \((X,d_X)\) is \(D\)-valued, then each relation \(R_r^X\) is reflexive and symmetric, \(R_0^X=\mathrm{id}_X\), and \(R_r^X\subseteq R_s^X\) whenever \(r,s\in D\) satisfy \(r\leq s\).

We now record the basic dictionary between \(D\)-valued metric spaces and their threshold structures.

\[
\begin{tikzcd}[column sep=huge,row sep=large]
(P_1,\dots,P_m)
\arrow[r,"\times_\infty"]
\arrow[d,"\mathcal T_D"']
&
P_1\times_\infty\cdots\times_\infty P_m
\arrow[d,"\mathcal T_D"]
\\
(\mathcal T_D(P_1),\dots,\mathcal T_D(P_m))
\arrow[r,"\prod"]
&
\prod_{j=1}^m\mathcal T_D(P_j)
\end{tikzcd}
\]

\begin{lemma}\label{lem:threshold-recovery}
Let \(D\subseteq[0,\infty)\) be finite and satisfy \(0\in D\). Let \((X,d_X)\) and \((Y,d_Y)\) be \(D\)-valued metric spaces. Then, for all \(x,x'\in X\),
\begin{equation}\label{eq:threshold-distance-recovery}
d_X(x,x')
=
\min\{r\in D:R_r^X(x,x')\}.
\end{equation}
Consequently, every isomorphism \(\mathcal T_D(X)\cong\mathcal T_D(Y)\) induces an isometry \((X,d_X)\cong(Y,d_Y)\).
\end{lemma}

\begin{proof}
We fix \(x,x'\in X\). Equation~\eqref{eq:threshold-relation} shows that \(\{r\in D:R_r^X(x,x')\}=\{r\in D:d_X(x,x')\leq r\}\). Since \(d_X(x,x')\in D\), the minimum of this set equals \(d_X(x,x')\). This proves Equation~\eqref{eq:threshold-distance-recovery}.

Let \(\varphi:\mathcal T_D(X)\to\mathcal T_D(Y)\) be an isomorphism. Since \(\varphi\) preserves each relation \(R_r\), Equation~\eqref{eq:threshold-distance-recovery} gives \(d_Y(\varphi(x),\varphi(x'))=d_X(x,x')\) for all \(x,x'\in X\). Since an isomorphism of structures is bijective, \(\varphi\) is an isometry.
\end{proof}

\begin{lemma}\label{lem:threshold-factor-valued}
Let \((X,d_X)\) be a metric space with finite distance set \(D=d_X(X\times X)\). If \((X,d_X)\cong P_1\times_\infty\cdots\times_\infty P_m\), then each factor \(P_j\) is \(D\)-valued.
\end{lemma}

\begin{proof}
Let \(P:=P_1\times_\infty\cdots\times_\infty P_m\), and let \(d_P\) denote the max product metric on \(P\). We fix \(j\in\{1,\dots,m\}\), and let \(p,p'\in P_j\). For each \(i\neq j\), choose \(a_i\in P_i\). Let \(u=(a_1,\dots,a_{j-1},p,a_{j+1},\dots,a_m)\) and \(v=(a_1,\dots,a_{j-1},p',a_{j+1},\dots,a_m)\). Equation~\eqref{eq:k-product-metric} gives \(d_P(u,v)=d_j(p,p')\). Since \((X,d_X)\) is isometric to \(P\), the value \(d_j(p,p')\) belongs to \(D\). Therefore \(P_j\) is \(D\)-valued.
\end{proof}

\begin{lemma}\label{lem:threshold-product}
Let \(D\subseteq[0,\infty)\) be finite and satisfy \(0\in D\). If \((P_j,d_j)\), where \(1\leq j\leq m\), are \(D\)-valued metric spaces, then \(\mathcal T_D(P_1\times_\infty\cdots\times_\infty P_m)\cong\prod_{j=1}^m\mathcal T_D(P_j)\) as \(\mathcal L_D\)-structures.
\end{lemma}

\begin{proof}
Let \(P:=P_1\times_\infty\cdots\times_\infty P_m\), and let \(d_P\) denote its max product metric. Since each \(d_j\) takes values in \(D\), Equation~\eqref{eq:k-product-metric} shows that \(P\) is \(D\)-valued.

Let \(u=(u_1,\dots,u_m)\) and \(v=(v_1,\dots,v_m)\), and let \(r\in D\). Then
\begin{align}
R_r^{P}(u,v)
&\overset{\eqref{eq:threshold-relation}}{\Longleftrightarrow}
d_P(u,v)\leq r \notag\\
&\overset{\eqref{eq:k-product-metric}}{\Longleftrightarrow}
\max_{1\leq j\leq m}d_j(u_j,v_j)\leq r \notag\\
&\Longleftrightarrow
d_j(u_j,v_j)\leq r
\quad\text{for all }j \notag\\
&\overset{\eqref{eq:threshold-relation}}{\Longleftrightarrow}
R_r^{P_j}(u_j,v_j)
\quad\text{for all }j. \label{eq:threshold-product-proof}
\end{align}
Equation~\eqref{eq:threshold-product-proof} is exactly the defining condition for the direct product of relational structures from Equation~\eqref{eq:mckenzie-product-relation}. Thus we obtain the desired isomorphism.
\end{proof}

We next prove the converse: direct-product decompositions of threshold structures recover \(\ell^\infty\)-product decompositions of metric spaces.

\begin{lemma}\label{lem:threshold-factor-recovery}
Let \(D\subseteq[0,\infty)\) be finite and satisfy \(0\in D\). Let \((X,d_X)\) be a \(D\)-valued metric space. Let \(\mathcal A=\langle A,(R_r^{\mathcal A})_{r\in D}\rangle\) and \(\mathcal B=\langle B,(R_r^{\mathcal B})_{r\in D}\rangle\) be \(\mathcal L_D\)-structures. If \(\theta:\mathcal T_D(X)\to\mathcal A\times\mathcal B\) is an isomorphism, then there exist metrics \(d_A\) on \(A\) and \(d_B\) on \(B\) such that \((A,d_A)\) and \((B,d_B)\) are \(D\)-valued metric spaces, \(\mathcal A=\mathcal T_D(A,d_A)\), \(\mathcal B=\mathcal T_D(B,d_B)\), and \((X,d_X)\cong A\times_\infty B\).
\end{lemma}

\begin{proof}
We recover threshold systems on the factors, reconstruct metrics on those factors, and then identify the transported metric with the max product metric.

Transport \(d_X\) along \(\theta\) by defining \(d^\theta\) on \(A\times B\) by \(d^\theta(u,v):=d_X(\theta^{-1}(u),\theta^{-1}(v))\). Then \(d^\theta\) is a \(D\)-valued metric, and \(\theta\) is an isometry from \((X,d_X)\) onto \((A\times B,d^\theta)\). For \(r\in D\), let \(R_r^\theta\) denote the threshold relation on \((A\times B,d^\theta)\). Since \(\theta\) is an isomorphism, Equation~\eqref{eq:mckenzie-product-relation} gives
\begin{equation}\label{eq:threshold-factor-product}
R_r^\theta((a,b),(a',b'))
\Longleftrightarrow
R_r^{\mathcal A}(a,a')
\text{ and }
R_r^{\mathcal B}(b,b')
\end{equation}
for all \(a,a'\in A\), all \(b,b'\in B\), and all \(r\in D\).

Since each \(R_r^\theta\) is reflexive, Equation~\eqref{eq:threshold-factor-product} makes each \(R_r^{\mathcal A}\) and each \(R_r^{\mathcal B}\) reflexive.

Let \(r,s\in D\) satisfy \(r\leq s\), and suppose that \(R_r^{\mathcal A}(a,a')\) holds. Choose \(b\in B\). Reflexivity gives \(R_r^{\mathcal B}(b,b)\), so Equation~\eqref{eq:threshold-factor-product} gives \(R_r^\theta((a,b),(a',b))\). Since \(R_r^\theta\subseteq R_s^\theta\), Equation~\eqref{eq:threshold-factor-product} gives \(R_s^{\mathcal A}(a,a')\). Hence \(R_r^{\mathcal A}\subseteq R_s^{\mathcal A}\). The same argument gives \(R_r^{\mathcal B}\subseteq R_s^{\mathcal B}\).

Since \(D\) is finite and \(d^\theta\) takes values in \(D\), the relation \(R_{\max D}^\theta\) equals \((A\times B)^2\). Let \(a,a'\in A\), and choose \(b\in B\). Equation~\eqref{eq:threshold-factor-product} gives \(R_{\max D}^{\mathcal A}(a,a')\). Hence \(R_{\max D}^{\mathcal A}=A^2\). The same argument gives \(R_{\max D}^{\mathcal B}=B^2\).

Define \(d_A(a,a'):=\min\{r\in D:R_r^{\mathcal A}(a,a')\}\) for \(a,a'\in A\), and define \(d_B(b,b'):=\min\{r\in D:R_r^{\mathcal B}(b,b')\}\) for \(b,b'\in B\). The argument above proves that these minima exist, and both functions take values in \(D\).

Fix \(b\in B\). Since every relation \(R_r^{\mathcal B}\) is reflexive, Equation~\eqref{eq:threshold-factor-product} gives \(R_r^\theta((a,b),(a',b))\) if and only if \(R_r^{\mathcal A}(a,a')\) for every \(r\in D\). Lemma~\ref{lem:threshold-recovery}, applied to \((A\times B,d^\theta)\), gives \(d^\theta((a,b),(a',b))=d_A(a,a')\). Thus \(d_A\) equals the restriction of \(d^\theta\) to \(A\times\{b\}\), so \(d_A\) is a metric. The same argument proves that \(d_B\) is a metric.

Let \(r\in D\). If \(R_r^{\mathcal A}(a,a')\) holds, then the definition of \(d_A\) gives \(d_A(a,a')\leq r\). Conversely, suppose that \(d_A(a,a')\leq r\). The definition of \(d_A\) gives \(R_{d_A(a,a')}^{\mathcal A}(a,a')\), and monotonicity gives \(R_r^{\mathcal A}(a,a')\). Hence \(R_r^{\mathcal A}(a,a')\) holds if and only if \(d_A(a,a')\leq r\). Therefore \(\mathcal A=\mathcal T_D(A,d_A)\). The same argument gives \(\mathcal B=\mathcal T_D(B,d_B)\).

Finally, let \((a,b),(a',b')\in A\times B\). Lemma~\ref{lem:threshold-recovery}, applied to \((A\times B,d^\theta)\), gives
\begin{align}
d^\theta((a,b),(a',b'))
&=
\min\{r\in D:R_r^\theta((a,b),(a',b'))\} \notag\\
&\overset{\eqref{eq:threshold-factor-product}}{=}
\min\{r\in D:R_r^{\mathcal A}(a,a')
\text{ and }
R_r^{\mathcal B}(b,b')\}. \label{eq:factor-recovery-distance}
\end{align}
Since \(\mathcal A=\mathcal T_D(A,d_A)\) and \(\mathcal B=\mathcal T_D(B,d_B)\), the set in Equation~\eqref{eq:factor-recovery-distance} equals \(\{r\in D:d_A(a,a')\leq r\text{ and }d_B(b,b')\leq r\}\). Since \(d_A(a,a')\) and \(d_B(b,b')\) belong to \(D\), the minimum of this set equals \(\max\{d_A(a,a'),d_B(b,b')\}\). Hence \(d^\theta\) is the max product metric induced by \(d_A\) and \(d_B\). Since \(\theta\) is an isometry from \((X,d_X)\) onto \((A\times B,d^\theta)\), we obtain \((X,d_X)\cong A\times_\infty B\).
\end{proof}

We have shown that threshold structures determine the underlying metric and that \(\ell^\infty\)-product decompositions of \(D\)-valued metric spaces correspond exactly to direct-product decompositions of threshold structures. We use this correspondence throughout the remainder of the subsection.

\subsubsection{Prime Factorizations of Finite Metric Spaces}

\begin{definition}\label{def:linfty-prime}
A metric space \((X,d_X)\) is \(\ell^\infty\)-prime if \(|X|>1\) and every decomposition \((X,d_X)\cong Y\times_\infty Z\) satisfies \(|Y|=1\) or \(|Z|=1\).
\end{definition}

The one-point metric space is the identity object for \(\times_\infty\), so Definition~\ref{def:linfty-prime} excludes it.

\begin{lemma}\label{lem:finite-prime-existence}
Every finite metric space with more than one point admits a finite \(\ell^\infty\)-product factorization into finite \(\ell^\infty\)-prime metric spaces.
\end{lemma}

\begin{proof}
We argue by strong induction on the integer \(|X|\), where \(|X|>1\). If \((X,d_X)\) is \(\ell^\infty\)-prime, then the conclusion holds. Otherwise,
\[
(X,d_X)\cong Y\times_\infty Z,
\]
where \(1<|Y|<|X|\) and \(1<|Z|<|X|\). Since \(X\) is finite, both \(Y\) and \(Z\) are finite. The induction hypothesis gives finite \(\ell^\infty\)-prime factorizations of \(Y\) and \(Z\). Associativity of \(\times_\infty\) then gives a finite \(\ell^\infty\)-prime factorization of \(X\).
\end{proof}

\begin{definition}\label{def:separating-threshold}
Let \((X,d_X)\) be a metric space, and let \(r\geq0\). We say that \(R_r^X\) is separating over \(X\) if, for all distinct \(x,y\in X\), there exists \(z\in X\) such that exactly one of \(R_r^X(z,x)\) and \(R_r^X(z,y)\) holds.
\end{definition}

\begin{lemma}\label{lem:separating-implies-thin}
Let \((X,d_X)\) be a metric space, and let \(r\geq0\). If \(R_r^X\) is separating over \(X\), then \(\langle X,R_r^X\rangle\) is thin.
\end{lemma}

\begin{proof}
Suppose that \(x\sim y\) in the sense of Equation~\eqref{eq:mckenzie-skeleton-equivalence}. Since \(R_r^X\) is symmetric, Equation~\eqref{eq:mckenzie-skeleton-equivalence} gives \(R_r^X(z,x)\) if and only if \(R_r^X(z,y)\) for every \(z\in X\). Thus no point separates \(x\) and \(y\). Since \(R_r^X\) is separating over \(X\), we obtain \(x=y\). Therefore \(\langle X,R_r^X\rangle\) is thin.
\end{proof}

\begin{lemma}\label{lem:threshold-structure-refinement}
Let \((X,d_X)\) be a finite metric space, and define \(D:=d_X(X\times X)\). Suppose that there exists \(r\in D\) such that \(R_r^X\) is connected over \(X\) in the sense preceding Lemma~\ref{lem:mckenzie-connected-reflexive-in-q} and separating over \(X\). Then \(\mathcal T_D(X)\) has the strict refinement property.
\end{lemma}

\begin{proof}
Since \(r\in D\), the relation \(R_r^X\) is fundamental in \(\mathcal T_D(X)\), so \(R_r^X\in A(\mathcal T_D(X))\). Lemma~\ref{lem:separating-implies-thin} shows that \(\langle X,R_r^X\rangle\) is thin. Since \(R_r^X\) is reflexive, Lemma~\ref{lem:mckenzie-connected-reflexive-in-q} gives \(\langle X,R_r^X\rangle\in\mathcal Q\). Theorem~\ref{thm:mckenzie-aq} therefore implies that \(\mathcal T_D(X)\) has the strict refinement property.
\end{proof}

\begin{lemma}\label{lem:metric-prime-threshold-prime}
Let \((X,d_X)\) be a \(D\)-valued finite metric space with \(|X|>1\). Then \((X,d_X)\) is \(\ell^\infty\)-prime if and only if \(\mathcal T_D(X)\) is directly indecomposable.
\end{lemma}

\begin{proof}
We prove the contrapositive of both implications.

Suppose first that \((X,d_X)\) is not \(\ell^\infty\)-prime. Then
\[
(X,d_X)\cong Y\times_\infty Z,
\]
where \(|Y|>1\) and \(|Z|>1\). Lemma~\ref{lem:threshold-factor-valued} shows that \(Y\) and \(Z\) are \(D\)-valued, and Lemma~\ref{lem:threshold-product} gives
\[
\mathcal T_D(X)\cong\mathcal T_D(Y)\times\mathcal T_D(Z).
\]
Hence \(\mathcal T_D(X)\) is not directly indecomposable.

Conversely, suppose that
\[
\mathcal T_D(X)\cong\mathcal A\times\mathcal B,
\]
where both factors have more than one element. Lemma~\ref{lem:threshold-factor-recovery} gives finite \(D\)-valued metric spaces \((A,d_A)\) and \((B,d_B)\) such that
\[
(X,d_X)\cong A\times_\infty B.
\]
Since both structures have more than one element, both metric spaces have more than one point. Hence \((X,d_X)\) is not \(\ell^\infty\)-prime.
\end{proof}

\begin{lemma}\label{lem:strict-refinement-implies-unique-factorization}
Let \(\mathcal M\) be a finite structure with the strict refinement property. Then every two finite direct decompositions of \(\mathcal M\) into nontrivial directly indecomposable finite structures have the same number of factors, and after reindexing corresponding factors are isomorphic.
\end{lemma}

\begin{proof}
Suppose that
\[
\mathcal M\cong\prod_{i=1}^m\mathcal B_i
\qquad\text{and}\qquad
\mathcal M\cong\prod_{j=1}^n\mathcal C_j
\]
are decompositions into nontrivial directly indecomposable finite structures.

McKenzie proves on page~64 of \cite{mckenzie1971cardinal} that the strict refinement property implies the refinement property. Equation~\eqref{eq:mckenzie-common-refinement} therefore gives finite structures \(\mathcal D_{ij}\), where \(1\leq i\leq m\) and \(1\leq j\leq n\), such that
\[
\mathcal B_i\cong\prod_{j=1}^n\mathcal D_{ij}
\]
for every \(i\), and
\[
\mathcal C_j\cong\prod_{i=1}^m\mathcal D_{ij}
\]
for every \(j\).

Fix \(i\in\{1,\dots,m\}\). Since \(\mathcal B_i\) is nontrivial and directly indecomposable, at least one factor \(\mathcal D_{ij}\) is nontrivial. Moreover, at most one factor \(\mathcal D_{ij}\) is nontrivial. Otherwise,
\[
\mathcal B_i
\cong
\mathcal D_{ij_1}
\times
\prod_{\substack{1\leq j\leq n\\ j\neq j_1}}
\mathcal D_{ij}
\]
would be a decomposition into two nontrivial factors, which contradicts direct indecomposability. Therefore there exists a unique index \(j(i)\in\{1,\dots,n\}\) such that \(\mathcal D_{i,j(i)}\) is nontrivial.

The same argument shows that, for each \(j\in\{1,\dots,n\}\), there exists a unique index \(i(j)\in\{1,\dots,m\}\) such that \(\mathcal D_{i(j),j}\) is nontrivial.

Fix \(i\in\{1,\dots,m\}\). Since \(\mathcal D_{i,j(i)}\) is nontrivial, the uniqueness in the definition of \(i(j(i))\) gives \(i(j(i))=i\). Similarly, \(j(i(j))=j\) for every \(j\in\{1,\dots,n\}\). Therefore \(i\mapsto j(i)\) is a bijection, so \(m=n\).

After reindexing, we may assume that \(j(i)=i\) for every \(i\). All factors other than \(\mathcal D_{ii}\) in the decompositions of \(\mathcal B_i\) and \(\mathcal C_i\) are trivial. Removing these trivial factors gives
\[
\mathcal B_i\cong\mathcal D_{ii}\cong\mathcal C_i.
\]
\end{proof}

\begin{theorem}\label{thm:finite-prime-factorization}
Let \((X,d_X)\) be a finite metric space with \(|X|>1\), and define \(D:=d_X(X\times X)\). Then \((X,d_X)\) admits a finite \(\ell^\infty\)-product factorization into finite \(\ell^\infty\)-prime metric spaces. If there exists \(r\in D\) such that \(R_r^X\) is connected over \(X\) in the sense preceding Lemma~\ref{lem:mckenzie-connected-reflexive-in-q} and separating over \(X\), then any two finite \(\ell^\infty\)-prime factorizations of \((X,d_X)\) have the same number of factors, and after reindexing corresponding factors are isometric.
\end{theorem}

\begin{proof}
Lemma~\ref{lem:finite-prime-existence} gives existence.

Suppose that
\[
(X,d_X)\cong P_1\times_\infty\cdots\times_\infty P_m
\]
and
\[
(X,d_X)\cong Q_1\times_\infty\cdots\times_\infty Q_n
\]
are finite \(\ell^\infty\)-prime factorizations. Lemma~\ref{lem:threshold-factor-valued} shows that every factor \(P_i\) and every factor \(Q_j\) is \(D\)-valued. Lemma~\ref{lem:threshold-product} gives
\[
\mathcal T_D(X)
\cong
\prod_{i=1}^m\mathcal T_D(P_i)
\qquad\text{and}\qquad
\mathcal T_D(X)
\cong
\prod_{j=1}^n\mathcal T_D(Q_j).
\]
Lemma~\ref{lem:metric-prime-threshold-prime} shows that every factor in these decompositions is directly indecomposable. Since every \(\ell^\infty\)-prime metric space has more than one point, every threshold structure in these decompositions is nontrivial. Lemma~\ref{lem:threshold-structure-refinement} shows that \(\mathcal T_D(X)\) has the strict refinement property. Lemma~\ref{lem:strict-refinement-implies-unique-factorization} therefore gives \(m=n\), and after reindexing
\[
\mathcal T_D(P_i)\cong\mathcal T_D(Q_i)
\]
for every \(1\leq i\leq m\). Lemma~\ref{lem:threshold-recovery} therefore gives an isometry \(P_i\cong Q_i\) for every \(1\leq i\leq m\).
\end{proof}

\begin{corollary}\label{cor:prime-cardinality}
Let \((X,d_X)\) be a finite metric space. If \(|X|\) is prime, then \((X,d_X)\) is \(\ell^\infty\)-prime.
\end{corollary}

\begin{proof}
Let \(Y\) and \(Z\) be finite metric spaces, and suppose that \((X,d_X)\cong Y\times_\infty Z\). Since the spaces are finite, one has \(|X|=|Y||Z|\). Since \(|X|\) is prime, one of \(|Y|\) or \(|Z|\) equals \(1\). Definition~\ref{def:linfty-prime} therefore shows that \((X,d_X)\) is \(\ell^\infty\)-prime.
\end{proof}

\begin{definition}\label{def:equilateral-space}
Let \(n\geq1\) and let \(\lambda>0\). We write \(E_n(\lambda)\) for any \(n\)-point metric space in which distinct points have distance \(\lambda\). Any two such spaces are isometric.
\end{definition}

\begin{corollary}\label{cor:equilateral-factorization}
Let \(a,b\geq1\) be integers, and let \(\lambda>0\). Then \(E_{ab}(\lambda)\) is isometric to \(E_a(\lambda)\times_\infty E_b(\lambda)\). Moreover, for every integer \(n\geq1\), the space \(E_n(\lambda)\) is \(\ell^\infty\)-prime if and only if \(n\) is prime.
\end{corollary}

\begin{proof}
The product contains \(ab\) points. Let \((p,q)\neq(p',q')\) in this product. Then \(p\neq p'\) or \(q\neq q'\). Each factor distance belongs to \(\{0,\lambda\}\), and at least one equals \(\lambda\). Hence the \(\ell^\infty\)-product distance from \((p,q)\) to \((p',q')\) equals \(\lambda\). Therefore \(E_a(\lambda)\times_\infty E_b(\lambda)\) is an equilateral metric space on \(ab\) points with common nonzero distance \(\lambda\). Thus \(E_{ab}(\lambda)\) is isometric to \(E_a(\lambda)\times_\infty E_b(\lambda)\).

If \(n=1\), then \(E_1(\lambda)\) is not \(\ell^\infty\)-prime by Definition~\ref{def:linfty-prime}, and \(1\) is not prime. Now suppose that \(n>1\). If \(n\) is prime, then Corollary~\ref{cor:prime-cardinality} shows that \(E_n(\lambda)\) is \(\ell^\infty\)-prime. If \(n=uv\) for integers \(u,v>1\), then \(E_n(\lambda)\) is isometric to \(E_u(\lambda)\times_\infty E_v(\lambda)\), where both factors have more than one point. Definition~\ref{def:linfty-prime} therefore shows that \(E_n(\lambda)\) is not \(\ell^\infty\)-prime.
\end{proof}

\begin{remark}\label{rem:equilateral-threshold-hypothesis}
Equilateral spaces illustrate that the threshold hypothesis in Theorem~\ref{thm:finite-prime-factorization} does not follow from the existence of a finite \(\ell^\infty\)-prime factorization. If \(n>1\), then \(R_0^{E_n(\lambda)}\) is separating over \(E_n(\lambda)\) but not connected over \(E_n(\lambda)\), while \(R_\lambda^{E_n(\lambda)}\) is connected over \(E_n(\lambda)\) but not separating over \(E_n(\lambda)\).
\end{remark}

\subsubsection{Classification of \texorpdfstring{\(k\)}{k}-Power Structures on Finite Metric Spaces}

Theorem~\ref{thm:finite-prime-factorization} reduces the classification of ordered \(k\)-power structures to the multiplicities of the \(\ell^\infty\)-prime factors of \(X\).

Fix an integer \(k\geq2\), and let \((X,d)\) be a finite metric space with \(|X|>1\). Assume that there exists \(r\in d(X\times X)\) such that \(R_r^X\) is connected over \(X\) and separating over \(X\).

Choose pairwise nonisometric finite \(\ell^\infty\)-prime metric spaces \(A_1,\dots,A_s\) and positive integers \(m_1,\dots,m_s\) such that
\begin{equation}\label{eq:k-root-prime-factorization}
(X,d)
\cong
A_1^{m_1}\times_\infty\cdots\times_\infty A_s^{m_s}.
\end{equation}
Theorem~\ref{thm:finite-prime-factorization} implies that the isometry classes represented by \(A_1,\dots,A_s\) and the multiplicities \(m_1,\dots,m_s\) are uniquely determined up to reindexing.

Assume that \(k\mid m_i\) for every \(1\leq i\leq s\). Define
\begin{equation}\label{eq:k-root-definition}
\sqrt[k]{X}
:=
A_1^{m_1/k}\times_\infty\cdots\times_\infty A_s^{m_s/k}.
\end{equation}
Fix a metric space representing this isometry class. For \(k=2\), write \(\sqrt{X}:=\sqrt[2]{X}\).

\[
\begin{tikzcd}[column sep=huge,row sep=large]
\sqrt[k]{X}
\arrow[r,"\Delta_k"]
&
X
\\
A_1^{m_1/k}\times_\infty\cdots\times_\infty A_s^{m_s/k}
\arrow[r,mapsto]
&
A_1^{m_1}\times_\infty\cdots\times_\infty A_s^{m_s}
\end{tikzcd}
\]

\begin{lemma}\label{lem:k-power-product}
Let \(A\) be a finite metric space, and let \(m,k\geq1\). Then \(\Delta_k(A^m)\cong A^{km}\).
\end{lemma}

\begin{proof}
Associativity of \(\times_\infty\) and coordinate permutations give the desired isometry.
\end{proof}

\begin{lemma}\label{lem:k-root-existence}
The following conditions are equivalent:
\begin{enumerate}
\item There exists a finite metric space \((Y,d_Y)\) such that \((X,d)\cong\Delta_k(Y,d_Y)\).

\item One has \(k\mid m_i\) for every \(1\leq i\leq s\).
\end{enumerate}

If these conditions hold, then \((X,d)\cong\Delta_k(\sqrt[k]{X})\).
\end{lemma}

\begin{proof}
Assume that \((X,d)\cong\Delta_k(Y,d_Y)\) for some finite metric space \((Y,d_Y)\). Since \(|X|=|Y|^k\) and \(|X|>1\), the space \(Y\) has more than one point. By Lemma~\ref{lem:finite-prime-existence}, there exist pairwise nonisometric finite \(\ell^\infty\)-prime metric spaces \(B_1,\dots,B_t\) and positive integers \(n_1,\dots,n_t\) such that
\[
(Y,d_Y)
\cong
B_1^{n_1}\times_\infty\cdots\times_\infty B_t^{n_t}.
\]
Lemma~\ref{lem:k-power-product} gives
\[
\Delta_k(Y,d_Y)
\cong
B_1^{kn_1}\times_\infty\cdots\times_\infty B_t^{kn_t}.
\]
Since \((X,d)\cong\Delta_k(Y,d_Y)\), Equation~\eqref{eq:k-root-prime-factorization} and the uniqueness statement in Theorem~\ref{thm:finite-prime-factorization} imply that \(s=t\) and, after reindexing, \(A_i\cong B_i\) and \(m_i=kn_i\) for every \(1\leq i\leq s\). Hence \(k\mid m_i\) for every \(i\).

Conversely, assume that \(k\mid m_i\) for every \(1\leq i\leq s\). Lemma~\ref{lem:k-power-product} gives
\[
\Delta_k(\sqrt[k]{X})
\cong
A_1^{m_1}\times_\infty\cdots\times_\infty A_s^{m_s}.
\]
Equation~\eqref{eq:k-root-prime-factorization} therefore gives \((X,d)\cong\Delta_k(\sqrt[k]{X})\).
\end{proof}

\begin{lemma}\label{lem:k-root-uniqueness}
Suppose that \((X,d)\cong\Delta_k(Y,d_Y)\). Then \((Y,d_Y)\cong\sqrt[k]{X}\).
\end{lemma}

\begin{proof}
By the first implication of Lemma~\ref{lem:k-root-existence}, one has \(k\mid m_i\) for every \(1\leq i\leq s\).

Let
\[
(Y,d_Y)
\cong
B_1^{n_1}\times_\infty\cdots\times_\infty B_t^{n_t}
\]
be the prime factorization of \(Y\) given by Lemma~\ref{lem:finite-prime-existence}. As in the proof of Lemma~\ref{lem:k-root-existence}, Equation~\eqref{eq:k-root-prime-factorization}, Lemma~\ref{lem:k-power-product}, and the uniqueness statement in Theorem~\ref{thm:finite-prime-factorization} imply that \(s=t\) and, after reindexing, \(A_i\cong B_i\) and \(m_i=kn_i\) for every \(1\leq i\leq s\). Therefore \(n_i=m_i/k\) for every \(i\), so \((Y,d_Y)\cong\sqrt[k]{X}\).
\end{proof}

\begin{corollary}\label{cor:k-root-existence-obstruction}
If some \(m_i\) is not divisible by \(k\), then \(\operatorname{Ob}(\mathbf{PowStr}_k(X))=\varnothing\).
\end{corollary}

\begin{proof}
Every object of \(\mathbf{PowStr}_k(X)\) arises from a presentation \((P,\phi)\) satisfying \((X,d)\cong\Delta_k(P)\). Lemma~\ref{lem:k-root-existence} therefore implies that \(k\mid m_i\) for every \(1\leq i\leq s\). The conclusion follows.
\end{proof}

For the remainder of this subsection, we assume that \(k\mid m_i\) for every \(1\leq i\leq s\). Fix an isometry \(\phi_0:\Delta_k(\sqrt[k]{X})\to X\).

\begin{equation}\label{eq:k-root-subgroup}
H_{\phi_0,k}
:=
\{\phi_0\circ f^{\times k}\circ\phi_0^{-1}:f\in\operatorname{Iso}(\sqrt[k]{X})\}.
\end{equation}
The subgroup \(H_{\phi_0,k}\) consists precisely of the isometries of \(X\) obtained from isometries of \(\sqrt[k]{X}\) through the identification \(\phi_0\).

\begin{lemma}\label{lem:powpres-equality-morphism}
Let \((P,\phi)\) and \((Q,\psi)\) be objects of \(\mathbf{PowPres}_k(X)\). Then \(\mathcal F_k(P,\phi)=\mathcal F_k(Q,\psi)\) if and only if there exists a morphism \((P,\phi)\to(Q,\psi)\) in \(\mathbf{PowPres}_k(X)\).
\end{lemma}

\begin{proof}
If there exists a morphism \(f:(P,\phi)\to(Q,\psi)\), then Lemma~\ref{lem:k-functoriality} gives \(\mathcal F_k(P,\phi)=\mathcal F_k(Q,\psi)\).

Conversely, suppose that \(\mathcal F_k(P,\phi)=\mathcal F_k(Q,\psi)\). Since \(\mathcal F_k\) is an equivalence by Theorem~\ref{thm:k-categorical-classification}, it is full and faithful. Hence the identity morphism of the common image lifts to mutually inverse morphisms \((P,\phi)\to(Q,\psi)\) and \((Q,\psi)\to(P,\phi)\). Therefore a morphism \((P,\phi)\to(Q,\psi)\) exists.
\end{proof}

\begin{lemma}\label{lem:k-root-coset-criterion}
Let \(g,h\in\operatorname{Iso}(X)\). Then \(\mathcal F_k(\sqrt[k]{X},g\circ\phi_0)=\mathcal F_k(\sqrt[k]{X},h\circ\phi_0)\) if and only if \(g^{-1}h\in H_{\phi_0,k}\).
\end{lemma}

\begin{proof}
Suppose that \(\mathcal F_k(\sqrt[k]{X},g\circ\phi_0)=\mathcal F_k(\sqrt[k]{X},h\circ\phi_0)\). Lemma~\ref{lem:powpres-equality-morphism} gives a morphism \(f:(\sqrt[k]{X},g\circ\phi_0)\to(\sqrt[k]{X},h\circ\phi_0)\). Definition~\ref{def:ordered-power-presentation} gives \(h\circ\phi_0\circ f^{\times k}=g\circ\phi_0\). Hence \(g^{-1}h=\phi_0\circ(f^{-1})^{\times k}\circ\phi_0^{-1}\), so \(g^{-1}h\in H_{\phi_0,k}\).

Conversely, suppose that \(g^{-1}h\in H_{\phi_0,k}\). Then there exists \(f\in\operatorname{Iso}(\sqrt[k]{X})\) such that \(g^{-1}h=\phi_0\circ f^{\times k}\circ\phi_0^{-1}\). Hence \(g\circ\phi_0=h\circ\phi_0\circ(f^{-1})^{\times k}\). Definition~\ref{def:ordered-power-presentation} shows that \(f^{-1}:(\sqrt[k]{X},g\circ\phi_0)\to(\sqrt[k]{X},h\circ\phi_0)\) is a morphism in \(\mathbf{PowPres}_k(X)\). Lemma~\ref{lem:powpres-equality-morphism} therefore gives \(\mathcal F_k(\sqrt[k]{X},g\circ\phi_0)=\mathcal F_k(\sqrt[k]{X},h\circ\phi_0)\).
\end{proof}

\begin{theorem}\label{thm:ordered-k-power-structure-classification}
With the preceding notation, the map \(g\mapsto\mathcal F_k(\sqrt[k]{X},g\circ\phi_0)\) induces a bijection
\begin{equation}\label{eq:k-root-coset-bijection}
\operatorname{Iso}(X)/H_{\phi_0,k}
\cong
\operatorname{Ob}(\mathbf{PowStr}_k(X)).
\end{equation}
\end{theorem}

\begin{proof}
Lemma~\ref{lem:k-root-coset-criterion} shows that the map in the theorem statement is constant on right cosets of \(H_{\phi_0,k}\). It also shows that two isometries \(g,h\in\operatorname{Iso}(X)\) determine the same ordered \(k\)-power structure if and only if they belong to the same right coset. Hence the map induces an injective map from \(\operatorname{Iso}(X)/H_{\phi_0,k}\) into \(\operatorname{Ob}(\mathbf{PowStr}_k(X))\).

We prove surjectivity. Let \(\mathfrak s\in\operatorname{Ob}(\mathbf{PowStr}_k(X))\), and let \((P,\phi):=\mathcal G_k(\mathfrak s)\). Lemma~\ref{lem:k-root-uniqueness} gives an isometry \(u:\sqrt[k]{X}\to P\). Define \(g:=(\phi\circ u^{\times k})\circ\phi_0^{-1}\). Then \(g\in\operatorname{Iso}(X)\), and Equation~\eqref{eq:powpres-morphism-compatibility} shows that \(u:(\sqrt[k]{X},g\circ\phi_0)\to(P,\phi)\) is a morphism in \(\mathbf{PowPres}_k(X)\). Lemma~\ref{lem:powpres-equality-morphism} gives \(\mathcal F_k(\sqrt[k]{X},g\circ\phi_0)=\mathcal F_k(P,\phi)\). Since \((P,\phi)=\mathcal G_k(\mathfrak s)\), Theorem~\ref{thm:k-categorical-classification} gives \(\mathcal F_k(P,\phi)=\mathfrak s\). Hence \(\mathfrak s=\mathcal F_k(\sqrt[k]{X},g\circ\phi_0)\), so the induced map is surjective.
\end{proof}

\begin{corollary}\label{cor:k-root-structure-count}
With the preceding notation,
\begin{equation}\label{eq:k-root-structure-count}
\#\operatorname{Ob}(\mathbf{PowStr}_k(X))
=
[\operatorname{Iso}(X):H_{\phi_0,k}].
\end{equation}
\end{corollary}

\begin{proof}
Theorem~\ref{thm:ordered-k-power-structure-classification} identifies \(\operatorname{Ob}(\mathbf{PowStr}_k(X))\) with the right-coset set \(\operatorname{Iso}(X)/H_{\phi_0,k}\). Since \(X\) is finite, \(\operatorname{Iso}(X)\) is finite, and the cardinality of this coset set is \([\operatorname{Iso}(X):H_{\phi_0,k}]\).
\end{proof}

\begin{corollary}\label{cor:k-root-rigidity}
With the preceding notation, if \(H_{\phi_0,k}=\operatorname{Iso}(X)\), then \(\#\operatorname{Ob}(\mathbf{PowStr}_k(X))=1\).
\end{corollary}

\begin{proof}
This follows immediately from Corollary~\ref{cor:k-root-structure-count}.
\end{proof}

\begin{lemma}\label{lem:square-identification}
The category \(\mathbf{SqPres}(X)\) is canonically isomorphic to \(\mathbf{PowPres}_2(X)\). The category \(\mathbf{SqStr}(X)\) is canonically isomorphic to \(\mathbf{PowStr}_2(X)\).
\end{lemma}

\begin{proof}
For \(k=2\), the definitions of \(\mathbf{SqPres}(X)\) and \(\mathbf{PowPres}_2(X)\) coincide up to relabeling the coordinates. The same relabeling identifies \(\mathbf{SqStr}(X)\) and \(\mathbf{PowStr}_2(X)\). The relabelings are mutually inverse and preserve morphisms. Hence the stated canonical isomorphisms hold.
\end{proof}

\begin{corollary}\label{cor:square-structure-classification}
If some \(m_i\) is odd, then \(\mathbf{SqStr}(X)\) has no objects.

Suppose that every \(m_i\) is even. Fix an isometry \(\phi_2:\Delta_2(\sqrt{X})\to X\), and let \(H_{\phi_2,2}\) denote the subgroup from Equation~\eqref{eq:k-root-subgroup} with \(k=2\). Then \(\operatorname{Iso}(X)/H_{\phi_2,2}\cong\operatorname{Ob}(\mathbf{SqStr}(X))\), where the quotient denotes right cosets. Moreover, \(\#\operatorname{Ob}(\mathbf{SqStr}(X))=[\operatorname{Iso}(X):H_{\phi_2,2}]\).
\end{corollary}

\begin{proof}
Lemma~\ref{lem:square-identification} identifies \(\mathbf{SqStr}(X)\) with \(\mathbf{PowStr}_2(X)\). The conclusions follow from Corollary~\ref{cor:k-root-existence-obstruction}, Theorem~\ref{thm:ordered-k-power-structure-classification}, and Corollary~\ref{cor:k-root-structure-count} with \(k=2\).
\end{proof}

\subsubsection{Extension to Metric Spaces with Finite Distance Set}\label{subsec:finite-distance-refinement}

We extend the uniqueness and classification results above from finite metric spaces to metric spaces with finite distance set. Since Lemma~\ref{lem:finite-prime-existence} uses induction on cardinality, its proof does not extend to metric spaces with finite distance set. Instead, we prove uniqueness of finite \(\ell^\infty\)-prime factorizations when such factorizations exist, and we derive the corresponding classification theorem for ordered \(k\)-power structures under that finite-factorization hypothesis.

Throughout this subsubsection, we use Definition~\ref{def:linfty-prime} for arbitrary metric spaces and Definition~\ref{def:threshold-structure} for metric spaces with finite distance set.

\begin{lemma}\label{lem:finite-distance-prime-threshold-prime}
Let \(D\subseteq[0,\infty)\) be finite and contain \(0\). Let \((P,d_P)\) be a nonempty \(D\)-valued metric space with more than one point. Then \(\mathcal T_D(P)\) is nontrivial, and \((P,d_P)\) is \(\ell^\infty\)-prime if and only if \(\mathcal T_D(P)\) is directly indecomposable.
\end{lemma}

\begin{proof}
Since \(P\) has more than one point, the structure \(\mathcal T_D(P)\) is nontrivial. We prove the equivalence by contraposition.

Suppose first that \(P\) is not \(\ell^\infty\)-prime. Then \(P\cong Y\times_\infty Z\), where \(Y\) and \(Z\) both have more than one point. Lemma~\ref{lem:threshold-factor-valued} shows that \(Y\) and \(Z\) are \(D\)-valued, and Lemma~\ref{lem:threshold-product} gives \(\mathcal T_D(P)\cong\mathcal T_D(Y)\times\mathcal T_D(Z)\). Since both factors are nontrivial, \(\mathcal T_D(P)\) is not directly indecomposable.

Conversely, suppose that \(\mathcal T_D(P)\) is not directly indecomposable. Then \(\mathcal T_D(P)\cong\mathcal A\times\mathcal B\), where both factors are nontrivial. Lemma~\ref{lem:threshold-factor-recovery} gives \(D\)-valued metric spaces \(A\) and \(B\), both with more than one point, such that \(P\cong A\times_\infty B\). Hence \(P\) is not \(\ell^\infty\)-prime.
\end{proof}

\begin{lemma}\label{lem:strict-refinement-finite-factor-transfer}
Let \(\mathcal M\) be a structure with the strict refinement property. Then every two finite direct decompositions of \(\mathcal M\) into nontrivial directly indecomposable structures have the same number of factors, and after reindexing corresponding factors are isomorphic.
\end{lemma}

\begin{proof}
Suppose that \(\mathcal M\cong\prod_{i=1}^m\mathcal B_i\) and \(\mathcal M\cong\prod_{j=1}^n\mathcal C_j\) are decompositions into nontrivial directly indecomposable structures. McKenzie proves on page~64 of \cite{mckenzie1971cardinal} that the strict refinement property implies the refinement property for finite direct decompositions. Equation~\eqref{eq:mckenzie-common-refinement} gives structures \(\mathcal D_{ij}\), where \(1\leq i\leq m\) and \(1\leq j\leq n\), such that \(\mathcal B_i\cong\prod_{j=1}^n\mathcal D_{ij}\) for every \(i\), and \(\mathcal C_j\cong\prod_{i=1}^m\mathcal D_{ij}\) for every \(j\).

Fix \(i\). Since \(\mathcal B_i\) is nontrivial, at least one \(\mathcal D_{ij}\) is nontrivial. Since \(\mathcal B_i\) is directly indecomposable, at most one \(\mathcal D_{ij}\) is nontrivial. Hence there is a unique index \(j(i)\) such that \(\mathcal D_{i,j(i)}\) is nontrivial. The same argument gives, for each \(j\), a unique index \(i(j)\) such that \(\mathcal D_{i(j),j}\) is nontrivial.

The uniqueness conditions imply \(i(j(i))=i\) for every \(i\) and \(j(i(j))=j\) for every \(j\). Hence \(i\mapsto j(i)\) is a bijection, so \(m=n\). After reindexing, assume \(j(i)=i\). Then all factors except \(\mathcal D_{ii}\) are trivial in the decompositions of \(\mathcal B_i\) and \(\mathcal C_i\). Since a finite product with all but one factor trivial is isomorphic to its nontrivial factor, \(\mathcal B_i\cong\mathcal D_{ii}\cong\mathcal C_i\) for every \(i\).
\end{proof}

\begin{lemma}\label{lem:direct-factor-finite-uniqueness}
Let \(\mathcal M\) be a structure with the strict refinement property. Suppose that \(\mathcal M\cong\mathcal N\times\mathcal P\), and suppose that \(\mathcal P\) admits a finite direct decomposition into nontrivial directly indecomposable structures. Then every two finite direct decompositions of \(\mathcal N\) into nontrivial directly indecomposable structures have the same number of factors, and after reindexing corresponding factors are isomorphic.
\end{lemma}

\begin{proof}
Let \(\mathcal P\cong\prod_{r=1}^q\mathcal E_r\) be a finite direct decomposition into nontrivial directly indecomposable structures. Let \(\mathcal N\cong\prod_{i=1}^m\mathcal A_i\) and \(\mathcal N\cong\prod_{j=1}^n\mathcal B_j\) be finite direct decompositions into nontrivial directly indecomposable structures. Then \(\mathcal M\cong\prod_{i=1}^m\mathcal A_i\times\prod_{r=1}^q\mathcal E_r\) and \(\mathcal M\cong\prod_{j=1}^n\mathcal B_j\times\prod_{r=1}^q\mathcal E_r\) are finite direct decompositions into nontrivial directly indecomposable structures.

Lemma~\ref{lem:strict-refinement-finite-factor-transfer} identifies the two finite multisets of isomorphism classes \(\{\mathcal A_1,\dots,\mathcal A_m,\mathcal E_1,\dots,\mathcal E_q\}\) and \(\{\mathcal B_1,\dots,\mathcal B_n,\mathcal E_1,\dots,\mathcal E_q\}\). By cancellation in finite multisets, the multisets \(\{\mathcal A_1,\dots,\mathcal A_m\}\) and \(\{\mathcal B_1,\dots,\mathcal B_n\}\) are equal. Hence \(m=n\), and after reindexing one has \(\mathcal A_i\cong\mathcal B_i\) for every \(1\leq i\leq m\).
\end{proof}

\begin{lemma}\label{lem:finite-distance-threshold-refinement}
Let \((X,d)\) be a metric space with finite distance set \(D=d(X\times X)\). Suppose that there exists \(r\in D\) such that \(R_r^X\) is connected over \(X\) and separating over \(X\). Then \(\mathcal T_D(X)\) has the strict refinement property.
\end{lemma}

\begin{proof}
Lemma~\ref{lem:separating-implies-thin} shows that \(\langle X,R_r^X\rangle\) is thin. Equation~\eqref{eq:threshold-relation} shows that \(R_r^X\) is reflexive. Since \(R_r^X\) is connected over \(X\), Lemma~\ref{lem:mckenzie-connected-reflexive-in-q} gives \(\langle X,R_r^X\rangle\in\mathcal Q\). Theorem~\ref{thm:mckenzie-aq} therefore implies that \(\mathcal T_D(X)\) has the strict refinement property.
\end{proof}

\begin{theorem}\label{thm:finite-distance-prime-uniqueness}
Let \((X,d)\) be a metric space with finite distance set \(D=d(X\times X)\). Suppose that there exists \(r\in D\) such that \(R_r^X\) is connected over \(X\) and separating over \(X\). If \((X,d)\cong P_1\times_\infty\cdots\times_\infty P_m\) and \((X,d)\cong Q_1\times_\infty\cdots\times_\infty Q_n\) are finite decompositions in which every factor \(P_i\) and every factor \(Q_j\) is \(\ell^\infty\)-prime, then \(m=n\), and after reindexing \(P_i\) and \(Q_i\) are isometric for every \(1\leq i\leq m\).
\end{theorem}

\begin{proof}
Lemma~\ref{lem:finite-distance-threshold-refinement} shows that \(\mathcal T_D(X)\) has the strict refinement property. Lemma~\ref{lem:threshold-factor-valued} shows that every factor \(P_i\) and every factor \(Q_j\) is \(D\)-valued, and Lemma~\ref{lem:threshold-product} gives
\[
\mathcal T_D(X)
\cong
\prod_{i=1}^m\mathcal T_D(P_i)
\qquad\text{and}\qquad
\mathcal T_D(X)
\cong
\prod_{j=1}^n\mathcal T_D(Q_j).
\]
Lemma~\ref{lem:finite-distance-prime-threshold-prime} shows that every displayed factor is nontrivial and directly indecomposable. Lemma~\ref{lem:strict-refinement-finite-factor-transfer} gives \(m=n\), and after reindexing \(\mathcal T_D(P_i)\cong\mathcal T_D(Q_i)\) for every \(1\leq i\leq m\). Lemma~\ref{lem:threshold-recovery} therefore gives an isometry \(P_i\cong Q_i\) for every \(1\leq i\leq m\).
\end{proof}

\begin{lemma}\label{lem:finite-factor-root-inheritance}
Let \((X,d)\) satisfy the hypotheses of Theorem~\ref{thm:finite-distance-prime-uniqueness}. Suppose that \((X,d)\cong A_1^{m_1}\times_\infty\cdots\times_\infty A_s^{m_s}\), where \(s\geq1\), where \(A_1,\dots,A_s\) are pairwise nonisometric \(\ell^\infty\)-prime metric spaces, and where \(m_i\geq1\) for every \(1\leq i\leq s\). Let \(k\geq2\). If \((X,d)\cong\Delta_k(Y)\), then \(Y\) admits a finite \(\ell^\infty\)-prime factorization, \(k\mid m_i\) for every \(1\leq i\leq s\), and \(Y\cong A_1^{m_1/k}\times_\infty\cdots\times_\infty A_s^{m_s/k}\).
\end{lemma}

\begin{proof}
Set \(D:=d(X\times X)\). Lemma~\ref{lem:finite-distance-threshold-refinement} shows that \(\mathcal T_D(X)\) has the strict refinement property. Lemma~\ref{lem:threshold-factor-valued} shows that each \(A_i\) and \(Y\) is \(D\)-valued. Lemma~\ref{lem:threshold-product} gives
\[
\mathcal T_D(X)
\cong
\prod_{i=1}^s\mathcal T_D(A_i)^{m_i}
\qquad\text{and}\qquad
\mathcal T_D(X)
\cong
\mathcal T_D(Y)^k.
\]
Lemma~\ref{lem:finite-distance-prime-threshold-prime} shows that each \(\mathcal T_D(A_i)\) is nontrivial and directly indecomposable.

Since strict refinement implies the refinement property for finite direct decompositions, Equation~\eqref{eq:mckenzie-common-refinement} applies to the two displayed decompositions. This common refinement distributes each copy of each \(\mathcal T_D(A_i)\) among the \(k\) copies of \(\mathcal T_D(Y)\). More precisely, Equation~\eqref{eq:mckenzie-common-refinement} gives structures \(\mathcal D_{i,\ell,j}\), where \(1\leq i\leq s\), \(1\leq\ell\leq m_i\), and \(1\leq j\leq k\), such that \(\mathcal T_D(A_i)\cong\prod_{j=1}^k\mathcal D_{i,\ell,j}\) for every \(i,\ell\), and
\[
\mathcal T_D(Y)
\cong
\prod_{i=1}^s
\prod_{\ell=1}^{m_i}
\mathcal D_{i,\ell,j}
\]
for every \(1\leq j\leq k\).

Since each \(\mathcal T_D(A_i)\) is nontrivial and directly indecomposable, for each pair \((i,\ell)\) there exists a unique index \(j\) such that \(\mathcal D_{i,\ell,j}\) is nontrivial. For that index, one has \(\mathcal D_{i,\ell,j}\cong\mathcal T_D(A_i)\), and every other factor is trivial.

For each \(1\leq j\leq k\), define \(I_j:=\{(i,\ell):\mathcal D_{i,\ell,j}\text{ is nontrivial}\}\). The sets \(I_1,\dots,I_k\) partition \(\{(i,\ell):1\leq i\leq s,\ 1\leq\ell\leq m_i\}\), and \(\mathcal T_D(Y)\cong\prod_{(i,\ell)\in I_j}\mathcal T_D(A_i)\) for every \(1\leq j\leq k\). Thus \(\mathcal T_D(Y)\) and \(\mathcal T_D(Y)^{k-1}\) admit finite direct decompositions into nontrivial directly indecomposable structures.

Fix \(j_0\in\{1,\dots,k\}\), and write \(I_{j_0}=\{(i_1,\ell_1),\dots,(i_q,\ell_q)\}\). Repeated application of Lemma~\ref{lem:threshold-factor-recovery} to the finite product decomposition \(\mathcal T_D(Y)\cong\prod_{t=1}^q\mathcal T_D(A_{i_t})\) gives metric spaces \(B_1,\dots,B_q\) such that \(Y\cong B_1\times_\infty\cdots\times_\infty B_q\) and \(\mathcal T_D(B_t)\cong\mathcal T_D(A_{i_t})\) for every \(1\leq t\leq q\).

Lemma~\ref{lem:threshold-recovery} therefore gives \(B_t\cong A_{i_t}\) for every \(1\leq t\leq q\). Hence \(Y\cong A_{i_1}\times_\infty\cdots\times_\infty A_{i_q}\).

Since each \(A_{i_t}\) is \(\ell^\infty\)-prime, this is a finite \(\ell^\infty\)-prime factorization of \(Y\). Consequently, \(\mathcal T_D(Y)\) admits a finite direct decomposition into nontrivial directly indecomposable structures.

For each \(i\) and \(j\), let \(a_{i,j}:=\#\{\ell:(i,\ell)\in I_j\}\). We claim that \(a_{i,j}\) is independent of \(j\). The preceding paragraph shows that the hypotheses of Lemma~\ref{lem:direct-factor-finite-uniqueness} hold for the direct factor \(\mathcal T_D(Y)\). The isomorphism \(\mathcal T_D(X)\cong\mathcal T_D(Y)^k\) lets us regard any one copy of \(\mathcal T_D(Y)\) as a direct factor and the product of the remaining \(k-1\) copies as its complement. Fix \(j\), and apply Lemma~\ref{lem:direct-factor-finite-uniqueness} with \(\mathcal N\) equal to the \(j\)-th copy of \(\mathcal T_D(Y)\) and \(\mathcal P\) equal to the product of the remaining copies. This shows that every finite decomposition of that \(j\)-th copy into nontrivial directly indecomposable structures has the same multiset of isomorphism classes.

Coordinate permutations of \(\mathcal T_D(Y)^k\) identify any two copies of \(\mathcal T_D(Y)\). Hence the multisets of directly indecomposable factors obtained from the sets \(I_j\) are the same for all \(j\). Therefore the numbers \(a_{i,j}\) do not depend on \(j\). Write the common value as \(a_i\).

Since \(I_1,\dots,I_k\) partition the \(m_i\) copies of \(\mathcal T_D(A_i)\), one has \(m_i=ka_i\). Thus \(k\mid m_i\). The recovered factorization of \(Y\) has exactly \(a_i=m_i/k\) factors isometric to \(A_i\) for every \(i\), so \(Y\cong A_1^{m_1/k}\times_\infty\cdots\times_\infty A_s^{m_s/k}\).
\end{proof}

\begin{theorem}\label{thm:finite-distance-k-root}
Let \((X,d)\) satisfy the hypotheses of Theorem~\ref{thm:finite-distance-prime-uniqueness}. Suppose that \((X,d)\cong A_1^{m_1}\times_\infty\cdots\times_\infty A_s^{m_s}\) is a finite \(\ell^\infty\)-prime factorization, where \(s\geq1\), where \(A_1,\dots,A_s\) are pairwise nonisometric, and where \(m_i\geq1\) for every \(1\leq i\leq s\). Let \(k\geq2\).

If some \(m_i\) is not divisible by \(k\), then \(\mathbf{PowStr}_k(X)\) has no objects. If \(k\mid m_i\) for every \(i\), let \(\sqrt[k]{X}\) denote a metric space representing the isometry class of \(A_1^{m_1/k}\times_\infty\cdots\times_\infty A_s^{m_s/k}\). Then \(\Delta_k(\sqrt[k]{X})\cong (X,d)\), and every metric space \(Y\) with \((X,d)\cong\Delta_k(Y)\) satisfies \(Y\cong\sqrt[k]{X}\).
\end{theorem}

\begin{proof}
If \(\mathbf{PowStr}_k(X)\) has an object, then Theorem~\ref{thm:k-categorical-classification} gives a presentation \((P,\phi)\) with \(X\cong\Delta_k(P)\). Lemma~\ref{lem:finite-factor-root-inheritance} then forces \(k\mid m_i\) for every \(1\leq i\leq s\). Hence \(\mathbf{PowStr}_k(X)\) has no objects if some \(m_i\) is not divisible by \(k\).

Assume that \(k\mid m_i\) for every \(i\). Theorem~\ref{thm:finite-distance-prime-uniqueness} shows that this isometry class does not depend on the chosen finite \(\ell^\infty\)-prime factorization of \(X\). Associativity of \(\times_\infty\) and coordinate permutations give \(\Delta_k(\sqrt[k]{X})\cong (X,d)\). If \((X,d)\cong\Delta_k(Y)\), then Lemma~\ref{lem:finite-factor-root-inheritance} gives \(Y\cong\sqrt[k]{X}\).
\end{proof}

\begin{theorem}\label{thm:finite-distance-k-power-classification}
Under the hypotheses and notation of Theorem~\ref{thm:finite-distance-k-root}, suppose that \(k\mid m_i\) for every \(1\leq i\leq s\), and let \(\sqrt[k]{X}\) be the metric space defined there. Fix an isometry \(\phi_0:\Delta_k(\sqrt[k]{X})\to X\), and define
\begin{equation}\label{eq:finite-distance-k-root-subgroup}
H_{\phi_0,k}
:=
\{\phi_0\circ f^{\times k}\circ\phi_0^{-1}:f\in\operatorname{Iso}(\sqrt[k]{X})\}.
\end{equation}
Then the map \(g\mapsto\mathcal F_k(\sqrt[k]{X},g\circ\phi_0)\) induces a bijection
\begin{equation}\label{eq:finite-distance-k-root-coset-bijection}
\operatorname{Iso}(X)/H_{\phi_0,k}
\cong
\operatorname{Ob}(\mathbf{PowStr}_k(X)),
\end{equation}
where the quotient denotes right cosets.
\end{theorem}

\begin{proof}
The set \(H_{\phi_0,k}\) is a subgroup of \(\operatorname{Iso}(X)\), since it is the conjugate by \(\phi_0\) of the image of the homomorphism \(f\mapsto f^{\times k}\) from \(\operatorname{Iso}(\sqrt[k]{X})\) to \(\operatorname{Iso}(\Delta_k(\sqrt[k]{X}))\).

Lemma~\ref{lem:k-root-coset-criterion} shows that two isometries \(g,h\in\operatorname{Iso}(X)\) determine the same object of \(\mathbf{PowStr}_k(X)\) if and only if \(g^{-1}h\in H_{\phi_0,k}\). Hence the map induces an injection from the right-coset set \(\operatorname{Iso}(X)/H_{\phi_0,k}\) into \(\operatorname{Ob}(\mathbf{PowStr}_k(X))\).

For surjectivity, let \(\mathfrak s\in\operatorname{Ob}(\mathbf{PowStr}_k(X))\), and let \((P,\phi):=\mathcal G_k(\mathfrak s)\). Theorem~\ref{thm:finite-distance-k-root} gives an isometry \(u:\sqrt[k]{X}\to P\). Define \(g:=(\phi\circ u^{\times k})\circ\phi_0^{-1}\). Equation~\eqref{eq:powpres-morphism-compatibility} gives \(\phi\circ u^{\times k}=g\circ\phi_0\). Hence \(u:(\sqrt[k]{X},g\circ\phi_0)\to(P,\phi)\) is a morphism in \(\mathbf{PowPres}_k(X)\). Lemma~\ref{lem:powpres-equality-morphism} therefore gives \(\mathcal F_k(\sqrt[k]{X},g\circ\phi_0)=\mathcal F_k(P,\phi)\). Since \((P,\phi)=\mathcal G_k(\mathfrak s)\), Theorem~\ref{thm:k-categorical-classification} gives \(\mathcal F_k(P,\phi)=\mathfrak s\). Hence \(\mathfrak s=\mathcal F_k(\sqrt[k]{X},g\circ\phi_0)\). Thus the induced map is surjective.
\end{proof}

\subsection{Self-Square Metric Spaces}

Throughout this subsection, all metric spaces are nonempty and all isometries are bijective. We write \(X\cong Y\) when \(X\) and \(Y\) are isometric.

For metric spaces \((X_i,d_i)\) indexed by a set \(I\), define the \(\ell^\infty\) product metric by
\begin{equation}
\label{eq:k-product-metric-2}
d((x_i),(y_i))
:=
\sup_{i\in I} d_i(x_i,y_i).
\end{equation}
Throughout this subsection, every product carries the metric from Equation~\eqref{eq:k-product-metric-2} unless otherwise stated. We write \(X\times_\infty Y\) for the binary product with this metric. We call such a product a metric space only when Equation~\eqref{eq:k-product-metric-2} takes finite values on every pair of points.

\begin{lemma}
\label{lem:product-reindexing}
Let \(A\) be a set. For each \(a\in A\), let \((P_a,d_a)\) be a metric space.

\begin{enumerate}
\item
Let \(I_a\) and \(J_a\) be sets, and suppose that \(\sigma_a:J_a\to I_a\) is a bijection for every \(a\in A\). Then the map
\[
\Phi:
\prod_{a\in A} P_a^{I_a}
\to
\prod_{a\in A} P_a^{J_a}
\]
defined by
\[
\Phi((x_{a,i})_{a\in A,\,i\in I_a})
:=
(x_{a,\sigma_a(j)})_{a\in A,\,j\in J_a}
\]
is an isometry.

\item
For each \(a\in A\), let \(K_a\) and \(L_a\) be sets. Then
\[
\left(
\prod_{a\in A} P_a^{K_a}
\right)
\times_\infty
\left(
\prod_{a\in A} P_a^{L_a}
\right)
\cong
\prod_{a\in A}
P_a^{(K_a\times\{0\})\cup(L_a\times\{1\})}.
\]
\end{enumerate}
\end{lemma}

\begin{proof}
We prove part~(1). Since each \(\sigma_a\) is bijective, the coordinate reindexing defined by the inverse bijections \(\sigma_a^{-1}\) is inverse to \(\Phi\). Hence \(\Phi\) is bijective.

Fix \(x=(x_{a,i})_{a\in A,\,i\in I_a}\) and \(y=(y_{a,i})_{a\in A,\,i\in I_a}\). Equation~\eqref{eq:k-product-metric-2} gives
\begin{align*}
d(\Phi(x),\Phi(y))
&=
\sup_{a\in A}
\sup_{j\in J_a}
d_a(x_{a,\sigma_a(j)},y_{a,\sigma_a(j)}) \\
&=
\sup_{a\in A}
\sup_{i\in I_a}
d_a(x_{a,i},y_{a,i}) \\
&=
d(x,y).
\end{align*}
Therefore \(\Phi\) is an isometry.

We prove part~(2). Define
\[
\Psi:
\left(
\prod_{a\in A}P_a^{K_a}
\right)
\times_\infty
\left(
\prod_{a\in A}P_a^{L_a}
\right)
\to
\prod_{a\in A}
P_a^{(K_a\times\{0\})\cup(L_a\times\{1\})}
\]
by
\[
\Psi((x,y))_{a,(i,0)}
:=
x_{a,i},
\qquad
\Psi((x,y))_{a,(j,1)}
:=
y_{a,j}.
\]

The inverse map sends
\[
z
\mapsto
\Bigl(
(z_{a,(i,0)})_{a\in A,\,i\in K_a},
(z_{a,(j,1)})_{a\in A,\,j\in L_a}
\Bigr).
\]
Hence \(\Psi\) is bijective.

Fix \((x,y)\) and \((x',y')\) in the domain of \(\Psi\). Equation~\eqref{eq:k-product-metric-2} gives
\begin{align*}
d(\Psi(x,y),\Psi(x',y'))
&=
\sup_{a\in A}
\max\left\{
\sup_{i\in K_a} d_a(x_{a,i},x'_{a,i}),
\sup_{j\in L_a} d_a(y_{a,j},y'_{a,j})
\right\} \\
&=
\max\left\{
\sup_{a\in A}\sup_{i\in K_a} d_a(x_{a,i},x'_{a,i}),
\sup_{a\in A}\sup_{j\in L_a} d_a(y_{a,j},y'_{a,j})
\right\} \\
&=
\max\{d(x,x'),d(y,y')\}.
\end{align*}
The last expression is the distance from \((x,y)\) to \((x',y')\) in the binary \(\ell^\infty\)-product. Therefore \(\Psi\) is an isometry.
\end{proof}

\subsubsection{Binary Branch Representations of Self-Square Metric Spaces}
\label{subsec:binary-branch-representations}

We now derive a binary branch representation of self-square metric spaces. Iterating the coordinate maps of a self-square isometry produces a binary branch geometry indexed by \(\Sigma\).

Define \(\Sigma:=\{0,1\}^{\mathbb N}\). We call elements of \(\Sigma\) branches. We write \(\{0,1\}^{<\mathbb N}\) for the set of finite binary words, \(\emptyset\) for the empty word, and \(|w|\) for the length of a finite word \(w\).

For \(\beta=(\beta_1,\beta_2,\ldots)\in\Sigma\), define \(\beta|0:=\emptyset\) and \(\beta|n:=\beta_1\cdots\beta_n\) for every \(n\geq1\). For \(i\in\{0,1\}\), define \(i\beta:=(i,\beta_1,\beta_2,\ldots)\).

Whenever \((X,d)\) is a metric space and maps \(f_0,f_1:X\to X\) are fixed, define \(f_\emptyset:=\operatorname{id}_X\). For each nonempty word \(w=i_1\cdots i_n\), define \(f_w:=f_{i_n}\circ\cdots\circ f_{i_1}\).

\begin{lemma}
\label{lem:binary-branch-iteration}
Let \((X,d)\) be a metric space, and let \(f_0,f_1:X\to X\) satisfy
\begin{equation}
\label{eq:binary-branch-max}
d(x,y)
=
\max\{
d(f_0(x),f_0(y)),
d(f_1(x),f_1(y))
\}
\end{equation}
for all \(x,y\in X\). Then, for every \(n\geq1\),
\begin{equation}
\label{eq:binary-branch-iterated-max}
d(x,y)
=
\max_{|w|=n} d(f_w(x),f_w(y))
\end{equation}
for all \(x,y\in X\).
\end{lemma}

\begin{proof}
Equation~\eqref{eq:binary-branch-max} implies that \(f_0\) and \(f_1\) are \(1\)-Lipschitz. We prove Equation~\eqref{eq:binary-branch-iterated-max} by induction on \(n\). The case \(n=1\) is Equation~\eqref{eq:binary-branch-max}.

Assume that Equation~\eqref{eq:binary-branch-iterated-max} holds for some \(n\geq1\). Fix a word \(w\) of length \(n\). Equation~\eqref{eq:binary-branch-max} applied to \(f_w(x)\) and \(f_w(y)\) gives
\[
d(f_w(x),f_w(y))
=
\max_{i\in\{0,1\}}
d(f_{wi}(x),f_{wi}(y)).
\]
Taking the maximum over all words \(w\) of length \(n\) and using the induction hypothesis gives Equation~\eqref{eq:binary-branch-iterated-max} for \(n+1\).
\end{proof}

\begin{lemma}
\label{lem:binary-branch-pseudometrics}
Let \((X,d)\) be a metric space, and let \(f_0,f_1:X\to X\) satisfy Equation~\eqref{eq:binary-branch-max}. For each \(\beta\in\Sigma\), define
\begin{equation}
\label{eq:binary-branch-limit}
\rho_\beta(x,y)
:=
\lim_{n\to\infty}
d(f_{\beta|n}(x),f_{\beta|n}(y)).
\end{equation}
Then each \(\rho_\beta\) is a pseudometric on \(X\), and
\begin{equation}
\label{eq:binary-branch-reconstruction}
d(x,y)
=
\sup_{\beta\in\Sigma}\rho_\beta(x,y)
\end{equation}
for all \(x,y\in X\).
\end{lemma}

\begin{proof}
Equation~\eqref{eq:binary-branch-max} implies that \(f_0\) and \(f_1\) are \(1\)-Lipschitz. Hence, for fixed \(\beta,x,y\), the sequence
\[
n
\mapsto
d(f_{\beta|n}(x),f_{\beta|n}(y))
\]
is nonincreasing, so the limit in Equation~\eqref{eq:binary-branch-limit} exists.

Symmetry and the identity \(\rho_\beta(x,x)=0\) are immediate. The triangle inequality for \(d\), applied to \(f_{\beta|n}(x)\), \(f_{\beta|n}(y)\), and \(f_{\beta|n}(z)\), gives
\[
d(f_{\beta|n}(x),f_{\beta|n}(z))
\leq
d(f_{\beta|n}(x),f_{\beta|n}(y))
+
d(f_{\beta|n}(y),f_{\beta|n}(z)).
\]
Taking limits on both sides gives the triangle inequality for \(\rho_\beta\).

Since each \(f_{\beta|n}\) is \(1\)-Lipschitz, we have \(\rho_\beta(x,y)\leq d(x,y)\) for all \(\beta,x,y\). Hence \(\sup_{\beta\in\Sigma}\rho_\beta(x,y)\leq d(x,y)\).

Fix \(x,y\in X\), and let \(r<d(x,y)\). For every \(n\geq1\), Lemma~\ref{lem:binary-branch-iteration} gives a word \(w\in\{0,1\}^n\) such that \(d(f_w(x),f_w(y))>r\). Define
\[
T_r(x,y)
:=
\{
w\in\{0,1\}^{<\mathbb N}
:
d(f_w(x),f_w(y))>r
\}.
\]
Then \(T_r(x,y)\) contains a word of every length.

Suppose that \(v=wu\) and \(v\in T_r(x,y)\). Then \(f_v=f_u\circ f_w\). Since \(f_u\) is \(1\)-Lipschitz,
\[
d(f_v(x),f_v(y))
\leq
d(f_w(x),f_w(y)).
\]
Therefore \(w\in T_r(x,y)\).

Since \(T_r(x,y)\) is finitely branching and contains vertices at every level, K\H{o}nig's lemma \cite[p.~203]{konig2001theorie} gives \(\beta\in\Sigma\) such that \(\beta|n\in T_r(x,y)\) for every \(n\geq1\). Thus \(d(f_{\beta|n}(x),f_{\beta|n}(y))>r\) for every \(n\geq1\), and Equation~\eqref{eq:binary-branch-limit} gives \(\rho_\beta(x,y)\geq r\). Since \(r<d(x,y)\) was arbitrary, we obtain \(\sup_{\beta\in\Sigma}\rho_\beta(x,y)\geq d(x,y)\). Combining this inequality with the reverse inequality gives Equation~\eqref{eq:binary-branch-reconstruction}.
\end{proof}

Assume now that \((X,d)\) admits an isometry \(\Phi:(X,d)\xrightarrow{\sim}X\times_\infty X\). Define \(f_0:=\pi_0\circ\Phi\) and \(f_1:=\pi_1\circ\Phi\). For each \(\beta\in\Sigma\), let \(\rho_\beta\) denote the pseudometric from Lemma~\ref{lem:binary-branch-pseudometrics}, and define \(x\sim_\beta y\) if and only if \(\rho_\beta(x,y)=0\). Let \(M_\beta:=X/\!\sim_\beta\). The pseudometric \(\rho_\beta\) induces a metric on \(M_\beta\), which we again denote by \(\rho_\beta\).

Define \(J:X\to\prod_{\beta\in\Sigma}M_\beta\) by \(J(x):=([x]_\beta)_{\beta\in\Sigma}\), and set \(S:=J(X)\).

\begin{lemma}
\label{lem:binary-branch-shift}
For every \(i\in\{0,1\}\), every \(\beta\in\Sigma\), and every \(x,y\in X\),
\begin{equation}
\label{eq:binary-branch-shift}
\rho_{i\beta}(x,y)
=
\rho_\beta(f_i(x),f_i(y)).
\end{equation}
\end{lemma}

\begin{proof}
For every \(n\geq0\), the identity \(f_{(i\beta)|(n+1)}=f_{\beta|n}\circ f_i\) follows from the definitions. Hence
\begin{align*}
\rho_{i\beta}(x,y)
&=
\lim_{n\to\infty}
d(f_{(i\beta)|(n+1)}(x),f_{(i\beta)|(n+1)}(y)) \\
&=
\lim_{n\to\infty}
d(f_{\beta|n}(f_i(x)),f_{\beta|n}(f_i(y))) \\
&=
\rho_\beta(f_i(x),f_i(y)).
\end{align*}
\end{proof}

\begin{definition}
\label{def:binary-branch-system}
A reconstructive binary branch system consists of metric spaces \((M_\beta,\rho_\beta)\) indexed by \(\beta\in\Sigma\), isometries \(\alpha_i^\beta:M_{i\beta}\xrightarrow{\sim}M_\beta\) for \(i\in\{0,1\}\) and \(\beta\in\Sigma\), and a subset \(S\subseteq\prod_{\beta\in\Sigma}M_\beta\).

Given such data, define
\[
d_{\sup}(z,z')
:=
\sup_{\beta\in\Sigma}
\rho_\beta(z_\beta,z'_\beta).
\]
Define maps \(L,R:S\to S\) by
\[
L(z)_\beta
:=
\alpha_0^\beta(z_{0\beta}),
\qquad
R(z)_\beta
:=
\alpha_1^\beta(z_{1\beta}).
\]
Define \(\Theta:S\to S\times_\infty S\) by \(\Theta(z):=(L(z),R(z))\).

We require:
\begin{enumerate}
\item
\label{item:binary-branch-finite}
the function \(d_{\sup}\) defines a metric on \(S\);

\item
\label{item:binary-branch-invariant}
the maps \(L\) and \(R\) map \(S\) into itself;

\item
\label{item:binary-branch-bijective}
the map \(\Theta:S\to S\times_\infty S\) is bijective.
\end{enumerate}
\end{definition}

\begin{figure}[ht]
\centering
\includegraphics[width=0.82\textwidth]{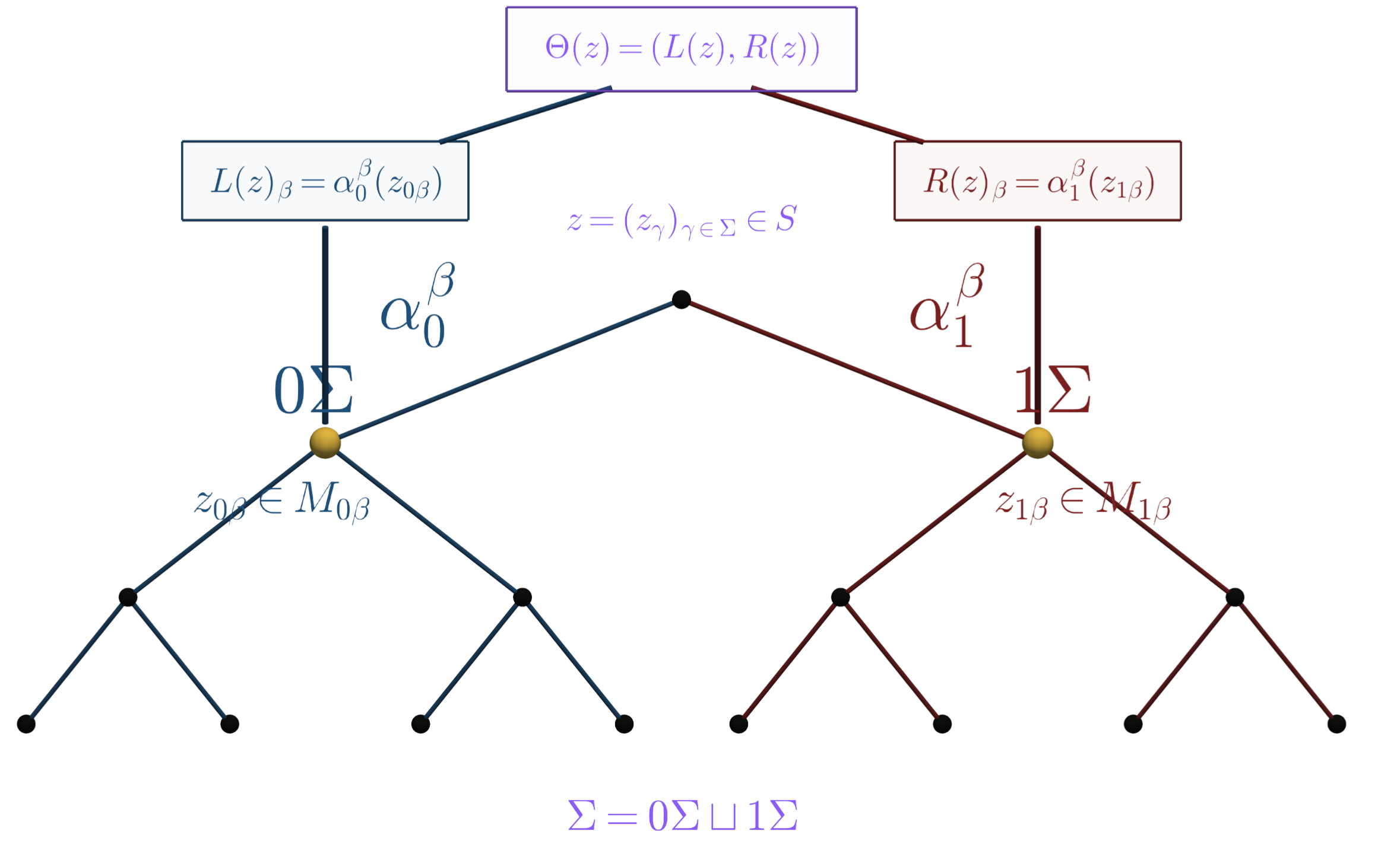}
\caption{
Binary branch reconstruction. The branch families \(0\Sigma\) and \(1\Sigma\) are transported back to \(\Sigma\), producing the two components \(L(z)\) and \(R(z)\) of the self-square map.
}
\label{fig:binary-branch-reconstruction}
\end{figure}

Figure~\ref{fig:binary-branch-reconstruction} shows how the branch families \(0\Sigma\) and \(1\Sigma\) determine the two components of the self-square map.

\begin{lemma}
\label{lem:self-square-to-branch-system}
Let \((X,d)\) be a metric space, and let \(\Phi:(X,d)\xrightarrow{\sim}X\times_\infty X\) be an isometry. Then these data define a reconstructive binary branch system. Moreover, \(J:(X,d)\xrightarrow{\sim}(S,d_{\sup})\) is an isometry.
\end{lemma}

\begin{proof}
Equation~\eqref{eq:binary-branch-reconstruction} gives
\[
d_{\sup}(J(x),J(y))
=
d(x,y)
\]
for all \(x,y\in X\). Hence \(J:(X,d)\to(S,d_{\sup})\) is an isometry. Since \(S=J(X)\), the function \(d_{\sup}\) defines a metric on \(S\).

For each \(i\in\{0,1\}\) and \(\beta\in\Sigma\), define \(\alpha_i^\beta:M_{i\beta}\to M_\beta\) by \(\alpha_i^\beta([x]_{i\beta}):=[f_i(x)]_\beta\).

If \([x]_{i\beta}=[y]_{i\beta}\), then \(\rho_{i\beta}(x,y)=0\). Lemma~\ref{lem:binary-branch-shift} gives \(\rho_\beta(f_i(x),f_i(y))=0\), so \([f_i(x)]_\beta=[f_i(y)]_\beta\). Thus \(\alpha_i^\beta\) is well defined.

For all \(x,y\in X\), Lemma~\ref{lem:binary-branch-shift} gives
\begin{align*}
\rho_\beta(
\alpha_i^\beta([x]_{i\beta}),
\alpha_i^\beta([y]_{i\beta})
)
&=
\rho_\beta(f_i(x),f_i(y)) \\
&=
\rho_{i\beta}(x,y).
\end{align*}
Hence \(\alpha_i^\beta\) preserves distances.

Let \([a]_\beta\in M_\beta\), and choose \(b\in X\). If \(i=0\), surjectivity of \(\Phi\) gives \(x\in X\) such that \(\Phi(x)=(a,b)\). If \(i=1\), surjectivity of \(\Phi\) gives \(x\in X\) such that \(\Phi(x)=(b,a)\). In both cases \(f_i(x)=a\), so \(\alpha_i^\beta([x]_{i\beta})=[a]_\beta\). Hence \(\alpha_i^\beta\) is surjective. Therefore \(\alpha_i^\beta\) is an isometry.

For every \(x\in X\), we have
\[
L(J(x))
=
J(f_0(x))
\]
and
\[
R(J(x))
=
J(f_1(x)).
\]
Hence \(L(S)\subseteq S\) and \(R(S)\subseteq S\).

Let \((u,v)\in S\times_\infty S\). Choose \(a,b\in X\) such that \(u=J(a)\) and \(v=J(b)\). Since \(\Phi\) is surjective, choose \(x\in X\) such that \(\Phi(x)=(a,b)\). Then \(f_0(x)=a\) and \(f_1(x)=b\), so \(\Theta(J(x))=(u,v)\). Hence \(\Theta\) is surjective.

Suppose that \(\Theta(J(x))=\Theta(J(y))\). Then \(L(J(x))=L(J(y))\) and \(R(J(x))=R(J(y))\). Hence
\[
J(f_0(x))
=
J(f_0(y))
\]
and
\[
J(f_1(x))
=
J(f_1(y)).
\]
Since \(J\) is injective, we obtain \(f_0(x)=f_0(y)\) and \(f_1(x)=f_1(y)\). Since \(\Phi\) is injective, we conclude that \(x=y\). Therefore \(\Theta\) is injective.
\end{proof}

\begin{lemma}
\label{lem:branch-system-square-isometry}
Let \((M_\beta,\rho_\beta)\), where \(\beta\in\Sigma\), be metric spaces. Let \(S\subseteq\prod_{\beta\in\Sigma}M_\beta\), and let \(\alpha_i^\beta:M_{i\beta}\xrightarrow{\sim}M_\beta\) be isometries for \(i\in\{0,1\}\) and \(\beta\in\Sigma\).

Assume that the resulting binary branch system is reconstructive. Then \(\Theta:(S,d_{\sup})\xrightarrow{\sim}S\times_\infty S\) is an isometry.
\end{lemma}

\begin{proof}
Since \(\Theta\) is bijective by Definition~\ref{def:binary-branch-system}, it suffices to prove that \(\Theta\) preserves distances.

Fix \(z,z'\in S\). Every \(\gamma\in\Sigma\) has a unique form \(0\beta\) or \(1\beta\). Therefore
\begin{align*}
d_{\sup}(z,z')
&=
\sup_{\gamma\in\Sigma}
\rho_\gamma(z_\gamma,z'_\gamma) \\
&=
\max\Bigl\{
\sup_{\beta\in\Sigma}
\rho_{0\beta}(z_{0\beta},z'_{0\beta}),
\sup_{\beta\in\Sigma}
\rho_{1\beta}(z_{1\beta},z'_{1\beta})
\Bigr\}.
\end{align*}

For every \(\beta\in\Sigma\), the maps \(\alpha_0^\beta\) and \(\alpha_1^\beta\) are isometries. Hence
\[
\rho_{0\beta}(z_{0\beta},z'_{0\beta})
=
\rho_\beta(L(z)_\beta,L(z')_\beta)
\]
and
\[
\rho_{1\beta}(z_{1\beta},z'_{1\beta})
=
\rho_\beta(R(z)_\beta,R(z')_\beta).
\]
Therefore
\[
d_{\sup}(z,z')
=
\max\{
d_{\sup}(L(z),L(z')),
d_{\sup}(R(z),R(z'))
\}.
\]
This quantity equals the distance from \(\Theta(z)\) to \(\Theta(z')\) in \(S\times_\infty S\). Therefore \(\Theta\) preserves distances.
\end{proof}

\begin{theorem}
\label{thm:binary-branch-representation}
A metric space \((X,d)\) is isometric to \(X\times_\infty X\) if and only if there exists a reconstructive binary branch system whose associated metric space \((S,d_{\sup})\) is isometric to \((X,d)\).
\end{theorem}

\begin{proof}
Assume first that \((X,d)\cong X\times_\infty X\). Choose an isometry \(\Phi:(X,d)\xrightarrow{\sim}X\times_\infty X\). Lemma~\ref{lem:self-square-to-branch-system} constructs a reconstructive binary branch system and an isometry \((X,d)\xrightarrow{\sim}(S,d_{\sup})\).

Conversely, suppose that there exists a reconstructive binary branch system whose associated metric space \((S,d_{\sup})\) is isometric to \((X,d)\). Let \(h:(X,d)\xrightarrow{\sim}(S,d_{\sup})\) be an isometry. Lemma~\ref{lem:branch-system-square-isometry} gives an isometry \(\Theta:(S,d_{\sup})\xrightarrow{\sim}S\times_\infty S\). Hence
\[
(h^{-1}\times h^{-1})
\circ
\Theta
\circ
h
:
(X,d)
\to
X\times_\infty X
\]
is an isometry.
\end{proof}

\begin{remark}
\label{rem:binary-branch-nonuniqueness}
Different choices of the self-square isometry \(\Phi\) may produce different binary branch systems. The representation in Theorem~\ref{thm:binary-branch-representation} is therefore generally nonunique.
\end{remark}

\subsubsection{Finite-Spectrum Self-Square Products and Prime Multiplicities}
\label{subsec:finite-spectrum-self-square-products}

We now construct self-square metric spaces from infinite multiplicities in finite-spectrum product decompositions. The finite-spectrum refinement theory from Subsection~\ref{subsec:finite-distance-refinement} gives uniqueness of finite \(\ell^\infty\)-prime decompositions under the hypotheses of Theorem~\ref{thm:finite-distance-prime-uniqueness}. We combine that rigidity with cardinal arithmetic.

If \((P,d_P)\) is a metric space and \(\kappa\) is a cardinal, then \(P^\kappa\) denotes the \(\ell^\infty\)-product of \(\kappa\) copies of \(P\).

\begin{definition}
\label{def:cardinal-multiplicity-uniqueness}
Let \((X,d)\) be a metric space. We say that \(X\) has cardinal multiplicity uniqueness if the following condition holds. Whenever \(X\cong\prod_{a\in A}P_a^{\kappa_a}\) and \(X\cong\prod_{b\in B}Q_b^{\lambda_b}\), where \(A\) and \(B\) are finite, where the metric spaces \(P_a\) and \(Q_b\) are pairwise nonisometric and \(\ell^\infty\)-prime, and where \(\kappa_a\neq0\) and \(\lambda_b\neq0\) for all \(a\in A\) and \(b\in B\), then, after reindexing, \(P_a\cong Q_a\) and \(\kappa_a=\lambda_a\) for every \(a\in A\).
\end{definition}

\begin{lemma}
\label{lem:finite-spectrum-product}
Let \(A\) be a finite set. For each \(a\in A\), let \((P_a,d_a)\) be a finite metric space, and let \(\kappa_a\) be a nonzero cardinal. Define \(X:=\prod_{a\in A}P_a^{\kappa_a}\). Then \(d(X\times X)=\bigcup_{a\in A}d_a(P_a\times P_a)\). In particular, \(d(X\times X)\) is finite.
\end{lemma}

\begin{proof}
Let \(x,y\in X\). Equation~\eqref{eq:k-product-metric-2} gives \(d(x,y)=\sup_{a\in A}\sup_{i\in\kappa_a}d_a(x_{a,i},y_{a,i})\). Since \(A\) is finite and each distance set \(d_a(P_a\times P_a)\) is finite, the preceding supremum is a maximum. Hence every distance in \(X\) belongs to \(\bigcup_{a\in A}d_a(P_a\times P_a)\).

Conversely, fix \(a\in A\), and let \(r\in d_a(P_a\times P_a)\). Choose \(u,v\in P_a\) such that \(d_a(u,v)=r\). Fix \(i\in\kappa_a\). Define \(x,y\in X\) by setting \(x_{a,i}=u\) and \(y_{a,i}=v\), and by setting every remaining coordinate equal in \(x\) and \(y\). Equation~\eqref{eq:k-product-metric-2} gives \(d(x,y)=r\). Hence \(r\in d(X\times X)\).

Therefore \(d(X\times X)=\bigcup_{a\in A}d_a(P_a\times P_a)\). Since \(A\) is finite and each distance set \(d_a(P_a\times P_a)\) is finite, the final assertion follows.
\end{proof}

\begin{lemma}
\label{lem:self-square-double-multiplicity}
Let \(A\) be a finite set. For each \(a\in A\), let \(P_a\) be a nonempty metric space, and let \(\kappa_a\) be a nonzero cardinal. Then
\[
\left(\prod_{a\in A}P_a^{\kappa_a}\right)
\times_\infty
\left(\prod_{a\in A}P_a^{\kappa_a}\right)
\cong
\prod_{a\in A}P_a^{2\kappa_a}.
\]
\end{lemma}

\begin{proof}
For each \(a\in A\), choose a set \(I_a\) with cardinality \(2\kappa_a\). Since \(|\kappa_a\sqcup\kappa_a|=2\kappa_a=|I_a|\), there exists a bijection \(I_a\to\kappa_a\sqcup\kappa_a\). Lemma~\ref{lem:product-reindexing} therefore gives
\[
P_a^{\kappa_a}\times_\infty P_a^{\kappa_a}
\cong
P_a^{I_a}
\cong
P_a^{2\kappa_a}.
\]
Taking products over \(a\in A\) yields the result.
\end{proof}

\begin{lemma}
\label{lem:cardinal-double}
Let \(\kappa\) be a nonzero cardinal. Then \(\kappa=2\kappa\) if and only if \(\kappa\) is infinite.
\end{lemma}

\begin{proof}
If \(\kappa\) is finite and nonzero, then \(2\kappa>\kappa\). Hence \(\kappa\neq2\kappa\).

Suppose that \(\kappa\) is infinite. Standard cardinal arithmetic gives \(2\kappa=\kappa\). Hence \(\kappa=2\kappa\).
\end{proof}

\begin{theorem}
\label{thm:infinite-multiplicity-self-square-construction}
Let \(A\) be a finite set. For each \(a\in A\), let \(P_a\) be a nonempty metric space, and let \(\kappa_a\) be a nonzero cardinal. Define \(X:=\prod_{a\in A}P_a^{\kappa_a}\). If every \(\kappa_a\) is infinite, then \(X\cong X\times_\infty X\).
\end{theorem}

\begin{proof}
Lemma~\ref{lem:cardinal-double} gives \(2\kappa_a=\kappa_a\) for every \(a\in A\). Therefore \(\prod_{a\in A}P_a^{2\kappa_a}\cong\prod_{a\in A}P_a^{\kappa_a}\). Lemma~\ref{lem:self-square-double-multiplicity} gives
\begin{align*}
X\times_\infty X
&\cong
\prod_{a\in A}P_a^{2\kappa_a}
\\
&\cong
\prod_{a\in A}P_a^{\kappa_a}
\\
&=
X.
\end{align*}
\end{proof}

\begin{theorem}
\label{thm:conditional-finite-spectrum-self-square-classification}
Let \((X,d)\) satisfy the hypotheses of Theorem~\ref{thm:finite-distance-prime-uniqueness}. Suppose that \((X,d)\cong\prod_{a\in A}P_a^{\kappa_a}\), where \(A\) is finite, where the metric spaces \(P_a\) are pairwise nonisometric finite \(\ell^\infty\)-prime metric spaces, and where every \(\kappa_a\) is a nonzero cardinal. Assume that \(X\) has cardinal multiplicity uniqueness. Then \(X\cong X\times_\infty X\) if and only if every \(\kappa_a\) is infinite.
\end{theorem}

\begin{proof}
Suppose first that every \(\kappa_a\) is infinite. Theorem~\ref{thm:infinite-multiplicity-self-square-construction} gives \(X\cong X\times_\infty X\).

Conversely, suppose that \(X\cong X\times_\infty X\). Lemma~\ref{lem:self-square-double-multiplicity} gives \(X\times_\infty X\cong\prod_{a\in A}P_a^{2\kappa_a}\). Hence \(X\cong\prod_{a\in A}P_a^{2\kappa_a}\).

Since \(X\) has cardinal multiplicity uniqueness, the two decompositions imply \(\kappa_a=2\kappa_a\) for every \(a\in A\). Lemma~\ref{lem:cardinal-double} therefore shows that every \(\kappa_a\) is infinite.
\end{proof}

\begin{corollary}
\label{cor:finite-spectrum-self-square-products}
Let \(P_1,\dots,P_s\) be finite metric spaces, and let \(\kappa_1,\dots,\kappa_s\) be infinite cardinals. Define \(X:=P_1^{\kappa_1}\times_\infty\cdots\times_\infty P_s^{\kappa_s}\). Then \(d(X\times X)\) is finite, and \(X\cong X\times_\infty X\).
\end{corollary}

\begin{proof}
Lemma~\ref{lem:finite-spectrum-product} shows that \(d(X\times X)\) is finite. Theorem~\ref{thm:infinite-multiplicity-self-square-construction} gives \(X\cong X\times_\infty X\).
\end{proof}

\begin{corollary}
\label{cor:binary-splitting-self-square}
Let \(X=\prod_{a\in A}P_a^{\kappa_a}\), where \(A\) is finite and every \(\kappa_a\) is infinite. Then every family of bijections \(\sigma_a:\kappa_a\to\{0,1\}\times\kappa_a\) determines an isometry \(X\to X\times_\infty X\). Moreover, an induction on \(n\) gives bijections \(\kappa_a\cong\{0,1\}^n\times\kappa_a\) for every \(n\geq1\).
\end{corollary}

\begin{proof}
For each \(a\in A\), the bijection \(\sigma_a:\kappa_a\to\{0,1\}\times\kappa_a\) induces a coordinate reindexing \(P_a^{\kappa_a}\to P_a^{\{0,1\}\times\kappa_a}\). Lemma~\ref{lem:product-reindexing} therefore gives an isometry \(X\to X\times_\infty X\).

The asserted bijections follow by induction on \(n\).
\end{proof}

Corollary~\ref{cor:binary-splitting-self-square} provides a concrete source of the binary branch structures from Subsection~\ref{subsec:binary-branch-representations}. The coordinate splittings determined by the bijections \(\kappa_a\cong\{0,1\}\times\kappa_a\) generate the branch decompositions that appear there.

\section*{Statements and Declarations}

\begin{itemize}

\item \textbf{Competing Interests} The authors declare that they have no competing interests.

\item \textbf{Funding} This research received no external funding.

\item \textbf{Authors' Contributions}

C.F. developed the theoretical framework, proved the main results, and wrote the manuscript. M.E.A. advised the project and provided feedback on the framework and manuscript.

\end{itemize}

\bibliography{sn-bibliography}

\end{document}